\documentclass[11pt]{article}
\usepackage[margin=1in]{geometry}
\usepackage{amsmath,amssymb,amsthm,mathtools,booktabs,array,enumitem,hyperref,microtype,amscd}
\usepackage[T1]{fontenc}
\usepackage{lmodern}
\hypersetup{
 colorlinks=true,
 linkcolor=blue,
 citecolor=blue,
 urlcolor=blue,
 pdftitle={Compact Complex Threefolds Fibred by Two-Dimensional Complex Tori},
 pdfauthor={Xiufan Yang},
 pdfsubject={Compact complex threefolds, complex-torus fibrations, algebraic reduction, toroidal degenerations, and deformations},
 pdfkeywords={compact complex threefold, complex torus fibration, algebraic reduction, non-K\"ahler geometry, toroidal degeneration, logarithmic transform, deformation theory}
}

\newtheorem{theorem}{Theorem}[section]
\newtheorem{proposition}[theorem]{Proposition}
\newtheorem{lemma}[theorem]{Lemma}
\newtheorem{corollary}[theorem]{Corollary}
\newtheorem{remark}[theorem]{Remark}

\newcommand{\Z}{\mathbb Z}
\newcommand{\Q}{\mathbb Q}
\newcommand{\R}{\mathbb R}
\newcommand{\C}{\mathbb C}
\newcommand{\Hh}{\mathbb H}
\newcommand{\Pone}{\mathbb P^1}
\newcommand{\rk}{\operatorname{rk}}
\newcommand{\coker}{\operatorname{coker}}
\newcommand{\im}{\operatorname{im}}
\newcommand{\Aut}{\operatorname{Aut}}
\newcommand{\diag}{\operatorname{diag}}
\newcommand{\NS}{\operatorname{NS}}

\title{Compact Complex Threefolds Fibred by\\
Two-Dimensional Complex Tori}
\author{Xiufan Yang\\
\small School of Science, Nanjing University of Posts and Telecommunications\\
\small \texttt{b25100020@njupt.edu.cn}}
\date{}

\begin{document}
\maketitle

\begin{abstract}
We construct two families of compact complex threefolds fibred over
$\mathbb P^1$ by two-dimensional complex tori.  For each positive integer
$r$ and each parameter $c$ in a fixed half-plane, the families
$Y_{r,c}$ and $Z_{r,c}$ have two multiple fibres of multiplicities
$(4,6)$ and $(3,6)$, respectively, and one reduced normal-crossings
fibre whose normalization consists of $r$ del Pezzo surfaces of degree
six.  Their fundamental groups are $\mathbb Z/2$ and $\mathbb Z/3$,
respectively, and both have Euler characteristic $2r$.  We compute their
integral homology; the index-one members are rational homology
six-spheres.

Every member has algebraic dimension one, lies outside Fujiki's class
$\mathcal C$, and has connected automorphism group $\mathbb C^*$.
For fixed type and index, two members are biholomorphic exactly when
their parameters differ by an element of $(6/r)\mathbb Z$.  Each
underlying smooth six-manifold therefore supports uncountably many
pairwise non-biholomorphic complex structures.

The arithmetic input classifies finite monodromy pairs and their
integral lifts to a parabolic Jacobi group acting on a rank-four lattice.
A quadratic form controls integral lifting and, on the rank-two
square-zero locus, computes the monodromy saturation index.
We prove that the relevant invariant tensors persist under finite-index
restriction, which extends the construction to arbitrary compact base
curves.  We also construct translation twists after base change,
establish a local periodic toroidal quotient theorem in arbitrary
dimension, and determine the additive Bockstein spectral sequences of
the original threefolds.
\end{abstract}

\tableofcontents

\section{Introduction}

\subsection*{Background and main results}

Fibrations over curves provide a natural approach to compact complex
threefolds of algebraic dimension one.  In the non-K\"ahler setting,
the general fibre can be a non-algebraic complex torus, while singular
fibres can be nonnormal.  Controlling the local completions is therefore
essential to both the complex geometry and the topology of the total
space.

Campana--Demailly--Peternell study algebraic reduction under restrictions
on low-degree cohomology \cite{CDP98}.  Their corrected
Theorem~2.2 in \cite{CDP20} assumes
$H^1(X;\mathbb Z)=H^2(X;\mathbb Z)=0$, algebraic dimension one,
and a holomorphic algebraic reduction onto a curve; it concludes that
$e(X)\leq0$.  The index-one examples below instead have
$H^2(Y_1;\mathbb Z)\simeq\mathbb Z/2$ and
$H^2(Z_1;\mathbb Z)\simeq\mathbb Z/3$.
Honda--Viaclovsky's Theorem~1.1 \cite{HondaViaclovsky} excludes
surjective holomorphic maps to two-dimensional complex spaces when
$b_1=b_2=0$ and the Euler characteristic is nonzero; its target has
dimension two, rather than one.

The closest construction is Alp\"oge's Type A family over the orbifold
curve of type $(3,4,\infty)$ \cite[Sections~2--6]{AlpogeS6}.  It combines
two logarithmic transforms with a toroidal completion at a unipotent
point.  In the same rank-four parabolic setting, we classify the finite
monodromy pairs and construct the $(4,6)$ and $(3,6)$ families for
every positive cusp index.  Their topology is computed below directly
from the local models and their attaching maps.

\paragraph{The index-one consequence.}
Albanese--Milivojevi\'c record Sullivan's conjecture that a compact
complex $n$-fold, $n\geq3$, has total Betti number at least four
\cite[Introduction]{AlbaneseMilivojevic}.  The index-one members below
have total rational Betti number two, giving counterexamples in
complex dimension three.  Their fundamental groups, $\mathbb Z/2$
and $\mathbb Z/3$, distinguish them from $S^6$.

\subsection*{The arithmetic statement}

Let
\[
 U_h=\begin{pmatrix}1&h\\0&1\end{pmatrix},\qquad h\geq1.
\]
Reversal of the parabolic orientation means replacing the oriented cusp
generator by its inverse, so that $U_h$ is replaced by $U_{-h}$ before
renormalizing the width to be positive.  The first theorem identifies the possible finite monodromies and the
integral obstruction to their affine lifts.

\begin{theorem}[Finite monodromy and Jacobi lifting]\label{thm:intro-arithmetic}
Let $B_1,B_2\in SL_2(\mathbb Z)$ be noncentral finite-order matrices and
suppose that $B_1B_2$ is conjugate to $U_h$.  Up to simultaneous
conjugacy, interchange of the two finite points, and reversal of the
parabolic orientation, the only possibilities are
\[
 (3,4;h=1),\qquad (4,6;h=1),\qquad (3,6;h=2).
\]
In particular, $h\leq2$.

Fix $s\in\mathbb Z\setminus\{0\}$.  For the associated rank-four
parabolic Jacobi representations of scale $s$, the classes for which
the logarithm of the unipotent monodromy has square
zero form a rank-one lattice.  If $k$ is a primitive coordinate on this
lattice and
\[
 \nu_\rho=\left|\det(I-B_1)\det(I-B_2)\right|,
\]
then an integral lift of scale $s$ exists if and only if
$s k^2/\nu_\rho\in\mathbb Z$.  For a liftable class with $k\neq0$, the
nilpotent monodromy has rank two and its image has finite index
\[
 r=\left|\frac{s k^2}{\nu_\rho}\right|
\]
in its saturated image lattice.  For $k=0$ the nilpotent rank is one;
this case is described in Proposition~\ref{prop:zero-affine-index}.
The parameter $r$ above refers to the rank-two case.
\end{theorem}

Thus the quadratic form controls both integral lifting and the cusp
index.  In the periodic $A_2$ compactifications constructed here, this
index equals the number of components of the normalized toroidal fibre.

\subsection*{The compact threefolds}

Write Type B and Type C for the $(4,6)$- and $(3,6)$-families.
Their parameter domains are fixed lower half-planes
$\mathcal U_\kappa=\{\operatorname{Im}c<C_\kappa\}$ as in
\eqref{eq:admissible-domain}; admissible always means $c\in\mathcal U_\kappa$.
A normal-crossings fibre here need not be simple: locally its equation
is $t=z_0\cdots z_k$, but different branches may belong to the same
global component.  The full integral homology is given in
Theorem~\ref{thm:complete-integral-homology}.

\begin{theorem}[The two compactification families]
For every $r\geq1$ and every admissible affine parameter $c$ there are
compact connected complex threefolds
\[
 f_{r,c}^B:Y_{r,c}\longrightarrow\mathbb P^1,
 \qquad
 f_{r,c}^C:Z_{r,c}\longrightarrow\mathbb P^1
\]
with connected fibres and the following properties.
\begin{enumerate}[label=\textup{(\roman*)}]
\item Away from three points, both maps are proper holomorphic
submersions with two-dimensional complex-torus fibres.
\item The two multiple fibres have multiplicities $(4,6)$ for $Y_{r,c}$ and
$(3,6)$ for $Z_{r,c}$; their reductions are smooth bielliptic surfaces.
The third special fibre is a reduced toroidal normal-crossings fibre whose
normalization is a disjoint union of $r$ del Pezzo surfaces of degree six.
\item
\[
 \pi_1(Y_{r,c})\simeq\mathbb Z/2,
 \qquad
 \pi_1(Z_{r,c})\simeq\mathbb Z/3,
 \qquad
 e(Y_{r,c})=e(Z_{r,c})=2r.
\]
\end{enumerate}
If $r\neq r'$, then $Y_{r,c}$ is not homeomorphic to $Y_{r',c'}$, and
the analogous statement holds for the $Z$-family.  No $Y_{r,c}$ is
homeomorphic to any $Z_{s,c'}$.
\end{theorem}

For a fixed admissible parameter we suppress $c$ and write
$Y_r=Y_{r,c}$ and $Z_r=Z_{r,c}$.

\begin{corollary}[Cohomology of the compactifications]
The rational cohomology of both families is
\[
 H^q(Y_r;\mathbb Q)\simeq H^q(Z_r;\mathbb Q)\simeq
 \begin{cases}
 \mathbb Q,&q=0,6,\\
 \mathbb Q^{r-1},&q=2,4,\\
 0,&q=1,3,5.
 \end{cases}
\]
Moreover,
\[
 \operatorname{Tor}H_2(Y_r)\simeq\operatorname{Tor}H_3(Y_r)
 \simeq\mathbb Z/r\oplus\mathbb Z/2,
 \qquad
 \operatorname{Tor}H_2(Z_r)\simeq\operatorname{Tor}H_3(Z_r)
 \simeq\mathbb Z/(3r),
\]
and the free rank of $H_2$ and $H_4$ is $r-1$.  In particular,
$Y_1$ and $Z_1$ are non-K\"ahler rational homology six-spheres.
\end{corollary}

The integer controlling the fundamental group has the familiar form of
a Seifert Euler number \cite{Orlik}.  If the two logarithmic transforms contribute
integers $\ell_1,\ell_2$ and the toroidal gluing contributes $\ell_0$,
then
\[
 e_{\mathrm{orb}}=\ell_0-\frac{\ell_1}{m_1}-\frac{\ell_2}{m_2},
 \qquad
 p=m_1m_2e_{\mathrm{orb}}.
\]
The integer $p$ agrees up to sign with the determinant of the relation
matrix for the surviving coinvariant circle: in the row convention of
Theorem~\ref{thm:van-kampen-reduction}, $\det R=-p$.  The canonical
choices give $|p|=2$ in the $(4,6)$ case and $|p|=3$ in the $(3,6)$ case.

\subsection*{Algebraic reduction and deformations}

The affine period parameter changes the complex structure of the compact
total space.  Fujiki's class $\mathcal C$ consists of compact complex
manifolds bimeromorphic to compact K\"ahler manifolds.

\begin{theorem}[Algebraic reduction and deformation]
For each $r\geq1$ and for either family, the affine period parameter gives
a proper holomorphic one-parameter deformation on a fixed oriented
smooth six-manifold.  At every admissible parameter the threefold is
outside Fujiki's class $\mathcal C$, has algebraic dimension one, and
its map to $\mathbb P^1$ is the algebraic reduction.  Both the Kodaira--Spencer class of the
family of pairs and the absolute Kodaira--Spencer class are nonzero at
every parameter.  On the entire admissible half-plane, the unmarked
isomorphism relation is translation by $(6/r)\mathbb Z$, and the
connected automorphism group of each member is $\C^*$.
Consequently every underlying smooth six-manifold in either family
supports uncountably many pairwise non-biholomorphic complex structures.
\end{theorem}

The proof of algebraic dimension one uses the rank-one monodromy-invariant
part of the second cohomology of a smooth fibre.  Its primitive generator
has Appell--Humbert signature $(1,1)$ at every admissible parameter.
Every line bundle on the total space restricts to a multiple of this
class; such a class cannot support a nonzero effective divisor on a
torus.  Hence no divisor on the total space is horizontal, and all
meromorphic functions come from the base.  The same invariant-class
argument excludes a K\"ahler class on the total space.
Appendix~\ref{app:fibre-NS} gives an explicit Hodge-locus equation and
determines the N\'eron--Severi group at a suitably chosen fixed smooth
fibre for very general parameters.

\subsection*{Extensions}
Section~\ref{sec:extensions} proves that the invariant alternating line
and the integral centralizer of the Type B/C monodromies are unchanged
under every finite-index restriction.  Finite maps from arbitrary
compact curves, unramified over the three special values, therefore
give compact threefolds of algebraic dimension one outside Fujiki's
class.  Their Euler number is $2rd$, where $d$ is the degree of the
base change, and the affine parameter still gives uncountably many
complex structures on a fixed smooth manifold.
The same base change has a translation-gluing space of dimension
$2g+d-1$, with its integral period quotient kept explicit.
We also give a local periodic quotient theorem for complex tori of
any dimension and determine the mod-prime cohomology and additive
Bockstein spectral sequences of the original threefolds.

\subsection*{Organization}
Sections~\ref{sec:cores}--\ref{sec:global} develop the arithmetic,
period maps, local completions, and global gluing.  Section~\ref{sec:affine-moduli}
proves algebraic reduction, the parameter classification, and the
statements on deformations and automorphisms.  Section~\ref{sec:topology}
computes the fundamental group and integral homology, and
Section~\ref{sec:extensions} treats base change, translation twists,
and higher-dimensional local quotients.  The additive Bockstein
spectral sequences are computed at the end of that section.
Appendices~\ref{app:monodromy-data} and \ref{app:integral-mv-ledger}
record the lattice calculations and geometric attaching data.
Appendices~\ref{app:fibre-NS} and \ref{app:nearby-cycles} contain the
independent fibrewise N\'eron--Severi and nearby-cycle calculations.

\section{Finite monodromy and integral Jacobi lifts}\label{sec:cores}

This section contains the arithmetic input for the two constructions.  We
first classify the finite monodromy pairs, then formulate the lifting
problem as a central-extension obstruction, and finally specialize the
obstruction to the three possible monodromy types.  The explicit reductions are in Appendix~\ref{app:lift-calculations}.

\subsection{Finite-order monodromy pairs}

For $h\geq1$ put
\[
 U_h=\begin{pmatrix}1&h\\0&1\end{pmatrix}.
\]
The exact orders of noncentral finite-order elements of $SL_2(\Z)$ are
$3,4,6$, with traces $-1,0,1$, respectively.

\begin{theorem}[Classification by parabolic width]\label{thm:core-classification}
Let $B_1,B_2\in SL_2(\Z)$ be noncentral finite-order matrices, and suppose
that $B_1B_2$ is conjugate to $U_h$ for some $h\geq1$.  Up to simultaneous
conjugacy, reversal of the parabolic orientation, and interchange of the
two elliptic factors, precisely the following cases occur:
\begin{align*}
\textup{Type A }(3,4;h=1):\quad
&B_1=\begin{pmatrix}-1&1\\-1&0\end{pmatrix},
&&B_2=\begin{pmatrix}0&-1\\1&0\end{pmatrix};\\[1ex]
\textup{Type B }(4,6;h=1):\quad
&B_1=\begin{pmatrix}0&1\\-1&0\end{pmatrix},
&&B_2=\begin{pmatrix}0&-1\\1&1\end{pmatrix};\\[1ex]
\textup{Type C }(3,6;h=2):\quad
&B_1=\begin{pmatrix}-1&1\\-1&0\end{pmatrix},
&&B_2=\begin{pmatrix}0&-1\\1&1\end{pmatrix}.
\end{align*}
The products in the first two cases are $U_1$, and the product in the third
is $U_2$.  In particular, no width $h>2$ occurs.
\end{theorem}

\begin{proof}
Conjugate the product to $U_h$ and write
\[
 t_i=\operatorname{tr}(B_i)\in\{-1,0,1\},
 \qquad
 B_2=\begin{pmatrix}a&b\\c&t_2-a\end{pmatrix}.
\]
Since $B_1=U_hB_2^{-1}$,
\begin{equation}\label{eq:width-trace}
 t_1=t_2-hc.
\end{equation}
The entry $c$ is nonzero, since otherwise $B_2$ would be an integral upper
triangular finite-order matrix and hence central.  Thus
$h\leq |t_2-t_1|\leq2$.

For $h=1$, a trace gap of two would give, after the stated symmetries,
$(t_1,t_2,c)=(-1,1,2)$.  The determinant equation
$a(1-a)-2b=1$ is impossible modulo two.  Hence the trace gap is one.  We may
take $c=1$, and the determinant equation gives
$b=a(t_2-a)-1$.  Conjugation by $U_k$, which centralizes $U_1$, replaces
$a$ by $a+k$; choosing $k=-a$ gives Types A and B.

For $h=2$, equation~\eqref{eq:width-trace} forces
$(t_1,t_2,c)=(-1,1,1)$ up to the stated symmetries.  The same determinant
calculation and conjugation by the centralizer of $U_2$ give Type C.  The
three products are then checked directly.
\end{proof}

In particular, a primitive unipotent product occurs only in Types A and B.

\subsection{The central-extension obstruction}

Let $(W,\omega)$ be a free abelian group of finite rank with an integral
alternating form.  Write
\[
 \mathcal A(W,\omega)=W\rtimes\operatorname{Sp}(W,\omega)
\]
for the affine symplectic group and define the integral Jacobi group by
\[
 \mathcal J(W,\omega)
 =\operatorname{Sp}(W,\omega)\times W\times\Z,
\]
with multiplication
\begin{equation}\label{eq:abstract-jacobi-law}
 (B,u,z)(B',u',z')
 =\bigl(BB',u+Bu',z+z'+\omega(u,Bu')\bigr).
\end{equation}
Thus
\begin{equation}
 1\longrightarrow\Z\longrightarrow
 \mathcal J(W,\omega)\longrightarrow
 \mathcal A(W,\omega)\longrightarrow1
\end{equation}
is a central extension.

Let $\rho:\Gamma\to\operatorname{Sp}(W,\omega)$ be a representation.  A
one-cocycle $d\in Z^1(\Gamma,W_\rho)$ defines the affine representation
$g\mapsto(\rho(g),d(g))$.

\begin{theorem}[Central obstruction for affine symplectic monodromy]
\label{thm:general-jacobi-obstruction}
For $d\in Z^1(\Gamma,W_\rho)$ put
\begin{equation}
 \Omega_d(g,h)=\omega\bigl(d(g),\rho(g)d(h)\bigr).
\end{equation}
Then $\Omega_d$ is a normalized integral two-cocycle and its cohomology
class depends only on $[d]\in H^1(\Gamma,W_\rho)$.  The affine
representation lifts to $\mathcal J(W,\omega)$ if and only if
\[
 \mathfrak o_\omega([d]):=[\Omega_d]=0
 \quad\text{in }H^2(\Gamma,\Z).
\]
When the obstruction vanishes, the set of central lifts is a torsor under
$H^1(\Gamma,\Z)$.  Moreover $\mathfrak o_\omega$ is quadratic, with
polarization represented by
\begin{align}
 \Omega_{d,e}(g,h)
 &=\omega\bigl(d(g),\rho(g)e(h)\bigr)\notag\\
 &\quad+\omega\bigl(e(g),\rho(g)d(h)\bigr).
\end{align}
\end{theorem}

\begin{proof}
The cocycle identity follows from the cocycle identity for $d$ and the
$\rho$-invariance of $\omega$.  Replacing $d$ by a cohomologous cocycle
conjugates the affine representation by a translation, so the pulled-back
central extension has the same class.  A lift has the form
$(\rho(g),d(g),z(g))$, and~\eqref{eq:abstract-jacobi-law} says precisely
that $\delta z=-\Omega_d$.  The final assertion follows by expanding
$\Omega_{d+e}$.  For the central-extension formalism, see \cite[Chapter~IV]{Brown}.
\end{proof}

\subsection{Rank-four lifts and the square-zero condition}\label{sec:lifts}

Let $W=\Z^2$ and
\[
 J_s=s\begin{pmatrix}0&1\\-1&0\end{pmatrix},
 \qquad s\in\Z\setminus\{0\}.
\]
On $V=\Z\gamma\oplus W\oplus\Z\delta$ consider
\[
 Q_s=
 \begin{pmatrix}
 0&0&1\\
 0&J_s&0\\
 -1&0&0
 \end{pmatrix}.
\]
The $Q_s$-isometries preserving the indicated isotropic flag and acting
trivially on its two rank-one quotients are
\[
 T(B,d,z)=
 \begin{pmatrix}
 1&d^{\mathsf T}J_sB&z\\
 0&B&d\\
 0&0&1
 \end{pmatrix},
 \qquad B\in SL_2(\Z),\ d\in\Z^2,\ z\in\Z,
\]
with multiplication
\begin{equation}\label{eq:jacobi-law}
 T(B,d,z)T(B',d',z')
 =T\bigl(BB',d+Bd',z+z'+d^{\mathsf T}J_sBd'\bigr).
\end{equation}
This is the rank-four matrix realization of the preceding Jacobi group.

For $\Gamma_{m_1,m_2}=C_{m_1}*C_{m_2}$, let $B_i=\rho(g_i)$ be the two
elliptic monodromies.  Since $B_i$ has no fixed vector, a cocycle is
specified by $d_i=d(g_i)\in W$, modulo
\begin{equation}
 (d_1,d_2)\longmapsto
 \bigl(d_1+(I-B_1)u,d_2+(I-B_2)u\bigr).
\end{equation}
The standard identification
\begin{equation}\label{eq:H2-free-product}
 H^2(\Gamma_{m_1,m_2},\Z)
 \cong\Z/m_1\oplus\Z/m_2
\end{equation}
records the restrictions to the two cyclic factors.  Moreover
$H^1(\Gamma_{m_1,m_2},\Z)=0$, so a central lift, when it exists, is
unique after the elliptic central coordinates have been fixed.  For
$B^m=I$ set
\begin{equation}
 \kappa_m(B,u)
 =\sum_{0\leq r<t<m}\omega(B^ru,B^tu).
\end{equation}

\begin{lemma}[Elliptic relation defects]\label{lem:local-obstruction}
Under~\eqref{eq:H2-free-product},
\[
 \mathfrak o_\omega([d_1,d_2])
 =\bigl(\kappa_{m_1}(B_1,d_1)\bmod m_1,
        \kappa_{m_2}(B_2,d_2)\bmod m_2\bigr).
\]
Hence a central lift exists exactly when
$m_i\mid\kappa_{m_i}(B_i,d_i)$ for $i=1,2$, and then
\[
 z_i=-\frac{\kappa_{m_i}(B_i,d_i)}{m_i}.
\]
\end{lemma}

\begin{proof}
Repeated use of~\eqref{eq:abstract-jacobi-law} gives
\[
 (B,u,0)^m=(I,N_Bu,\kappa_m(B,u)),\qquad
 N_B=I+B+\cdots+B^{m-1}.
\]
For each finite monodromy used here, $B$ has no eigenvalue $1$, so
$N_B=0$.  Adding central coordinate $z$ changes
the final coordinate by $mz$.  The wedge decomposition of the classifying
space of a free product gives~\eqref{eq:H2-free-product}; see
\cite[Chapter VII]{Brown}.
\end{proof}

Now suppose $B_1B_2=U_h$ and put
$L_\infty=\ker(U_h-I)$.  The cusp residue of a cocycle is
\begin{equation}
 \operatorname{res}_\infty([d])
 =d(g_1g_2)\bmod L_\infty\in W/L_\infty.
\end{equation}
It is well defined on cohomology because a coboundary changes
$d(g_1g_2)$ by an element of $(I-U_h)W\subset L_\infty$.  Let
$H^1_{\square}(\rho,W)$ denote its kernel.

\begin{lemma}[Square-zero monodromy and affine cohomology]\label{lem:square-zero-cohomology}
A Jacobi lift has cusp monodromy $P$ with $(P-I)^2=0$ if and only if
its affine class belongs to $H^1_{\square}(\rho,W)$.  For each of Types A,
B, and C,
\[
 H^1(\Gamma_{m_1,m_2},W_\rho)\cong\Z^2,
 \qquad
 H^1_{\square}(\rho,W)\cong\Z.
\]
If $d_1=(x,y)^{\mathsf T}$ and $d_2=(a,b)^{\mathsf T}$, the square-zero
condition is $a=y$, and primitive coordinates on the three rank-one
lattices are
\[
 k_A=2x-y+3b,
 \qquad
 k_B=x+y+2b,
 \qquad
 k_C=x+y+3b.
\]
\end{lemma}

\begin{proof}
In a basis with $U_h=\begin{psmallmatrix}1&h\\0&1\end{psmallmatrix}$,
write $d(g_1g_2)=(p,q)$.  A block multiplication gives
\[
 \bigl(T(U_h,(p,q),z)-I\bigr)^2
 =\begin{pmatrix}
 0&0&-hsq&-hsq^2\\
 0&0&0&hq\\
 0&0&0&0\\
 0&0&0&0
 \end{pmatrix},
\]
so the square vanishes exactly when the residue is zero.  For the three
monodromy pairs, the coboundary map has Smith form
$\operatorname{diag}(1,1)$, while the residue is the primitive coordinate
$a-y$.  On its kernel, the conjugacy moves and primitive quotient
coordinates are
\[
\begin{array}{c|c|c}
A&(2,1,-1),\ (-1,1,1)&2x-y+3b\\
B&(1,1,-1),\ (-1,1,0)&x+y+2b\\
C&(2,1,-1),\ (-1,1,0)&x+y+3b.
\end{array}
\]
In each row the two move vectors generate the saturated kernel of the
listed functional.
\end{proof}

Define
\begin{equation}
 \nu_\rho=\left|\det(I-B_1)\det(I-B_2)\right|
\end{equation}
and, on the rank-one square-zero lattice,
\begin{equation}
 Q_\rho(k)=\frac{k^2}{\nu_\rho}.
\end{equation}
The data needed below are summarized by
\[
\begin{array}{c|c|c|c|c}
\text{type}&(m_1,m_2;h)&(z_1,z_2)&k_\rho&\nu_\rho\\ \hline
A&(3,4;1)&\left(\frac{s}{3}(x^2-xy+y^2),-\frac{s}{2}(a^2+b^2)\right)&2x-y+3b&6\\[1ex]
B&(4,6;1)&\left(\frac{s}{2}(x^2+y^2),-s(a^2+ab+b^2)\right)&x+y+2b&2\\[1ex]
C&(3,6;2)&\left(\frac{s}{3}(x^2-xy+y^2),-s(a^2+ab+b^2)\right)&x+y+3b&3.
\end{array}
\]

\begin{theorem}[Integral Jacobi lifts and the cusp index]
\label{thm:obstruction-discriminant}

For each row of the table, the two elliptic relations hold exactly when the
displayed central coordinates are integral.  On the square-zero locus
$a=y$, the coordinate $k_\rho$, up to sign, is the complete conjugacy
invariant under the integral Heisenberg subgroup and the central Levi
involution.  The central obstruction is
\[
 \begin{array}{c|c}
 A&(-sk^2\bmod3,\;2sk^2\bmod4)\\
 B&(-2sk^2\bmod4,\;0\bmod6)\\
 C&(-sk^2\bmod3,\;0\bmod6),
 \end{array}
\]
so a lift exists if and only if
\[
 sQ_\rho(k)\in\Z,
 \qquad\text{equivalently}\qquad
 \nu_\rho\mid sk^2.
\]
For $k\neq0$, write $P=T_1T_2$, $N=P-I=\log P$, and
$\Lambda_N=V\cap\operatorname{im}(N\otimes\Q)=\ker(N:V\to V)$.
Then $\Lambda_N$ has rank two and the monodromy image has index
\begin{equation}\label{eq:index-as-Q}
 d_{\mathrm{cusp}}=[\Lambda_N:N(V)]
 =\left|sQ_\rho(k)\right|
 =\frac{|s|k^2}{\nu_\rho}.
\end{equation}
For the periodic $A_2$ compactifications of Section~\ref{sec:cusp},
this saturation index equals the number of components of the normalized
central fibre; see Remark~\ref{rem:model-dependent-components}.
\end{theorem}

\begin{proof}
The elliptic relations and the obstruction table follow from
Lemma~\ref{lem:local-obstruction}.  On the square-zero locus, the quotient
to the saturated image lattice is represented by a $2\times2$ integral matrix.
Appendix~\ref{app:lift-calculations} lists these three matrices and computes
their determinants as
\[
 \frac{s}{6}k_A^2,
 \qquad
 \frac{s}{2}k_B^2,
 \qquad
 \frac{s}{3}k_C^2.
\]
Lemma~\ref{lem:square-zero-cohomology} computes the Heisenberg move
lattices and shows that their quotients are generated by $k_\rho$.  This proves completeness of the invariant and
formula~\eqref{eq:index-as-Q}.
\end{proof}

\begin{proposition}[The zero affine class]\label{prop:zero-affine-index}
For any of the three rank-four Jacobi types, every lift on the square-zero
locus with $k_\rho=0$ is integrally Heisenberg-conjugate to the lift with
$d_1=d_2=0$ and $z_1=z_2=0$.  Its cusp logarithm has rank one, and
\begin{equation}\label{eq:rank-one-saturation-index}
 [\operatorname{sat}N(V):N(V)]=h_\rho.
\end{equation}
Here $\operatorname{sat}N(V)$ has rank one whereas $\ker N$ has rank
three.  Thus \eqref{eq:rank-one-saturation-index} is not the rank-two
cusp parameter $r$ used for the constructed compactifications.
\end{proposition}

\begin{proof}
The primitive quotient calculation in
Lemma~\ref{lem:square-zero-cohomology} identifies $k_\rho=0$ with the
zero cohomology class, not merely with a rationally trivial class.
Consequently a single integral Heisenberg conjugacy makes both $d_i$
zero.  In the resulting lifts the finite relations say $m_i z_i=0$;
since $z_i\in\Z$, both central coordinates are zero.  The product is
$\diag(1,U_{h_\rho},1)$.  In the basis $(\gamma,u,w,\delta)$ its
logarithm sends $w$ to $h_\rho u$ and kills the other three vectors.
Its image is $h_\rho\Z u$, its saturation is $\Z u$, and its kernel
is $\Z\gamma\oplus\Z u\oplus\Z\delta$.  Integral conjugacy preserves
all these assertions.
\end{proof}

\begin{corollary}[Lifts of cusp index one]

A cusp-index-one lift exists precisely when $|s|=\nu_\rho$ and $|k_\rho|=1$.
After orientation and conjugacy, representatives are
\[
\begin{array}{c|c|c|c}
\text{type}&s& (d_1,d_2)&(z_1,z_2)\\ \hline
A&6&((1,1),(1,0))&(2,-3)\\
B&2&((1,0),(0,0))&(1,0)\\
C&3&((1,0),(0,0))&(1,0).
\end{array}
\]
The standard Type B matrices used in the period construction are displayed
in~\eqref{eq:typeB-standard-matrices}.
\end{corollary}

\begin{corollary}[Monodromy cokernels]
Let $k\neq0$ be liftable, put $d=|s|k^2/\nu_\rho$, and let
$\overline N:V/\Lambda_N\to\Lambda_N$ be the map induced by $N$.
For Types A and B,
\[
 \coker\overline N\cong\Z/d.
\]
For Type C, with $g=\gcd(2,k)$,
\[
 \coker\overline N\cong\Z/g\oplus\Z/(d/g).
\]
In particular, the Type C group is cyclic exactly when $k$ is odd.
\end{corollary}

\begin{proof}
For Types A and B, the quotient-to-kernel matrices in
Appendix~\ref{app:lift-calculations} contain an entry $1$.  For Type C the
first Smith invariant is $\gcd(2,k)$; the second is the determinant divided
by the first.
\end{proof}

\section{Period families for the \texorpdfstring{$(4,6)$}{(4,6)} and
\texorpdfstring{$(3,6)$}{(3,6)} cases}
\label{sec:period}

The two remaining monodromy types admit parallel period constructions.  We
write $\kappa=B,C$ for the $(4,6)$- and $(3,6)$-cases, respectively.  The
basic numerical data are
\[
\begin{array}{c|c|c|c|c}
 \kappa&(m_{\kappa,1},m_{\kappa,2})&h_\kappa&s_\kappa&J_\kappa(t)\\ \hline
 B&(4,6)&1&2&1728(1-t)\\
 C&(3,6)&2&3&6912t(1-t).
\end{array}
\]
Here $h_\kappa$ is the parabolic width and $s_\kappa$ is the coefficient
appearing in the first period column.  In both cases the base is
$B_\kappa=\Pone_t$, with elliptic points $p_{\kappa,1}=0$ and
$p_{\kappa,2}=1$ and cusp $p_{\kappa,0}=\infty$.  Let
\[
 \pi_\kappa:\Hh_z\longrightarrow B_\kappa\setminus\{p_{\kappa,0}\}
\]
be the orbifold universal cover, with deck generators
$g_{\kappa,1},g_{\kappa,2}$ of orders $m_{\kappa,1},m_{\kappa,2}$ and
$g_{\kappa,0}=(g_{\kappa,1}g_{\kappa,2})^{-1}$.

We write $\zeta_m=e^{2\pi i/m}$.  Deck transformations act on the left
and are composed as maps.  For based paths, in contrast, a product
$a_1a_2$ means that $a_1$ is traversed first and $a_2$ second.  Choose
clockwise meridians and basing arcs so that the lift of $a_j$ from a
fixed reference point $z_*$ ends at $g_{\kappa,j}z_*$ and
$a_0=(a_1a_2)^{-1}$.  The lift of $a_1a_2$ ends at
$g_{\kappa,1}g_{\kappa,2}z_*$: its second segment is the
$g_{\kappa,1}$-translate of the reference lift of $a_2$.
The inverse appearing in forward fibre transport is made explicit below.

\subsection{The two modular coordinates}

\begin{proposition}[Modular coordinates]
There are holomorphic maps $\tau_B=\tau$ and $\tau_C$ from the respective
orbifold universal covers to $\Hh$ with
\begin{equation}
 j(\tau)=1728(1-t\circ\pi_B)
\end{equation}
and
\begin{equation}
 j(\tau_C)=6912(t\circ\pi_C)(1-t\circ\pi_C).
\end{equation}
After choosing the standard lifts of the elliptic generators,
\begin{equation}\label{eq:tau-laws}
 \tau\circ g_{B,1}=-\frac1\tau,
 \qquad
 \tau\circ g_{B,2}=-\frac1{\tau+1},
\end{equation}
whereas
\begin{equation}\label{eq:typeC-tau-laws}
 \tau_C\circ g_{C,1}=\frac{\tau_C-1}{\tau_C},
 \qquad
 \tau_C\circ g_{C,2}=-\frac1{\tau_C+1}.
\end{equation}
Consequently
\[
 \tau_B\circ(g_{B,1}g_{B,2})=\tau_B+1,
 \qquad
 \tau_C\circ(g_{C,1}g_{C,2})=\tau_C+2.
\]
\end{proposition}

\begin{proof}
For Type B, the coarse map $t\mapsto1728(1-t)$ sends the orbifold points
of orders $4$ and $6$ to the modular points of orders $2$ and $3$ through
the quotient maps $\Z/4\to\Z/2$ and $\Z/6\to\Z/3$.  Its lift to universal
covers gives \eqref{eq:tau-laws}.  For Type C, the rational function
$6912t(1-t)$ has simple zeros at $0$ and $1$, simple ramification over
$1728$ at $t=1/2$, and a pole of order two at infinity.  It therefore
induces the orbifold morphism of signature $(3,6,\infty)$ and cusp width
two.  The normalized elliptic lifts are
\[
 \begin{pmatrix}1&-1\\1&0\end{pmatrix},
 \qquad
 \begin{pmatrix}0&-1\\1&1\end{pmatrix},
\]
which give \eqref{eq:typeC-tau-laws}.  Their products act by translations
by $1$ and $2$, respectively.
\end{proof}

Put $t_c=1/t$.  On distinguished components over a punctured cusp disc
choose logarithmic coordinates $\sigma_B=\sigma$ and $\sigma_C$ with
$\sigma_\kappa\circ g_{\kappa,0}=\sigma_\kappa-1$.  The Fourier expansion
of $j$ gives holomorphic functions $h_B=h$ and $h_C$ at $t_c=0$ such that
\begin{equation}
 \tau=\sigma+h(t_c)
\end{equation}
and
\begin{equation}
 \tau_C=2\sigma_C+h_C(t_c).
\end{equation}

\subsection{The automorphic line and the first affine coordinate}

\begin{lemma}[Automorphic roots and a homogeneous section]
\label{lem:aux-roots}
For $\kappa=B,C$, put $m_1=m_{\kappa,1}$ and $m_2=6$.
There are holomorphic roots on the orbifold universal cover
\[
 u_{\kappa,1}^{m_1}=t\circ\pi_\kappa,\qquad
 u_{\kappa,2}^{6}=1-t\circ\pi_\kappa
\]
with characters
\[
\begin{array}{c|cc}
 &g_{\kappa,1}&g_{\kappa,2}\\ \hline
 u_{\kappa,1}&\zeta_{m_1}^{-1}u_{\kappa,1}&u_{\kappa,1}\\
 u_{\kappa,2}&u_{\kappa,2}&\zeta_6^{-1}u_{\kappa,2}.
\end{array}
\]
The function
\begin{equation}\label{eq:uniform-eta-section}
 F_\kappa=\eta(\tau_\kappa)^{-2}
 u_{\kappa,1}^{m_1-1}u_{\kappa,2}
\end{equation}
satisfies
\begin{equation}\label{eq:homogeneous-eta-laws}
 F_\kappa\circ g_{\kappa,1}=-\frac{F_\kappa}{\tau_\kappa},\qquad
 F_\kappa\circ g_{\kappa,2}=\frac{F_\kappa}{\tau_\kappa+1},\qquad
 F_\kappa\circ g_{\kappa,0}=F_\kappa.
\end{equation}
At the cusp, $F_\kappa=t_c^{-1}$ times a holomorphic unit.
\end{lemma}

\begin{proof}
The zeros of $t\circ\pi_\kappa$ and $1-t\circ\pi_\kappa$ have
multiplicities $m_1$ and $6$, so the required roots exist locally.
Their transition ratios form a \v{C}ech cocycle with values in the
corresponding group of roots of unity.  It vanishes on the simply
connected universal cover, giving global roots.  Continuation around
the chosen clockwise generators gives the stated characters.
Write $S\tau=-1/\tau$ and $T\tau=\tau+1$.
The eta identities \cite[Proposition~5.1 and Theorem~5.2]{KongTeo} give
\[
 \eta(S\tau)^{-2}=\frac{i}{\tau}\eta(\tau)^{-2},\qquad
 \eta(TS\tau)^{-2}=\frac{\zeta_6}{\tau}\eta(\tau)^{-2},\qquad
 \eta(ST\tau)^{-2}=\frac{\zeta_6}{\tau+1}\eta(\tau)^{-2}.
\]
The first elliptic transformations are $S$ in Type B and $TS$ in
Type C; the second is $ST$ in both.  For the first generator the
root factors give $i(-i)^3=-1$ in Type B and
$\zeta_6\zeta_3^{-2}=-1$ in Type C.  For the second, the factor
$\zeta_6^{-1}$ from $u_{\kappa,2}$ cancels the eta multiplier.
Composition gives the cusp identity.

Since $q=e^{2\pi i\tau_\kappa}=t_c^{h_\kappa}$ times a holomorphic
unit, the product formula for eta gives the cusp order
\[
 -\frac{h_\kappa}{12}-\frac{m_1-1}{m_1}-\frac16=-1
 \quad\text{for }(m_1,h_\kappa)=(4,1),(3,2).
\]
Cusp invariance makes this an integral order on the punctured coarse
disc; the remaining factor extends as a holomorphic unit.
\end{proof}

The first affine coordinates are required to satisfy
\begin{equation}\label{eq:mu-laws}
 \mu\circ g_{B,1}=1-\frac{\mu}{\tau},
 \qquad
 \mu\circ g_{B,2}=\frac{\mu}{\tau+1},
\end{equation}
and
\begin{equation}\label{eq:typeC-mu-laws}
 \mu_C\circ g_{C,1}=1-\frac{\mu_C}{\tau_C},
 \qquad
 \mu_C\circ g_{C,2}=\frac{\mu_C}{\tau_C+1}.
\end{equation}
Differences of two solutions obey the common homogeneous law
\begin{equation}\label{eq:nu-laws}
 \nu\circ g_{\kappa,1}=-\frac{\nu}{\tau_\kappa},
 \qquad
 \nu\circ g_{\kappa,2}=\frac{\nu}{\tau_\kappa+1},
 \qquad
 \nu\circ g_{\kappa,0}=\nu.
\end{equation}

\begin{lemma}[The homogeneous line bundle]\label{lem:homogeneous-bundle}
For both types, the sheaf of holomorphic functions satisfying
\eqref{eq:nu-laws} and regular at the cusp is a line bundle on the
coarse base, isomorphic to $\mathcal O_{\mathbb P^1}(-1)$.
\end{lemma}

\begin{proof}
At the first root point, the homogeneous multiplier at the fixed point
is $\zeta_{m_1}$; a local root coordinate transforms by
$\zeta_{m_1}^{-1}$.  Equivariant functions therefore have orders
congruent to $m_1-1$ modulo $m_1$.  At the second point their orders
are congruent to $1$ modulo $6$.  Since eta has no zeros in $\mathbb H$,
$F_\kappa$ has exactly these minimal orders, and is a local generator
of the invariant direct image at both finite points.  This direct
image is locally free of rank one: dividing an equivariant function
by $F_\kappa$ gives an invariant holomorphic function on each root
disc.  Away from the finite points and cusp, $F_\kappa$ is nonzero.
At the cusp it has order $-1$ by Lemma~\ref{lem:aux-roots}.
Its divisor as a meromorphic section of the coarse line bundle is
therefore $-p_{\kappa,0}$, proving the assertion.
\end{proof}

\begin{proposition}[First affine period coordinates]

There are unique holomorphic functions $\mu$ and $\mu_C$ satisfying
\eqref{eq:mu-laws} and \eqref{eq:typeC-mu-laws}, respectively, such that
\[
 \mu-\tau
 \quad\text{and}\quad
 \mu_C-\frac{\tau_C}{2}
\]
extend holomorphically across the cusp.
\end{proposition}

\begin{proof}
In each case the local solutions form a torsor under the line bundle of
Lemma~\ref{lem:homogeneous-bundle}.  On a root chart, local solutions are
first constructed upstairs and then averaged over the finite stabilizer;
finite-stabilizer invariants are exact over $\mathbb C$, so the resulting
torsor descends to the coarse curve.  At the first elliptic point one may
use $(\tau+1)/2$ in Type B and $(\tau_C+1)/3$ in Type C; at the second
elliptic point use $0$, and at the cusp use $\tau$ or $\tau_C/2$.
Existence and uniqueness follow from
\[
 H^1(\Pone,\mathcal O(-1))=H^0(\Pone,\mathcal O(-1))=0.
\]
\end{proof}

Write
\begin{equation}
 m(t_c)=\mu-\tau
\end{equation}
and
\begin{equation}
 m_C(t_c)=\mu_C-\sigma_C.
\end{equation}
Both functions are holomorphic at $t_c=0$.

\subsection{The second affine coordinate}

For Type B put
\[
 \phi_1=1-\frac{2\mu^2}{\tau},
 \qquad
 \phi_2=-\frac{2\mu^2}{\tau+1},
\]
and require
\begin{equation}\label{eq:beta-laws}
 \beta\circ g_{B,1}=\beta+\phi_1,
 \qquad
 \beta\circ g_{B,2}=\beta+\phi_2.
\end{equation}
For Type C put
\[
 \phi^C_1=1-\frac{3\mu_C^2}{\tau_C},
 \qquad
 \phi^C_2=-\frac{3\mu_C^2}{\tau_C+1},
\]
and require
\begin{equation}\label{eq:typeC-beta-laws}
 \beta_C\circ g_{C,1}=\beta_C+\phi^C_1,
 \qquad
 \beta_C\circ g_{C,2}=\beta_C+\phi^C_2.
\end{equation}
The cusp conditions are that $\beta-\tau$ and
$\beta_C-\tau_C/2$ extend holomorphically.

\begin{lemma}[Finite-order consistency]

The additive cocycles have zero norm on the finite cyclic groups:
\[
 \sum_{j=0}^{3}\phi_1\circ g_{B,1}^j=0,
 \qquad
 \sum_{j=0}^{5}\phi_2\circ g_{B,2}^j=0,
\]
\[
 \sum_{j=0}^{2}\phi^C_1\circ g_{C,1}^j=0,
 \qquad
 \sum_{j=0}^{5}\phi^C_2\circ g_{C,2}^j=0.
\]
At the cusps,
\[
 (\tau,\mu,\beta)\circ(g_{B,1}g_{B,2})
 =(\tau+1,\mu+1,\beta+1)
\]
and
\begin{equation}
 (\tau_C,\mu_C,\beta_C)\circ(g_{C,1}g_{C,2})
 =(\tau_C+2,\mu_C+1,\beta_C+1).
\end{equation}
\end{lemma}

\begin{proof}
The order-four orbit of $(\tau,\mu)$ is
\[
 (\tau,\mu),\quad
 \left(-\frac1\tau,\frac{\tau-\mu}{\tau}\right),\quad
 (\tau,\tau+1-\mu),\quad
 \left(-\frac1\tau,\frac{\mu-1}{\tau}\right),
\]
while the order-three orbit of $(\tau_C,\mu_C)$ is
\[
 (\tau_C,\mu_C),\quad
 \left(\frac{\tau_C-1}{\tau_C},
       \frac{\tau_C-\mu_C}{\tau_C}\right),\quad
 \left(-\frac1{\tau_C-1},
       \frac{\mu_C-1}{\tau_C-1}\right).
\]
The order-six orbit is common to the two constructions after replacing
$(\tau,\mu)$ by $(\tau_C,\mu_C)$.  Substitution gives the four norm-zero
identities.  Composition of the two affine transformations gives the
cusp laws.
\end{proof}

\begin{proposition}[Second affine period coordinates]

The solutions of \eqref{eq:beta-laws} with $\beta-\tau$ regular at the
cusp form a nonempty affine line
\[
 \{\beta_{\mathrm{part}}+c_B:c_B\in\C\}.
\]
The solutions of \eqref{eq:typeC-beta-laws} with
$\beta_C-\tau_C/2$ regular at the cusp form a nonempty affine line
\[
 \{\beta_{C,\mathrm{part}}+c_C:c_C\in\C\}.
\]
\end{proposition}

\begin{proof}
For a cyclic generator $g$ of order $m$, the zero-norm identity for an
additive cocycle $\phi$ gives the local solution
\[
 b=\frac1m\sum_{j=0}^{m-1}j\,\phi\circ g^j.
\]
At the cusps use $\tau$ and $\tau_C/2$.  Differences of solutions are
invariant functions regular at the cusp; hence each solution sheaf is a
torsor under $\mathcal O_{\Pone}$.  The claim follows from
$H^1(\Pone,\mathcal O)=0$ and $H^0(\Pone,\mathcal O)=\C$.
\end{proof}

Put
\begin{equation}
 b(t_c)=\beta-\tau
\end{equation}
and
\begin{equation}
 b_C(t_c)=\beta_C-\sigma_C.
\end{equation}
These are holomorphic at the cusp.

\subsection{Nondegeneracy and equivariance}

Set
\[
 Z_B=\begin{pmatrix}2\mu&\tau\\\beta&\mu\end{pmatrix},
 \qquad
 Z_C=\begin{pmatrix}3\mu_C&\tau_C\\\beta_C&\mu_C\end{pmatrix},
 \qquad
 \Pi_\kappa=[Z_\kappa\mid I_2],
\]
and
\[
 \mathcal D_B=\operatorname{Im}\beta
 -2\frac{(\operatorname{Im}\mu)^2}{\operatorname{Im}\tau},
 \qquad
 \mathcal D_C=\operatorname{Im}\beta_C
 -3\frac{(\operatorname{Im}\mu_C)^2}{\operatorname{Im}\tau_C}.
\]

\begin{proposition}[Nondegenerate period families]\label{prop:nondegenerate-periods}

The functions $\mathcal D_B$ and $\mathcal D_C$ are invariant under their
respective deck groups.  For sufficiently negative imaginary part of
$c_B$ or $c_C$, they are negative everywhere.  Hence every matrix
$\Pi_\kappa(z)$ defines a lattice in $\C^2$.
\end{proposition}

\begin{proof}
The transformation laws give
\[
 -\det\operatorname{Im}Z_\kappa
 =(\operatorname{Im}\tau_\kappa)\mathcal D_\kappa.
\]
For an elliptic generator with lower row $(c,d)$, period equivariance
implies
\[
 \det\operatorname{Im}Z_\kappa(gz)
 =|c\tau_\kappa+d|^{-2}\det\operatorname{Im}Z_\kappa(z),
 \qquad
 \operatorname{Im}\tau_\kappa(gz)
 =|c\tau_\kappa+d|^{-2}\operatorname{Im}\tau_\kappa(z).
\]
Hence $\mathcal D_\kappa(gz)=\mathcal D_\kappa(z)$ exactly.  Moreover the
real $4\times4$ matrix of the four period columns has determinant, up to
sign, $\det\operatorname{Im}Z_\kappa$; therefore
$\mathcal D_\kappa<0$ implies the required real-linear
independence.  At the cusps,
\[
 \mathcal D_B=-\operatorname{Im}\tau+O(1),
 \qquad
 \mathcal D_C=-\frac14\operatorname{Im}\tau_C+O(1).
\]
Thus each function tends to $-\infty$ at the cusp and is bounded above on
the compact complement.  Changing the affine constant adds its imaginary
part to $\mathcal D_\kappa$.
\end{proof}

Fix the particular solutions $\beta_{\kappa,0}$ and let
$\mathcal D_{\kappa,0}$ be the corresponding scalar.  It descends to
the coarse base minus the cusp, is continuous at the finite points,
and tends to $-\infty$ at the cusp.  Hence
$M_\kappa=\max\mathcal D_{\kappa,0}$ is finite.  Choose once and for all
$C_\kappa\leq-M_\kappa$ and set
\begin{equation}\label{eq:admissible-domain}
 \mathcal U_\kappa=\{c\in\C:\operatorname{Im}c<C_\kappa\}.
\end{equation}
Throughout the paper, an \emph{admissible parameter} means a point of
this fixed half-plane.  Then
$\beta_\kappa=\beta_{\kappa,0}+c$ satisfies $\mathcal D_\kappa<0$
everywhere.  Real translations preserve $\mathcal U_\kappa$.

For Type B let $T_1,T_2$ be the standard lifts in
\eqref{eq:typeB-standard-matrices}, and set $A^B_j=(T_j^{-1})^{\mathsf T}$.
For Type C put
\[
 T^C_1=
 \begin{pmatrix}
 1&-3&0&1\\
 0&-1&1&1\\
 0&-1&0&0\\
 0&0&0&1
 \end{pmatrix},
 \qquad
 T^C_2=
 \begin{pmatrix}
 1&0&0&0\\
 0&0&-1&0\\
 0&1&1&0\\
 0&0&0&1
 \end{pmatrix},
 \qquad
 A^C_j=((T^C_j)^{-1})^{\mathsf T}.
\]
Define
\[
 R^B_1=\begin{pmatrix}-1/\tau&0\\-\mu/\tau&1\end{pmatrix},
 \quad
 R^B_2=\begin{pmatrix}1/(\tau+1)&0\\-\mu/(\tau+1)&1\end{pmatrix},
\]
\[
 R^C_1=\begin{pmatrix}-1/\tau_C&0\\-\mu_C/\tau_C&1\end{pmatrix},
 \quad
 R^C_2=\begin{pmatrix}1/(\tau_C+1)&0\\-\mu_C/(\tau_C+1)&1\end{pmatrix}.
\]

\begin{proposition}[Period equivariance]

For $\kappa=B,C$ and $j=1,2$,
\[
 \Pi_\kappa(g_{\kappa,j}z)
 =R^\kappa_j(z)\Pi_\kappa(z)(A^\kappa_j)^{-1}.
\]
Consequently the quotients are holomorphic families of compact complex
two-tori over the complements of the three special points.
\end{proposition}

\begin{proof}
The last two columns determine the change of complex basis $R^\kappa_j$.
For the first two columns the identity is equivalent to the transformation
laws for $\tau_\kappa,\mu_\kappa,\beta_\kappa$ established above.
\end{proof}

\paragraph{Deck matrices and forward transport.}
\label{par:deck-forward-convention}
The matrices $A_j$ in the period identity describe the deck action on
the homology marking.  Write
\begin{equation}\label{eq:deck-forward-distinction}
 \mathsf D_j:=A_j,\qquad \mathsf P_j:=\mathsf D_j^{-1}.
\end{equation}
Here $\mathsf P_j$ is the forward return map on homology along the
chosen clockwise meridian.  Indeed, on the root arc
$s(q)=s_0e^{-2\pi iq/m_j}$ a flat period $\nu\in\Lambda$ arrives as
$\Pi(gs_0)\nu$.  The quotient identifies $(s,u)$ with $(gs,R(s)u)$,
so returning the endpoint to the initial fibre gives
\begin{equation}\label{eq:clockwise-forward-period}
 R(s_0)^{-1}\Pi(gs_0)\nu
 =\Pi(s_0)A_j^{-1}\nu=\Pi(s_0)\mathsf P_j\nu.
\end{equation}
The same calculation applies to the affine root action on homology,
since translations induce the identity there.  With the path convention
above,
\begin{equation}\label{eq:path-deck-transport-product}
 \mathsf D(a_1a_2)=\mathsf D_1\mathsf D_2,\qquad
 \mathsf P(a_1a_2)=\mathsf P_2\mathsf P_1
 =(\mathsf D_1\mathsf D_2)^{-1}.
\end{equation}
The contragredient deck matrices are $T_j=(A_j^{-1})^{\mathsf T}$;
forward cohomology transport is $T_j^{-1}=A_j^{\mathsf T}$.
A matrix and its inverse have the same fixed lattice, so the invariant
calculations below may use $T_j$.  These conventions also apply to the
index-$r$ matrices constructed next.

\subsection{Isogenies and cusp expansions}

For $r\ge1$ define
\begin{equation}
 T^B_{1,r}=
 \begin{pmatrix}
 1&-2r&0&r\\
 0&0&1&1\\
 0&-1&0&0\\
 0&0&0&1
 \end{pmatrix},
 \qquad T^B_{2,r}=T_2,
\end{equation}
\begin{equation}
 T^C_{1,r}=
 \begin{pmatrix}
 1&-3r&0&r\\
 0&-1&1&1\\
 0&-1&0&0\\
 0&0&0&1
 \end{pmatrix},
 \qquad T^C_{2,r}=T^C_2.
\end{equation}
Let $A^\kappa_{j,r}=((T^\kappa_{j,r})^{-1})^{\mathsf T}$ and put
\begin{equation}
 \Pi^B_r=
 \begin{pmatrix}
 2r\mu&\tau&1&0\\
 r\beta&\mu&0&1
 \end{pmatrix},
\end{equation}
\begin{equation}
 \Pi^C_r=
 \begin{pmatrix}
 3r\mu_C&\tau_C&1&0\\
 r\beta_C&\mu_C&0&1
 \end{pmatrix}.
\end{equation}
The same matrices $R^\kappa_j$ give equivariance.  With
$D_r=\operatorname{diag}(r,1,1,1)$,
\[
 \Pi^\kappa_r=\Pi^\kappa_1D_r,
 \qquad
 A^\kappa_{j,1}D_r=D_rA^\kappa_{j,r}.
\]
Thus the two open families form isogeny towers of degree $r$.

For Type B, write $\tau=\sigma+h$, $\mu=\tau+m$, and $\beta=\tau+b$.
Then
\begin{equation}\label{eq:typeB-cusp-decomposition}
 Z^B_r=\sigma B^B_{0,r}+C^B_r,
 \qquad
 B^B_{0,r}=\begin{pmatrix}2r&1\\r&1\end{pmatrix},
\end{equation}
where
\[
 C^B_r=
 \begin{pmatrix}
 2r(h+m)&h\\r(h+b)&h+m
 \end{pmatrix},
\]
and
\begin{equation}
 C^B_r(B^B_{0,r})^{-1}
 =\begin{pmatrix}h+2m&-2m\\b-m&-b+h+2m\end{pmatrix}.
\end{equation}
For Type C, write $\tau_C=2\sigma_C+h_C$, $\mu_C=\sigma_C+m_C$,
and $\beta_C=\sigma_C+b_C$.  Then
\begin{equation}\label{eq:typeC-cusp-decomposition}
 Z^C_r=\sigma_CB^C_{0,r}+C^C_r,
 \qquad
 B^C_{0,r}=\begin{pmatrix}3r&2\\r&1\end{pmatrix},
\end{equation}
where
\[
 C^C_r=\begin{pmatrix}3rm_C&h_C\\rb_C&m_C\end{pmatrix},
\]
and
\begin{equation}
 C^C_r(B^C_{0,r})^{-1}
 =\begin{pmatrix}
 -h_C+3m_C&3h_C-6m_C\\
 b_C-m_C&-2b_C+3m_C
 \end{pmatrix}.
\end{equation}
Both normalized twists are independent of $r$, and
\[
 \operatorname{SNF}(B^B_{0,r})
 =\operatorname{SNF}(B^C_{0,r})
 =\operatorname{diag}(1,r).
\]

For the topological marking let
$\Lambda=\Z\langle\gamma,u,w,\delta\rangle$ be the first-homology lattice.
We use the invariant functional
\[
 \chi=\gamma^*:\Lambda\longrightarrow\Z,
 \qquad\chi(\gamma)=1,\quad\chi(u)=\chi(w)=\chi(\delta)=0,
\]
and put $\Lambda_{\mathrm{van}}=\Z\langle w,\delta\rangle$.
Proposition~\ref{prop:cusp-pi1} identifies this primitive angular
sublattice with the kernel of $H_1(F;\Z)\to H_1(N_{0,r}^\kappa;\Z)$.
It must be distinguished from the sublattice
$\operatorname{im}(\mathsf M^\kappa_{0,r}-I)$, whose index in it is $r$.  The two cusp deck matrices are
\begin{equation}
 \mathsf M^B_{0,r}=
 \begin{pmatrix}
 1&0&0&0\\
 0&1&0&0\\
 2r&1&1&0\\
 r&1&0&1
 \end{pmatrix}
\end{equation}
and
\begin{equation}\label{eq:typeC-cusp-monodromy}
 \mathsf M^C_{0,r}=
 \begin{pmatrix}
 1&0&0&0\\
 0&1&0&0\\
 3r&2&1&0\\
 r&1&0&1
 \end{pmatrix}.
\end{equation}
In both cases
\[
 \ker(\mathsf M^\kappa_{0,r}-I)=\Lambda_{\mathrm{van}},
\]
and the induced map
$\Lambda/\Lambda_{\mathrm{van}}\to\Lambda_{\mathrm{van}}$ has matrix
$B^\kappa_{0,r}$ in the bases $(\bar\gamma,\bar u)$ and $(w,\delta)$.
Its image therefore has index $r$.
For forward transport along the clockwise cusp meridian, put
\begin{equation}\label{eq:clockwise-cusp-transport}
 \mathsf D^\kappa_{0,r}:=\mathsf M^\kappa_{0,r}
 =(A^\kappa_{1,r}A_{2,r})^{-1},\qquad
 \mathsf P^\kappa_{0,r}:=(\mathsf D^\kappa_{0,r})^{-1}
 =\begin{pmatrix}I_2&0\\-B^\kappa_{0,r}&I_2\end{pmatrix}.
\end{equation}
To see the sign directly, continuation along this meridian changes
$\sigma_\kappa$ to $\sigma_\kappa-1$.  A flat lattice vector
$(\alpha,\beta)\in\mathbb Z^2\oplus\mathbb Z^2$ therefore returns as
$(\alpha,\beta-B^\kappa_{0,r}\alpha)$ in the initial period marking.
Thus $\mathsf D_{0,r}-I$ induces $B^\kappa_{0,r}$, whereas
$\mathsf P_{0,r}-I$ induces $-B^\kappa_{0,r}$.  Their integral kernels
and images agree.  More generally, for every integral automorphism $D$,
\begin{equation}\label{eq:inverse-integral-image}
 D^{-1}-I=(D-I)(-D^{-1}),\qquad
 \ker(D^{-1}-I)=\ker(D-I),\quad
 \operatorname{im}(D^{-1}-I)=\operatorname{im}(D-I).
\end{equation}
The first identity is an integral change of basis on the source; it
applies also to every exterior power of $D$.

\section{Local completions}\label{sec:local-completions}

The period families are completed over the three missing points of the
base.  At the two finite-monodromy points we use affine logarithmic
transforms.  At the unipotent point we use a periodic toroidal model.
The finite construction is elementary once freeness is checked; the
periodic quotient requires a separate properness theorem because the
underlying toric space is not of finite type.

\subsection{Logarithmic transforms at the finite-monodromy points}
\label{sec:finite}

Fix integers $n\geq1$ and $m\geq2$.  Let $\Lambda$ be a lattice of
rank $2n$ and let $\mathcal T\to\Delta_s$ be the holomorphic torus family
\[
 \mathcal T_s=\mathbb C^n/\Pi(s)\Lambda,
\]
where $\Pi(s)$ is holomorphic and maps $\Lambda_{\mathbb R}$
real-linearly isomorphically onto $\mathbb C^n$ for every $s$.
Put $g(s)=\zeta_m^{-1}s$.  Suppose that $A\in GL(\Lambda)$,
$A^m=I$, and a holomorphic map $R:\Delta_s\to GL_n(\mathbb C)$ satisfy
\begin{equation}\label{eq:root-holomorphic-equivariance}
 \Pi(gs)A=R(s)\Pi(s),\qquad
 R(g^{m-1}s)\cdots R(gs)R(s)=I_n.
\end{equation}
The second identity also follows from the first and $A^m=I$, since
the period columns span $\mathbb C^n$.  Let $v\in\Lambda^A$.

\begin{lemma}[Affine logarithmic transform]\label{lem:log-transform}
Under \eqref{eq:root-holomorphic-equivariance}, the formula
\begin{equation}\label{eq:holomorphic-affine-root-action}
 \Phi(s,[u])=\left(gs,\left[R(s)u+\frac1m\Pi(gs)v\right]\right)
\end{equation}
defines a holomorphic action of order $m$ on $\mathcal T$.
Suppose that $\chi\in\Lambda^\vee$ is primitive and $A$-invariant.  If
\[
 \gcd\bigl(m,\chi(v)\bigr)=1,
\]
then every nonidentity element acts freely.  The quotient is smooth,
and its central divisor is $mS$, where $S$ is the smooth quotient of
$\mathcal T_0$ by the induced affine action.  If $\mathcal T\to\Delta_s$
is proper, the quotient is proper over the coarse disc $t=s^m$.
\end{lemma}

\begin{proof}
The period identity makes \eqref{eq:holomorphic-affine-root-action}
well defined on lattice quotients, and both its linear part and its
translation section are holomorphic.  In the flat real coordinates
$x\in\Lambda_{\mathbb R}/\Lambda$, where $u=\Pi(s)x$, it becomes
\[
 (s,x)\longmapsto(gs,Ax+v/m).
\]
Since $Av=v$, its $k$th iterate has fibre coordinate $A^kx+kv/m$.
The $m$th iterate is translation by the integral period $v$, hence is
the identity.  The rotation of the base has exact order $m$.
A nontrivial power has no fixed point off $s=0$.  A fixed point on the
central torus would give, for some $1\leq k<m$ and
$\widetilde x\in\Lambda_{\mathbb R}$,
\[
 (A^k-I)\widetilde x+\frac{k}{m}v\in\Lambda.
\]
Applying the real-linear extension of $\chi$ gives
$m\mid k\chi(v)$, contrary to coprimality.  The free finite holomorphic
quotient is therefore smooth.  The local coordinate $s$ vanishes to
order one on the root central fibre, whereas $t=s^m$, so the quotient
central divisor has multiplicity $m$.  The root map is finite, and a
finite quotient of a proper family over it is proper over the coarse disc.
\end{proof}

For the families below, the period equivariance of
Section~\ref{sec:period}, restricted to each root disc, is precisely
\eqref{eq:root-holomorphic-equivariance}.  The periods and the matrices
$R_j^\kappa$ are holomorphic there.  Iterating the period identity and
using $(A_{j,r}^\kappa)^{m_j}=I$ gives its finite-order cocycle identity.

For the index-$r$ Type B matrices, the dual deck matrices
$A_{j,r}=(T_{j,r}^{-1})^{\mathsf T}$ fix
\[
 v_{1,r}=(1,-r,-r,0)^{\mathsf T},
 \qquad
 v_{2,r}=(-1,0,0,0)^{\mathsf T}.
\]
For Type C put
\begin{equation}\label{eq:typeC-filling-vectors}
 v^C_{1,r}=(1,-r,-r,0)^{\mathsf T},
 \qquad
 v^C_{2,r}=(-1,0,0,0)^{\mathsf T}.
\end{equation}
The corresponding dual monodromies fix these vectors.  In both families
$\chi(v_1)=1$ and $\chi(v_2)=-1$, so the lemma applies to the
multiplicities $(4,6)$ and $(3,6)$.

\begin{proposition}[The reduced finite fibres and their normal characters]
\label{prop:bielliptic-fibres}
For the Type B and Type C actions, every reduced finite fibre is a
bielliptic surface.  If its multiplicity is $m$, then its normal bundle
has exact order $m$.
\end{proposition}

\begin{proof}
Let $T\to S=T/C_m$ be the quotient of the central complex two-torus.  It
is finite and \'etale, hence $S$ is compact K\"ahler.  The invariant
subspace of $H^0(T,\Omega_T^1)$ is one-dimensional in each case, so
$q(S)=1$.  Since $\chi(\mathcal O_T)=0$ and the cover is \'etale,
$\chi(\mathcal O_S)=0$, whence $p_g(S)=0$.  The determinant character
of the nontrivial linear holonomy makes $K_S$ torsion, so
$\kappa(S)=0$.  The Enriques--Kodaira classification therefore identifies
$S$ as a bielliptic surface; see \cite[Chapter~VI]{BHPV}.

The normal line of the central torus is generated by
$\partial/\partial s$, on which the cyclic generator acts through the
character
\[
 \vartheta_m(\zeta_m)=\zeta_m^{-1}.
\]
Thus the quotient normal bundle $N_{S/N}$ satisfies
$N_{S/N}^{\otimes m}\simeq\mathcal O_S$.  If a smaller positive tensor
power were trivial, its nowhere-zero section would pull back to a
holomorphic function on the compact torus transforming by
$\vartheta_m^d$.  Such a function is constant, forcing $\vartheta_m^d=1$.
Since $\vartheta_m$ has exact order $m$, the normal bundle also has exact
order $m$.
\end{proof}

\begin{lemma}[The two collar identifications]\label{lem:finite-collar}
For either family, choose a branch
$\ell(s)=\log(s)/(2\pi i)$.  The maps
\[
 \Theta_{\mathrm{lin\to aff}}(s,u)
 =\bigl(s,u-\ell(s)\Pi_r^\kappa(z)v_{j,r}^\kappa\bigr),
\]
\[
 \phi_{\mathrm{aff\to lin}}(s,u)
 =\bigl(s,u+\ell(s)\Pi_r^\kappa(z)v_{j,r}^\kappa\bigr)
\]
are inverse biholomorphisms on the punctured collars.  The first
conjugates the linear model to the affine root-disc model and is the map
used in the global gluing; the second is its inverse.  They cover $t=s^{m_j^\kappa}$,
intertwine the monodromies, and are independent of the logarithm branch
on the torus quotient.
\end{lemma}

\begin{proof}
The affine generator sends $s$ to $\zeta_{m_j}^{-1}s$ and acts on the
fibre by the linear monodromy followed by translation by
$v_{j,r}^\kappa/m_j^\kappa$.  Since
$A_{j,r}^\kappa v_{j,r}^\kappa=v_{j,r}^\kappa$, substitution gives the two
conjugacy identities.  Replacing $\log s$ by $\log s+2\pi i n$ changes
the fibre coordinate by the integral period
$\pm n\Pi_r^\kappa(z)v_{j,r}^\kappa$, so the maps descend to the torus
quotient.
\end{proof}

\begin{proposition}[Boundary map of a finite logarithmic filling]
\label{prop:finite-filling-pi1}
Let $N_j$ be the quotient in Lemma~\ref{lem:log-transform}, let $M_j$ be
its boundary over a circle in the coarse base, and write $t_\lambda$ for
translation by $\lambda\in\Lambda$.  Let $a_j$ be the lift of the
clockwise meridian through the linear zero section, transferred by the
collar map.  Put $\mathsf D_j=A_j$ and $\mathsf P_j=A_j^{-1}$.
The boundary is the mapping torus with forward return $\mathsf P_j$,
namely $(F\times[0,1])/((x,1)\sim(\mathsf P_jx,0))$.  With
left-to-right traversal of path products its fundamental group is
\[
 \pi_1(M_j)\cong\Lambda\rtimes_{\mathsf D_j}\mathbb Z\langle a_j\rangle,
 \qquad a_jt_\lambda a_j^{-1}=t_{\mathsf D_j\lambda}.
\]
Equivalently, $a_j^{-1}t_\lambda a_j=t_{\mathsf P_j\lambda}$.
Moreover
\[
 \pi_1(N_j)=
 \left\langle t_\lambda,g_j\ \middle|\
 t_\lambda t_\mu=t_{\lambda+\mu},\
 g_jt_\lambda g_j^{-1}=t_{A_j\lambda},\
 g_j^{m_j}=t_{v_j}\right\rangle.
\]
The boundary inclusion sends $t_\lambda$ to $t_\lambda$ and $a_j$ to
$g_j$; its kernel is the normal closure of
$a_j^{m_j}t_{-v_j}$.
\end{proposition}

\begin{proof}
In the mapping-torus square, a fibre loop $t_\lambda$ at the initial
end becomes $t_{\mathsf P_j\lambda}$ at the terminal end.  Its boundary
relation is $t_\lambda a_jt_{\mathsf P_j\lambda}^{-1}a_j^{-1}=1$
when paths are traversed from left to right.  This is exactly the first
presentation; the conjugation matrix is the inverse of the forward return.

For the finite filling, the collar image of the linear zero section
along $s(q)=s_0e^{-2\pi iq/m_j}$ has complex coordinate
$-\ell(s_0)\Pi(s(q))v_j+(q/m_j)\Pi(s(q))v_j$.
The first term is an invariant holomorphic section and is removed by
a translation isotopy.  The remaining path has flat real coordinate
$qv_j/m_j$.  At the central torus the remaining path
$q\mapsto[qv_j/m_j]$ ends at the affine deck translate of its starting
point.  Denote its class by $g_j$.  The affine deck transformation
conjugates lattice translations by $A_j$, and its $m_j$th power is
translation by $v_j$.  Thus $g_j^{m_j}=t_{v_j}$, with the positive sign
in the stated presentation.  Filling the root disc introduces exactly
$a_j^{m_j}t_{-v_j}=1$, and the boundary inclusion sends $a_j$ to $g_j$.
This proves the kernel assertion.
\end{proof}

\subsection{The periodic toroidal model}\label{sec:cusp}

Put $e_1=(1,0)$ and $e_2=(0,1)$.  For $v\in\mathbb Z^2$ let
\[
 \Delta_v^+=\operatorname{conv}(v,v+e_1,v+e_2),
 \qquad
 \Delta_v^-=\operatorname{conv}(v+e_1+e_2,v+e_1,v+e_2).
\]
These triangles form the standard $\mathbb Z^2$-periodic
$A_2$ triangulation of $\mathbb R^2$.

\begin{lemma}[The periodic $A_2$ fan]\label{lem:periodic-A2-fan}
Let $N'=\mathbb Z^2\oplus\mathbb Z$.  For every simplex $\tau$ of the
periodic triangulation, let
\[
 \sigma_\tau=\mathbb R_{\geq0}(\tau\times\{1\})\subset N'_\mathbb R,
\]
and include all faces and the zero cone.  These cones form a smooth fan
$\Sigma_{A_2}$.  The triangulation is locally finite, every nonzero cone
has finite star, and every maximal cone is unimodular.  The affine toric
charts therefore glue to a separated smooth analytic toric threefold
$Y$.  The character of the last lattice coordinate defines
\[
 t_c:Y\longrightarrow\mathbb C.
\]
The irreducible components of $t_c^{-1}(0)$ are indexed by
$\mathbb Z^2$, and each is the toric del Pezzo surface $dP_6$ of degree six.
\end{lemma}

\begin{proof}
A compact subset of $\mathbb R^2$ meets only finitely many translates of
the two fundamental triangles, and every simplex belongs to finitely
many triangles.  Intersections of the cones are cones over common faces,
so the fan axiom holds.  For $\Delta_v^+$, subtracting the ray
$(v,1)$ from the other two rays gives $(e_1,0)$ and $(e_2,0)$; hence the
three primitive ray vectors form a basis of $N'$.  The same calculation
applies to $\Delta_v^-$.  Thus every maximal cone is unimodular.  The
usual affine toric gluing is separated because any two charts meet in
the chart of their common face.

The divisor corresponding to the ray through $(v,1)$ has fan equal to
the star of $v$ in the triangulation, modulo that ray.  This is the
hexagonal fan, so the divisor is $dP_6$.
\end{proof}

For a holomorphic matrix $D(t_c)$ define on the dense torus
\begin{equation}\label{eq:master-action}
 \widehat\Psi_\lambda(x,t_c)
 =\left(e^{2\pi iD(t_c)\lambda}t_c^\lambda x,t_c\right),
 \qquad \lambda\in\mathbb Z^2.
\end{equation}
Here $t_c^\lambda x$ means componentwise multiplication on
$(\mathbb C^*)^2$.  Translation of the triangulation by $\lambda$ sends
$U_\sigma$ to $U_{\sigma+\lambda}$, while the exponential factor is a
torus translation; together they extend \eqref{eq:master-action} to $Y$.

\begin{lemma}[Properness of the periodic action]\label{lem:tropical-properness}
Assume that $D(t_c,s)$ is holomorphic on
$\Delta_{\epsilon_0}\times S$.  For every compact $K\Subset S$, after
shrinking to a closed subdisc $\overline\Delta_{\epsilon_K}$, put
\[
 L(t_c,s)=I-\frac{2\pi\operatorname{Im}D(t_c,s)}{\log|t_c|}
 \qquad(0<|t_c|\leq\epsilon_K).
\]
Then $L$ and $L^{-1}$ are uniformly bounded on
$\Delta_{\epsilon_K}^*\times K$.  If
$C_1,C_2\subset Y_{\leq\epsilon_K}\times K$ are compact, only finitely
many $\lambda\in\mathbb Z^2$ satisfy
$\widehat\Psi_\lambda(C_1)\cap C_2\neq\varnothing$.  Thus the action is
properly discontinuous, uniformly on compact parameter sets.
\end{lemma}

\begin{proof}
On the dense torus define
\[
 y=-\frac{\log|x|}{\log|t_c|}.
\]
Then
\begin{equation*}\label{eq:tropical-action-shift}
 y(\widehat\Psi_\lambda p)=y(p)-L(t_c,s)\lambda.
 \tag{\ref*{lem:tropical-properness}.1}
\end{equation*}
We first record the chartwise bound that remains valid as $t_c\to0$.
Let $\sigma$ be a maximal cone over a triangle with vertices
$v_0,v_1,v_2$, let $m_0,m_1,m_2$ be the dual basis, and put
\[
 \ell_{\sigma,i}(q)=\langle m_i,(q,1)\rangle.
\]
The toric coordinates $z_i=\chi^{m_i}$ satisfy the exact identity
\begin{equation*}\label{eq:tropical-chart-logarithm}
 \log|z_i|=\log|t_c|\,\ell_{\sigma,i}(-y),
 \qquad \sum_i\ell_{\sigma,i}=1.
 \tag{\ref*{lem:tropical-properness}.2}
\end{equation*}
If a compact set in $U_\sigma$ is contained in $|z_i|\leq R$, then
$t_c=z_0z_1z_2$ gives, for $t_c\neq0$,
\[
 -\frac{\log R}{|\log|t_c||}\leq
 \ell_{\sigma,i}(-y)\leq
 1+\frac{2\log R}{|\log|t_c||}.
\]
Thus the tropical coordinates of dense-torus points in any fixed compact
chart portion are uniformly bounded, including along sequences with
$t_c\to0$.

Cover $C_1$ and $C_2$ by finitely many such compact chart portions.
If $q=\widehat\Psi_\lambda(p)$ with $t_c\neq0$, the two tropical
coordinates are uniformly bounded.  Equation
\eqref{eq:tropical-action-shift} and the uniform bound on $L^{-1}$ then
bound $\lambda$.  If $p$ and $q$ lie on the central fibre, they belong to
toric orbits $O_\sigma,O_\tau$ from finite sets of cones, and
$q=\widehat\Psi_\lambda(p)$ forces $\tau=\sigma+\lambda$; again only
finitely many $\lambda$ occur.  Hence the transporter of two compact sets
is finite.
\end{proof}

\begin{proposition}[Compact sets meeting every orbit]
\label{prop:uniform-fundamental-blocks}
Let $D(t_c,s)$ be as above.  For every compact $K\Subset S$ there are
$0<\epsilon_K<\epsilon_0$ and a compact set
\[
 \mathcal F_K\subset Y_{\leq\epsilon_K}\times K
\]
meeting every $\mathbb Z^2$-orbit.  It is the union of finitely many closed polydiscs in affine toric
charts, intersected with $Y_{\leq\epsilon_K}\times K$, and one
compact representative of the central component over $K$.  The same conclusion holds for every
finite-index subgroup of the deck lattice.
\end{proposition}

\begin{proof}
After shrinking $\epsilon_K$, $L$ and $L^{-1}$ are uniformly bounded.
For a point over $t_c\neq0$, choose $\lambda\in\mathbb Z^2$ so that
\[
 L(t_c,s)^{-1}y-\lambda\in[0,1]^2.
\]
The transformed tropical coordinate belongs to the bounded set
$L(t_c,s)[0,1]^2$.  Hence $-y$ lies in one of finitely many triangles of
the periodic triangulation.  In the chart of that triangle the exact
formula \eqref{eq:tropical-chart-logarithm} has
$0\leq\ell_{\sigma,i}(-y)\leq1$, so $|z_i|\leq1$.  Intersecting the corresponding finite union of closed unit
polydiscs with the inverse image of the closed disc gives a compact
set meeting every orbit over the punctured closed disc.

Over $t_c=0$, the deck lattice acts transitively on the normalization
components.  Add one entire representative component
$D_0\simeq dP_6$, which is compact and meets every central orbit after a
deck translation.  The resulting finite union is compact and meets every
orbit over the closed disc.  For a finite-index subgroup, add finitely
many component representatives and finitely many translates of the same
polydiscs.
\end{proof}

\begin{theorem}[Relative periodic toroidal quotient]\label{thm:master-cusp}
Let $S$ be a complex manifold, let $D(t_c,s)$ be holomorphic on
$\Delta_{\epsilon_0}\times S$, and let $B\in M_2(\Z)$ have nonzero
determinant.  Put $\Gamma=B\Z^2$ and $d=|\det B|$.
For every relatively compact open set $S'\Subset S$, there is
$\epsilon>0$ such that the action
\[
 \widehat\Psi_\nu(x,t_c,s)=
 \bigl(e^{2\pi iD(t_c,s)\nu}t_c^\nu x,t_c,s\bigr),
 \qquad \nu\in\Gamma,
\]
is free and properly discontinuous on $Y_{<\epsilon}\times S'$.
The quotient
\[
 \mathcal N_{B,D}=(Y_{<\epsilon}\times S')/\Gamma
 \longrightarrow\Delta_\epsilon\times S'
\]
is proper and has smooth Hausdorff total space; its projection to
$S'$ is a holomorphic submersion.  In local coordinates the map to the
disc is $t_c=z_0$, $z_0z_1$, or $z_0z_1z_2$.
For each parameter its central fibre is reduced with normal crossings;
its normalization is a disjoint union of $d$ copies of $dP_6$.  The
normalization of its double locus consists of $3d$ copies of $\mathbb P^1$;
it has $2d$ triple points and Euler number $2d$.
The quotient map
\[
 \mathcal N_{B,D}\longrightarrow\mathcal N_{I,D}
\]
is a finite \'etale cover of degree $d$, with group $\Z^2/\Gamma$.
The case of a fixed parameter gives a proper map from a smooth
complex threefold to the disc.
\end{theorem}

\begin{proof}
Apply Lemma~\ref{lem:tropical-properness} and
Proposition~\ref{prop:uniform-fundamental-blocks} over
$\overline{S'}$, shrinking $\epsilon$ uniformly.  They give proper
discontinuity and compact sets meeting every $\Gamma$-orbit over
compact subsets of the base; the finite-index case uses finitely
many translates of the same sets.

On the central fibre, $\widehat\Psi_\nu$ sends the toric orbit
$O_\sigma$ to $O_{\sigma+\nu}$.  A nonzero translation cannot fix a
finite simplex, so a central fixed point forces $\nu=0$.
For $t_c\neq0$, a fixed point would imply
\[
 (\log|t_c|I-2\pi\operatorname{Im}D(t_c,s))\nu=0.
\]
This matrix is invertible for small $|t_c|$, uniformly on
$\overline{S'}$.  Thus the action is free.  The quotient is Hausdorff
by proper discontinuity and smooth because the quotient map is a
local biholomorphism.  This also proves that the projection to $S'$
is a submersion.  Images of the compact orbit-meeting sets prove
properness over $\Delta_\epsilon\times S'$.

The local equations follow from the unimodular fan, with all
multiplicities equal to one.  The quotient triangulation has $d$
vertices, $3d$ edges, and $2d$ triangles.  Its vertex surfaces are
$dP_6$; inclusion--exclusion on the normalization gives
\[
 e=6d-2(3d)+2d=2d.
\]
Finally, the inclusion $\Gamma\subset\Z^2$ of index $d$ between
free properly discontinuous actions gives the stated finite
\'etale cover.
\end{proof}

\begin{remark}[Component counts are model-dependent]\label{rem:model-dependent-components}
The equality between a saturation index and a component count in
Theorem~\ref{thm:obstruction-discriminant} uses the specified periodic
$A_2$ model.  Indeed, blow up a smooth point of a central component away
from the double locus.  Before blowing up the map is locally $t=z_0$;
afterwards its charts have the forms $t=x$ or $t=xy$.  The total space
remains smooth and the central fibre remains reduced with normal
crossings, but an additional normalization component $\mathbb P^2$
appears.  Nothing changes over the punctured disc, so neither monodromy
nor its saturation index changes.
\end{remark}

\begin{corollary}[Finite-index period matrices]\label{thm:finite-index-reduction}
Let $B\in M_2(\Z)$ be nonsingular and let $C(t_c,s)$ be holomorphic.
For $D=CB^{-1}$ the action
\[
 \Psi^{B,C}_\lambda(x,t_c,s)=
 \bigl(e^{2\pi iC(t_c,s)\lambda}t_c^{B\lambda}x,t_c,s\bigr)
\]
equals $\widehat\Psi_{B\lambda}$.  Its quotient has all the
properties of Theorem~\ref{thm:master-cusp}, locally uniformly in $s$.
\end{corollary}

\begin{lemma}[The toric covering tube is simply connected]
\label{lem:toric-simply-connected}
For sufficiently small $\epsilon$, both
$Y_{<\epsilon}$ and $Y_{\leq\epsilon}$ are homotopy equivalent to $Y$
and are simply connected.
\end{lemma}

\begin{proof}
The cocharacter $(0,0,1)$ gives a positive-real torus action $\rho_a$ on
$Y$ with $t_c(\rho_a y)=a\,t_c(y)$.  The function
$a(y)=\epsilon/(\epsilon+|t_c(y)|)$ and the homotopy
\[
 H_s(y)=\rho_{(1-s)+sa(y)}y
\]
give homotopy equivalences from the open and closed tubes to $Y$.

To compute $\pi_1(Y)$, exhaust $\mathbb R^2$ by nested finite
connected patches of the periodic triangulation.  The cones over each patch and
their faces form a finite subfan $\Sigma_R$, with open toric subvariety
$Y_R\subset Y$.  Every compact subset of $Y$ lies in some $Y_R$.
Consequently every loop in $Y$, and every null-homotopy between loops,
is contained in some $Y_R$: the images of the relevant compact source
spaces are compact.  Hence the natural map
\[
 \varinjlim_R\pi_1(Y_R)\longrightarrow\pi_1(Y)
\]
is an isomorphism.  For a finite fan, van Kampen gives
\cite[Section~3.2]{Fulton}
\[
 \pi_1(Y_R)=N'/\langle n_\rho:\rho\in\Sigma_R(1)\rangle.
\]
Taking the direct limit gives the same formula with all rays of
$\Sigma_{A_2}$.  The three rays through $(0,0,1)$, $(1,0,1)$, and
$(0,1,1)$ already generate $N'=\mathbb Z^2\oplus\mathbb Z$.  Hence
$\pi_1(Y)=0$.
\end{proof}

Write $W_r^\kappa$ for the reduced central fibre of the toroidal
model.  For the topological statements fix $0<\rho<\epsilon$.  For
$\kappa\in\{B,C\}$ write
\begin{equation}\label{eq:closed-cusp-tube}
 N^\kappa_{0,r}(\rho)=
 (f_r^\kappa)^{-1}(|t_c|\leq\rho),
 \qquad
 M^\kappa_{0,r}(\rho)=
 (f_r^\kappa)^{-1}(|t_c|=\rho),
\end{equation}
Here $N^\kappa_{0,r}(\rho)$ is a compact oriented six-manifold
with boundary $M^\kappa_{0,r}(\rho)$, a closed oriented five-manifold
carrying the outward-normal-first boundary orientation.

\begin{proposition}[Cusp fundamental group and boundary map]
\label{prop:cusp-pi1}
For $\kappa\in\{B,C\}$ and every $r\geq1$,
\[
 \pi_1\bigl(N^\kappa_{0,r}(\rho)\bigr)
 \cong\overline\Lambda\cong\mathbb Z^2.
\]
For a nearby smooth fibre marked by $H_1(F;\mathbb Z)=\Lambda$, inclusion
induces
\[
 \Lambda\longrightarrow\overline\Lambda
 =\Lambda/\Lambda_{\mathrm{van}}.
\]
Thus the primitive vanishing lattice is
$\Lambda_{\mathrm{van}}=\mathbb Z\langle w,\delta\rangle$.

If $a_0$ is the clockwise meridian represented by the extended zero
section, its forward return is $\mathsf P^\kappa_{0,r}$ from
\eqref{eq:clockwise-cusp-transport}.  The boundary is its mapping torus,
and, with left-to-right path products,
\[
 \pi_1\bigl(M^\kappa_{0,r}(\rho)\bigr)
 \cong\Lambda\rtimes_{\mathsf D^\kappa_{0,r}}
 \mathbb Z\langle a_0\rangle,\qquad
 a_0^{-1}t_\lambda a_0=t_{\mathsf P^\kappa_{0,r}\lambda}.
\]
and the boundary inclusion is
\[
 t_\lambda\longmapsto t_{\bar\lambda},
 \qquad a_0\longmapsto1.
\]
Its kernel is the normal closure of
$\Lambda_{\mathrm{van}}\cup\{a_0\}$.
\end{proposition}

\begin{proof}
The preimage of the closed tube in $Y$ is simply connected by
Lemma~\ref{lem:toric-simply-connected}.  Its free deck group is
$\overline\Lambda$, so it is the universal cover of the cusp
neighbourhood.  Over $t_c\neq0$ the preimage of the fibre is
$(\mathbb C^*)^2$, with fundamental group $\Lambda_{\mathrm{van}}$.
The endpoint deck transformation gives the exact sequence
\[
 0\longrightarrow\Lambda_{\mathrm{van}}
 \longrightarrow\Lambda
 \longrightarrow\overline\Lambda\longrightarrow0.
\]

The point $(x_1,x_2)=(1,1)$ defines a section; in the chart over
$\Delta_0^+$ it is $(z_0,z_1,z_2)=(t_c,1,1)$, so it extends over the
central fibre.  Its boundary meridian bounds the section disc and is
null-homotopic in the closed tube.  The mapping-torus presentation uses
the same square relation as Proposition~\ref{prop:finite-filling-pi1}:
forward transport is $\mathsf P_{0,r}$ and conjugation by $a_0$ is
$\mathsf D_{0,r}$.  It gives the displayed boundary map and its normal kernel.
\end{proof}

\begin{corollary}[The two cusp-cover families]
\label{cor:typeB-cusp-ladder}
For the Type B matrices, the neighbourhoods
\[
 N^B_{0,r}(\rho)\longrightarrow N^B_{0,1}(\rho)
\]
form cyclic covers of degree $r$.  The same holds for Type C:
\[
 N^C_{0,r}(\rho)\longrightarrow N^C_{0,1}(\rho).
\]
On interiors the covers are finite \'etale.  Their central fibres have $r$
normalization components and Euler number $2r$.
\end{corollary}

\begin{proof}
For both families the normalized twist is independent of $r$, while the
cusp matrix has Smith form $\operatorname{diag}(1,r)$.  Apply
Corollary~\ref{thm:finite-index-reduction} to
\eqref{eq:typeB-cusp-decomposition} and
\eqref{eq:typeC-cusp-decomposition}.
\end{proof}

\section{Global compactification}\label{sec:global}

We glue the smooth torus family to the three local models constructed in
Section~\ref{sec:local-completions}: two affine logarithmic transforms
and one periodic toroidal completion.  After writing the three boundary
identifications in a common notation, we prove a relative gluing lemma and
apply it uniformly to the Type B and Type C families.

\subsection{The open family and the three collar maps}

We use \(\kappa\in\{B,C\}\) to denote the two realized monodromy types.
For Type B the base is \(B_B=B\), the marked points are
\(p_{B,i}=p_i\), and the finite multiplicities are \((4,6)\).  For Type C
the base is \(B_C\), the marked points are \(p_{C,i}=p_i^C\), and the
finite multiplicities are \((3,6)\).  Put
\[
 B_\kappa^\circ
 =B_\kappa\setminus\{p_{\kappa,0},p_{\kappa,1},p_{\kappa,2}\}.
\]
Fix an admissible value of the affine period parameter and suppress it
from the notation.  Let
\[
 J_r^\kappa\longrightarrow B_\kappa^\circ
\]
be the corresponding smooth family of complex two-tori.

\begin{lemma}[Properness of the open torus family]
\label{lem:open-family-proper}
The map \(J_r^\kappa\to B_\kappa^\circ\) is a holomorphic submersion with
connected fibres and is proper over every compact subset of
\(B_\kappa^\circ\).
\end{lemma}

\begin{proof}
On a simply connected coordinate disc \(U\Subset B_\kappa^\circ\), the
local system is marked and the family is
\[
 (U\times\mathbb C^2)/\Pi_r^\kappa(z)\mathbb Z^4.
\]
For a compact set \(K\subset U\), the image of
\[
 K\times[0,1]^4
 \longrightarrow U\times\mathbb C^2,
 \qquad (z,t)\longmapsto\bigl(z,\Pi_r^\kappa(z)t\bigr),
\]
is compact and meets every lattice orbit over \(K\).  It therefore maps
onto the inverse image of \(K\) in the quotient.  Hence the family is
proper over \(K\).  A finite cover of an arbitrary compact subset by such
coordinate discs proves the assertion.  Smoothness and connectedness of
the fibres follow from the period construction.
\end{proof}

Choose pairwise disjoint coordinate discs
\[
 D_{\kappa,i}^-\Subset D_{\kappa,i}^+\subset B_\kappa,
 \qquad i=0,1,2,
\]
centred at the three marked points, and write
\[
 A_{\kappa,i}
 =D_{\kappa,i}^+\setminus\overline{D_{\kappa,i}^-}.
\]
Let \(N_{\kappa,0,r}\to D_{\kappa,0}^+\) be the toroidal completion of
Section~\ref{sec:cusp}, and for \(j=1,2\) let
\(N_{\kappa,j,r}\to D_{\kappa,j}^+\) be the logarithmic transform of
Section~\ref{sec:finite}.  The gluing maps are oriented from the open
family to the local models.

At a finite point, pass to the root coordinate \(t=s^{m_{\kappa,j}}\)
and choose
\(\ell(s)=\log(s)/(2\pi i)\).  Lemma~\ref{lem:finite-collar} gives
\begin{equation}\label{eq:global-finite-collar}
 \psi_{\kappa,j}:
 J_r^\kappa|_{A_{\kappa,j}}
 \xrightarrow{\ \sim\ }
 N_{\kappa,j,r}|_{A_{\kappa,j}},
 \qquad
 [s,u]\longmapsto
 \bigl[s,u-\ell(s)\Pi_r^\kappa(z)v_{\kappa,j,r}\bigr].
\end{equation}
The expression is written on the root cover; the equivariance proved in
Lemma~\ref{lem:finite-collar} makes it a biholomorphism over the coarse
annulus and removes the dependence on the logarithm branch.

At the toroidal point write
\[
 Z_r^\kappa=\sigma_\kappa B_{0,r}^\kappa+C_r^\kappa,
 \qquad
 t_c=e^{2\pi i\sigma_\kappa}.
\]
On the punctured collar define
\begin{equation}\label{eq:global-cusp-collar}
 \psi_{\kappa,0}:
 J_r^\kappa|_{A_{\kappa,0}}
 \xrightarrow{\ \sim\ }
 N_{\kappa,0,r}|_{A_{\kappa,0}},
 \qquad
 [t_c,u]\longmapsto
 [t_c,\exp(2\pi i u)].
\end{equation}
Indeed, for \(\lambda,\mu\in\mathbb Z^2\),
\[
 \exp\bigl(2\pi i(u+Z_r^\kappa\lambda+\mu)\bigr)
 =e^{2\pi iC_r^\kappa\lambda}
  t_c^{B_{0,r}^\kappa\lambda}\exp(2\pi i u),
\]
which is precisely the toric deck action of
Theorem~\ref{thm:master-cusp}.  Thus exponentiation identifies the
additive period quotient with the multiplicative toroidal quotient.  A
local logarithm gives the inverse map.

The three maps are summarized by the fibrewise diagram
\begin{equation}\label{eq:global-gluing-diagram}
\begin{CD}
 J_r^\kappa|_{A_{\kappa,0}\sqcup A_{\kappa,1}\sqcup A_{\kappa,2}}
 @>{\psi_{\kappa,0}\sqcup\psi_{\kappa,1}\sqcup\psi_{\kappa,2}}>>
 \displaystyle\bigsqcup_{i=0}^2
 N_{\kappa,i,r}|_{A_{\kappa,i}}\\
 @VVV @VVV\\
 A_{\kappa,0}\sqcup A_{\kappa,1}\sqcup A_{\kappa,2}
 @= A_{\kappa,0}\sqcup A_{\kappa,1}\sqcup A_{\kappa,2}.
\end{CD}
\end{equation}
The annuli are disjoint, so there are no triple-overlap conditions.

\subsection{Analytic gluing}

\begin{lemma}[Relative analytic gluing along annuli]
\label{lem:analytic-gluing}
Let \(B\) be a compact Riemann surface and \(S\) a complex space.  For
pairwise disjoint coordinate discs choose
\[
 D_i^-\Subset D_i^0\Subset D_i^+,
 \qquad
 A_i=D_i^+\setminus\overline{D_i^-}.
\]
Suppose
\[
 X^\circ\longrightarrow
 \left(B\setminus\{p_0,\ldots,p_n\}\right)\times S
\]
is holomorphic and proper over compact subsets, and suppose
\(X_i\to D_i^+\times S\) is proper.  If there are biholomorphisms over
\(A_i\times S\),
\[
 \psi_i:X^\circ|_{A_i\times S}
 \xrightarrow{\ \sim\ }X_i|_{A_i\times S},
\]
then gluing
\[
 X^\circ|_{(B\setminus\cup_i\overline{D_i^-})\times S}
 \quad\text{to}\quad
 \bigsqcup_i X_i
\]
along the \(\psi_i\) produces a Hausdorff complex space \(X\), and the
induced map \(X\to B\times S\) is proper.  If the pieces are complex
manifolds, then \(X\) is a complex manifold.  If the fibre of every
piece is connected, then the fibres of \(X\to B\times S\) are connected.
\end{lemma}

\begin{proof}
Let \(X_0\) denote the restriction of \(X^\circ\) to
\((B\setminus\cup_i\overline{D_i^-})\times S\), and consider the
disjoint union \(X_0\sqcup\bigsqcup_iX_i\).  The equivalence relation is
the union of the diagonal and the graphs of \(\psi_i\) and
\(\psi_i^{-1}\).  Each graph is closed.  Indeed, if a sequence in the
graph converges in both pieces, then its base point converges to a point
belonging to both base domains, hence to \(A_i\times S\), where
closedness follows from continuity.  Since the annuli are disjoint,
there are no chains involving two different local pieces.

The quotient map is open: the saturation of an open set is obtained by
adding its images and inverse images under the finitely many
biholomorphisms \(\psi_i\).  An open quotient by a closed equivalence
relation is Hausdorff.  The complex atlases on the pieces agree on the
overlaps, so they define a complex-space atlas on the quotient; if the
pieces are manifolds, all charts are manifold charts.

To prove properness, let \(K\subset B\times S\) be compact.  The sets
\[
 K_0=K\cap\left((B\setminus\cup_iD_i^0)\times S\right),
 \qquad
 K_i=K\cap\left(\overline{D_i^0}\times S\right)
\]
are compact, and \(K_0\) lies in the punctured base.  The inverse images
of \(K_0\) in \(X^\circ\) and of \(K_i\) in \(X_i\) are compact.  Their
images in the quotient cover the inverse image of \(K\), which is
therefore compact.  The assertion about connected fibres follows by
examining separately a special point, an annulus, and a point in the
complement of the larger discs.
\end{proof}

\subsection{The compact threefolds}

\begin{theorem}[Global compactification of the two period families]

For every \(r\ge1\) and every admissible affine period parameter $c$, the
maps in \eqref{eq:global-gluing-diagram} define compact connected complex
threefolds
\[
 f_{r,c}^B:Y_{r,c}\longrightarrow B=\mathbb P^1,
 \qquad
 f_{r,c}^C:Z_{r,c}\longrightarrow B_C=\mathbb P^1.
\]
When $c$ is fixed, we abbreviate these spaces to $Y_r$ and $Z_r$.  They
have connected fibres.  Over the complement of the three marked points
they are proper holomorphic submersions by complex two-tori.  The finite
fibres have multiplicities \((4,6)\) for \(Y_r\) and \((3,6)\) for
\(Z_r\), and their reductions are smooth bielliptic surfaces.  The
remaining fibre is reduced and has normal crossings; its normalization is a
disjoint union of \(r\) copies of \(dP_6\), and its Euler number is
\(2r\).
\end{theorem}

\begin{proof}
Fix \(\kappa\in\{B,C\}\).  Lemma~\ref{lem:open-family-proper} gives the
required properness of \(J_r^\kappa\) over compact subsets of the
punctured base.  The finite local models are smooth and proper by
Lemma~\ref{lem:log-transform}, and the toroidal model is smooth and
proper by Theorem~\ref{thm:master-cusp}.  Equations
\eqref{eq:global-finite-collar} and \eqref{eq:global-cusp-collar} give
the three fibre-preserving biholomorphisms on the annuli.  Applying
Lemma~\ref{lem:analytic-gluing} with \(S\) a point produces a Hausdorff
complex manifold proper over \(B_\kappa\).  Since the base is compact,
the total space is compact.

The description of the finite fibres is
Proposition~\ref{prop:bielliptic-fibres}.  The toroidal fibre and its
finite-index covers are described by
Corollary~\ref{cor:typeB-cusp-ladder}.  All smooth and local fibres are
connected, so the glued fibres and the total space are connected.
\end{proof}

Shrinking the coordinate discs or changing the intermediate radii does
not change the compactification: two choices admit a common refinement,
and the identity on the common open and local pieces gives a canonical
biholomorphism of the resulting quotients.  The affine parameters in the
second period coordinates therefore survive as parameters of the compact
threefolds; their deformation-theoretic interpretation is given in
Section~\ref{sec:affine-moduli}.

\section{Algebraic reduction and deformations}\label{sec:affine-moduli}

The invariant alternating class determines the algebraic reduction at
every parameter.  The remaining task is to decide which integral
re-markings extend across the local completions.  We use the notation
\[
 \begin{array}{c|c|c|c|c}
 \kappa&s_\kappa&(m_{\kappa,1},m_{\kappa,2})
 &X^\kappa_{r,c}&B_\kappa\\ \hline
 B&2&(4,6)&Y_{r,c}&B\\
 C&3&(3,6)&Z_{r,c}&B_C.
 \end{array}
\]
Here $\beta_{B,0}=\beta_{\mathrm{part}}$ and
$\beta_{C,0}=\beta_{C,\mathrm{part}}$, and the marked period matrix over a
fixed smooth point of the corresponding orbifold base is
\begin{equation}\label{eq:parameter-period-matrix}
 \Pi^\kappa_{r,c}
 =\begin{pmatrix}
 s_\kappa r\mu_\kappa&\tau_\kappa&1&0\\
 r(\beta_{\kappa,0}+c)&\mu_\kappa&0&1
 \end{pmatrix}.
\end{equation}
We use the particular solutions and parameter half-planes fixed in
\eqref{eq:admissible-domain}.

\begin{proposition}[Simultaneous compact deformation]\label{prop:simultaneous-parameter-family}
For every $r\geq1$ and $\kappa\in\{B,C\}$, the constructions of the
preceding sections glue holomorphically in $c$ to a proper holomorphic
submersion
\[
 \varpi_{\kappa,r}:\mathcal X_{\kappa,r}\longrightarrow\mathcal U_\kappa
\]
whose fibre at $c$ is $X^\kappa_{r,c}$.  There is simultaneously a
holomorphic map
\[
 \mathcal F_{\kappa,r}:\mathcal X_{\kappa,r}
 \longrightarrow B_\kappa\times\mathcal U_\kappa
\]
restricting to the torus fibration $f^\kappa_{r,c}$ on every parameter
fibre.  In particular all $X^\kappa_{r,c}$ with $c\in\mathcal U_\kappa$
are oriented-diffeomorphic.
\end{proposition}

\begin{proof}
The period matrices \eqref{eq:parameter-period-matrix} are holomorphic in
$(z,c)$, and the defining nondegeneracy scalar is shifted by
$\operatorname{Im}c$.  Hence the open torus quotient is a holomorphic
family over $B_\kappa^\circ\times\mathcal U_\kappa$.  The finite affine
actions and their translation vectors are independent of $c$; only the
complex structures of the ambient tori vary holomorphically.  At the
cusp, adding $c$ changes the holomorphic twist matrix $C_r(t_c)$ by a
constant matrix, so the standard toric action remains holomorphic in $(t_c,c)$.
The freeness and tropical properness estimates are uniform on compact
subsets of $\mathcal U_\kappa$.

Cover $\mathcal U_\kappa$ by relatively compact parameter discs
$S_a\Subset\mathcal U_\kappa$.  Theorem~\ref{thm:master-cusp}
provides a relative cusp model on each $S_a$, and Lemma~\ref{lem:analytic-gluing}
constructs a proper holomorphic family over $S_a$.  On an overlap
$S_a\cap S_b$ both families are obtained from the same period matrices,
finite affine actions, toric deck action, and collar formulae.  Their
restrictions are therefore canonically identical after shrinking the
local cusp discs to their common radius.  These canonical identifications
satisfy the cocycle condition, so the local families descend to a
holomorphic family over all of $\mathcal U_\kappa$.  Properness and the
submersion property are local on the target.  Hence
$\varpi_{\kappa,r}$ is a proper holomorphic submersion, and Ehresmann's
theorem gives the final assertion.
\end{proof}

\subsection{The invariant class and algebraic reduction}

Put $\lambda_{\kappa,r}=s_\kappa r$, and let
$\gamma^*,u^*,w^*,\delta^*$ be the basis dual to the marked homology
basis.  The primitive integral alternating class relevant to the total
space is
\begin{equation}\label{eq:distinguished-NS-class}
 \widetilde q_{\kappa,r}
 =\lambda_{\kappa,r}\,\gamma^*\wedge\delta^*
 +u^*\wedge w^*
 \in\bigwedge^2\Lambda^*.
\end{equation}

\begin{lemma}[The monodromy-invariant alternating class]
\label{lem:invariant-class-all-parameters}
For every $r\geq1$, $\kappa\in\{B,C\}$, and smooth fibre $F$, the
simultaneous invariant groups are
\begin{equation}\label{eq:joint-H2-invariants}
 H^2(F;\Z)^{\langle T_1,T_2\rangle}
 =\Z\widetilde q_{\kappa,r},\qquad
 H^2(F;\R)^{\langle T_1,T_2\rangle}
 =\R\widetilde q_{\kappa,r}.
\end{equation}
At every admissible parameter $c$ and every smooth base point, this
class is of type $(1,1)$ and its Appell--Humbert form has signature
$(1,1)$.
\end{lemma}

\begin{proof}
Use the contragredient deck matrices $T_j=(A_j^{-1})^{\mathsf T}$.
By \eqref{eq:clockwise-forward-period}, the forward cohomology transports
are $T_j^{-1}$ and have the same fixed groups.
The common order-six matrix fixes in $H^2(F;\Z)$ precisely
\[
 \Z(\gamma^*\wedge\delta^*)\oplus\Z(u^*\wedge w^*).
\]
Substitute $a\gamma^*\wedge\delta^*+b u^*\wedge w^*$ into the first
monodromy equation.  The explicit matrices give $a=s_\kappa r b$
(also see the coefficient equations in
Appendix~\ref{app:finite-quotient-h2}).  This proves the integral
assertion; the coefficient of $u^*\wedge w^*$ is one, so the displayed
generator is primitive.  Solving the same equations over $\R$ gives
the real assertion.

At any smooth base point abbreviate the period coordinates by
$\tau,\mu,\beta=\beta_{\kappa,0}+c$, and write
\[
 \Pi=[Z\mid I_2],\qquad
 Z=\begin{pmatrix}s_\kappa r\mu&\tau\\r\beta&\mu\end{pmatrix},
 \qquad K=\begin{pmatrix}I_2\\-Z\end{pmatrix}.
\]
If $\widetilde Q$ is the skew matrix of
$\widetilde q_{\kappa,r}$, direct multiplication gives
$K^{\mathsf T}\widetilde QK=0$, the type-$(1,1)$ condition.  Moreover,
\begin{equation}\label{eq:all-parameter-AH-matrix}
 G:=i\Pi\widetilde Q^{-1}\overline\Pi^{\mathsf T}=2M,
 \qquad
 M=\begin{pmatrix}
 \operatorname{Im}\tau&\operatorname{Im}\mu\\
 \operatorname{Im}\mu&\dfrac1{s_\kappa}\operatorname{Im}\beta
 \end{pmatrix}.
\end{equation}
With the Appell--Humbert convention
$E(\lambda,\mu)=\operatorname{Im}H(\lambda,\mu)$ of
\cite[Chapter~2]{BL}, the Hermitian matrix of this class is
$H^{\mathrm{AH}}=2G^{-1}=M^{-1}$.  Admissibility gives
\[
 \det M=
 \frac{\operatorname{Im}\tau}{s_\kappa}
 \left(\operatorname{Im}\beta
 -s_\kappa\frac{(\operatorname{Im}\mu)^2}{\operatorname{Im}\tau}\right)<0.
\]
Thus $H^{\mathrm{AH}}$ has signature $(1,1)$.  Neither calculation
requires a condition on the full N\'eron--Severi group of the fibre.
\end{proof}

\begin{lemma}[Effective classes on a complex two-torus]
\label{lem:effective-torus-class}
On a compact complex two-torus, the class of a nonzero effective divisor
has a nonzero translation-invariant semipositive $(1,1)$ representative.
Consequently neither the zero class nor a class with a nondegenerate
indefinite Appell--Humbert form can be the class of such a divisor.
\end{lemma}

\begin{proof}
For a nonzero effective divisor $D$, average its positive closed
integration current over the compact group of translations, using
normalized Haar measure.  The result is translation invariant, hence
is a smooth constant-coefficient semipositive $(1,1)$ form.  Translations
are homotopic to the identity, so averaging preserves its de Rham class.
For any translation-invariant K\"ahler form $\omega$ on the torus, its
pairing with $\omega$ equals $\int_D\omega>0$, so the averaged form is
nonzero.  Translation-invariant forms give the usual exterior-algebra
model for the real cohomology of a torus; in particular, an invariant
exact form is zero and the invariant representative of a class is unique.
The zero class and a class with indefinite Hermitian representative
therefore cannot occur.
\end{proof}

\begin{theorem}[Algebraic reduction and non-K\"ahlerness for every parameter]
\label{thm:algebraic-dimension-one}
For every $r\geq1$, $\kappa\in\{B,C\}$, and $c\in\mathcal U_\kappa$,
all divisors on $X^\kappa_{r,c}$ are vertical,
\[
 a(X^\kappa_{r,c})=1,
\]
and $f^\kappa_{r,c}:X^\kappa_{r,c}\to B_\kappa$ is its algebraic
reduction in the sense of \cite[Chapter~3]{Ueno}.  Every member is
non-K\"ahler.  Every biholomorphism between two members of the same
family preserves the torus fibration and induces the identity on the
base.
\end{theorem}

\begin{proof}
The restriction of a global cohomology class to the smooth fibres is a
flat section of the corresponding cohomology local system: along a
path, the fibre inclusions are homotopic through the total space.
In particular, for a line bundle $L$ on $X^\kappa_{r,c}$,
Lemma~\ref{lem:invariant-class-all-parameters} gives
\[
 c_1(L|_{F_t})=k\widetilde q_{\kappa,r}\quad\text{for some }k\in\Z.
\]
Suppose an irreducible effective divisor dominates $B_\kappa$.
Its restriction to a general smooth fibre is a nonzero effective divisor.
If $k\neq0$, its invariant class is nondegenerate and indefinite;
if $k=0$, its class is zero.  Both possibilities contradict
Lemma~\ref{lem:effective-torus-class}.  Thus every irreducible divisor
is vertical.

The divisor of a nonzero meromorphic function is therefore vertical.
Its restriction to a general smooth fibre has neither zeros nor poles;
normality removes possible codimension-two indeterminacy, so the
restriction is holomorphic and hence constant.  Let
$\Gamma_g\subset X^\kappa_{r,c}\times\mathbb P^1$ be the closure of its
graph.  The proper map $f^\kappa_{r,c}\times\mathrm{id}$ sends
$\Gamma_g$ to an irreducible analytic curve in
$B_\kappa\times\mathbb P^1$ whose projection to $B_\kappa$ has generic
degree one.  By Remmert's theorem and normalization of this curve, it
is the graph of a meromorphic function $h$ on $B_\kappa$, and
$g=(f^\kappa_{r,c})^*h$.  Conversely, pullback from the projective base
supplies nonconstant meromorphic functions.  Hence the meromorphic
function field is exactly $\C(B_\kappa)$.

If the total space had a K\"ahler form, its restriction to a smooth
fibre would have a positive definite translation-invariant representative,
obtained by averaging.  Its class must belong to
$\R\widetilde q_{\kappa,r}$ by the same flatness argument.
This real line contains no positive definite class, by
Lemma~\ref{lem:invariant-class-all-parameters}.  This proves
non-K\"ahlerness without any use of the homology calculation below.

Finally, a biholomorphism induces an automorphism of the meromorphic
function field and hence of $B_\kappa$.  The cusp is the unique
normal-crossings fibre, and the other special fibres have distinct
multiplicities, $(4,6)$ or $(3,6)$.  The base automorphism fixes all
three special values and is therefore the identity.
\end{proof}

\begin{corollary}[Exclusion from Fujiki's class]\label{cor:non-Fujiki}
Every $X^\kappa_{r,c}$ lies outside Fujiki's class $\mathcal C$.
\end{corollary}

\begin{proof}
A compact complex manifold in Fujiki's class admits a K\"ahler current;
by regularization it admits one with analytic singularities
\cite[Theorems~3.2 and~3.4]{DemaillyPaun}.  Suppose such a current $T$ existed on
$X=X^\kappa_{r,c}$.  Choose a smooth fibre $F_t$ not contained in its
singular locus.  The local potentials of $T$ restrict to $F_t$, so
$T|_{F_t}$ is a closed strictly positive $(1,1)$ current representing
$[T]|_{F_t}$.  Averaging over the compact translation group of $F_t$
gives a positive definite invariant form in this class.  For example,
if $T\geq\epsilon\omega_X$ for a Hermitian form $\omega_X$, its
restriction and average dominate a positive multiple of an invariant
K\"ahler form on $F_t$.

The global class $[T]$ restricts to a monodromy-invariant class,
so $[T]|_{F_t}\in\R\widetilde q_{\kappa,r}$ by
Lemma~\ref{lem:invariant-class-all-parameters}.  That line contains no
positive definite class, a contradiction.
\end{proof}

\subsection{Translations and the arithmetic period}

Let $E_{41}$ denote the matrix with a single $1$ in row four and column
one.  A \emph{homology marking} fixes the base, the identification of
the homology local system with the one constructed above, and the
specified local completion data; it does not specify fibrewise origins.
A \emph{strict marking} also includes the zero section on the smooth
locus.  A homology-marked isomorphism induces the identity on this local
system; a strictly marked isomorphism must in addition preserve the
specified zero section.  Thus homology-marked maps may be translations,
whereas strictly marked maps preserve zero.
An \emph{isomorphism over the base} may have any compatible flat integral
linear part and a fibrewise translation.  An \emph{unmarked
biholomorphism} is an arbitrary biholomorphism of total spaces.
Theorem~\ref{thm:algebraic-dimension-one} makes every unmarked
biholomorphism between members of one family an isomorphism over the
identity of the base.

\subsubsection{Affine automorphisms on the smooth locus}

\begin{lemma}[Monodromy centralizer]\label{lem:monodromy-centralizer}
For either realized family and every $r\geq1$,
\[
 \operatorname{Cent}_{GL_4(\Z)}(A^\kappa_{1,r},A^\kappa_{2,r})
 =\{\varepsilon I+nE_{41}:\varepsilon\in\{\pm1\},\ n\in\Z\}.
\]
\end{lemma}

\begin{proof}
The common order-six matrix forces a commuting matrix to have the form
\[
 \begin{pmatrix}
 a&0&0&b\\
 0&p+q&-p&0\\
 0&p&q&0\\
 e&0&0&f
 \end{pmatrix}.
\]
Commutation with the first Type B or Type C matrix gives
$b=p=0$ and $q=f=a$, while $e$ remains arbitrary.  Thus the rational
centralizer is $aI+eE_{41}$.  Integrality and invertibility force
$a=\pm1$ and $e\in\Z$.
\end{proof}

\begin{lemma}[Fibrewise biholomorphisms are affine]
\label{lem:fibrewise-affine-form}
Let $\Phi:X^\kappa_{r,c}\to X^\kappa_{r,c'}$ be a biholomorphism over the
identity of the base.  On a marked simply connected part of the smooth
locus it has the unique affine form
\[
 \Phi_z([u])=[R(z)u]+t(z),\qquad
 \Pi^\kappa_{r,c'}(z)L=R(z)\Pi^\kappa_{r,c}(z),
\]
where $L\in GL_4(\mathbb Z)$ is the flat integral linear part and $t$
is a holomorphic section of the target torus bundle.  On the universal
cover $L$ is constant and commutes with the two monodromies.
At a finite special point a necessary compatibility condition is
\[
 Lv_j-v_j\in N_{A_j}\Lambda.
\]
This condition concerns the central affine actions; extension of a
punctured map also requires holomorphic extension in the root-disc model.
\end{lemma}

\begin{proof}
A biholomorphism between compact complex tori is affine.  Its action on
integral first homology is locally constant; analytic continuation
therefore gives a flat integral map commuting with monodromy.  Its
complex-linear realization is $R(z)$, determined by the period identity,
and the image of zero is the translation section $t(z)$.

At a finite special point, base change by $t=s^{m_j}$ followed by
normalization recovers the root-disc torus family used in the
construction.  A global $\Phi$ induces an equivariant isomorphism of
these normalized base changes.  On their central tori, its translation
part conjugates the affine actions with vectors $Lv_j$ and $v_j$.
Lemma~\ref{lem:affine-cyclic-conjugacy} gives the displayed necessary
condition.  The root-disc extensions used below are constructed in
Lemma~\ref{lem:family-root-conjugacy}.
\end{proof}

\begin{remark}[A nonextendible punctured translation]
In the common order-six root model, $\Pi(s)\delta=e_2$ and
$A_2\delta=\delta$.  The punctured fibre translation
$[u]\mapsto[u+s^{-6}e_2]$ is equivariant under $s\mapsto\zeta_6^{-1}s$;
its linear part is the identity and its affine-vector defect is zero.
Nevertheless, at $s_n=n^{-1/6}$ and $s'_n=(n+1/2)^{-1/6}$ its values on
the root zero section approach respectively the classes of $0$ and
$\delta/2$.  These remain distinct in the central affine quotient:
the orbit of $0$ consists of the classes of $-k\gamma/6$,
$0\leq k<6$.  Hence the descended punctured map has no continuous extension.
\end{remark}

\begin{lemma}[Conjugacy of finite affine cyclic actions]
\label{lem:affine-cyclic-conjugacy}
Let $A^m=I$ on a lattice $\Lambda$, let
$N_A=I+A+\cdots+A^{m-1}$, and let $v,v'\in\Lambda^A$.  The two affine
actions
\[
 x\longmapsto Ax+\frac vm,
 \qquad
 x\longmapsto Ax+\frac{v'}m
\]
on the torus are conjugate by a translation if and only if
\[
 v'-v\in N_A\Lambda.
\]
\end{lemma}

\begin{proof}
If translation by $t$ conjugates the first action to the second, then
\[
 \frac{v'-v}{m}=(I-A)t+\lambda
\]
for some $\lambda\in\Lambda$.  Applying $N_A$ gives
$v'-v=N_A\lambda$.  Conversely, if $v'-v=N_A\lambda$, then
$N_A(N_A\lambda/m-\lambda)=0$.  Over $\Lambda\otimes\mathbb R$ one has
$\ker N_A=\operatorname{im}(I-A)$, so
$N_A\lambda/m-\lambda=(I-A)t$ for some $t$, which gives the conjugacy.
\end{proof}

For the two families the norm matrices are
\[
 N^B_1=\begin{pmatrix}
 4&0&0&0\\-4r&0&0&0\\-4r&0&0&0\\0&2&-2&4
 \end{pmatrix},
 \qquad
 N^C_1=\begin{pmatrix}
 3&0&0&0\\-3r&0&0&0\\-3r&0&0&0\\0&1&-1&3
 \end{pmatrix},
\]
while the common order-six norm is
\[
 N_2=\operatorname{diag}(6,0,0,6).
\]

\subsubsection{Translations extending across the special fibres}

Let $\mathbb L_{\kappa,r}$ be the homology local system on
$B_\kappa^\circ$ with fibre $\Lambda$ and clockwise forward transports
$\mathsf P_j=A_j^{-1}$, and let
$j:B_\kappa^\circ\hookrightarrow B_\kappa$.  Introduce the root stack \cite{Cadman}
\begin{equation}
 \mathfrak B_\kappa=
 \sqrt[m_{\kappa,1}]{(B_\kappa,p_{\kappa,1})}
 \times_{B_\kappa}
 \sqrt[m_{\kappa,2}]{(B_\kappa,p_{\kappa,2})},
\end{equation}
with coarse map $\mathfrak c_\kappa:\mathfrak B_\kappa\to B_\kappa$.
On the root stack, let $\mathcal V^{\mathrm{orb}}_\kappa$ be the rank-two
bundle of vertical translation fields.  Across the toroidal point we use
its logarithmic extension, generated in a toric chart by
$x_1\partial_{x_1}$ and $x_2\partial_{x_2}$.  Put
\[
 \mathcal V_\kappa=(\mathfrak c_\kappa)_*
 \mathcal V^{\mathrm{orb}}_\kappa.
\]
Let $\mathcal T^{0,\mathrm{comp}}_{\kappa,r}$ be the identity component
of the sheaf of fibre translations extending across all three local
completions, and let $\mathcal T^{\mathrm{comp}}_{\kappa,r}$ be the full
sheaf.

\begin{lemma}[Valuation of a toroidal translation]
\label{lem:toroidal-translation-valuation}
Let $h$ be a fibre translation on the punctured toroidal model, and
assume that both $h$ and $h^{-1}$ extend holomorphically across the
toroidal completion.  After choosing a lift to the periodic toric cover,
it has a unique form in dense-torus coordinates
\[
 h(t_c)=t_c^\nu u(t_c),
 \qquad \nu\in\mathbb Z^2,
\]
where $u:\Delta\to(\mathbb C^*)^2$ is holomorphic.  Changing the lift of
$h$ by the deck transformation indexed by $\eta\in\mathbb Z^2$ changes
$\nu$ by $B^\kappa_{0,r}\eta$.
\end{lemma}

\begin{proof}
Write $p:Y_{<\epsilon}\to N^\kappa_{0,r}$ for the periodic covering.
By Lemma~\ref{lem:toric-simply-connected}, its source is simply
connected, so the extended automorphism composed with $p$ has a lift
$\widehat h:Y_{<\epsilon}\to Y_{<\epsilon}$.  Covering lifts are
holomorphic because $p$ is a local biholomorphism.  Equivalently, its
value on the extended zero section is obtained by lifting the image
section over the simply connected disc, after choosing one point of
the lift.  Toric characters can then be pulled back along this lifted section.

Let $\widetilde e:\Delta\to Y_{<\epsilon}$ be the resulting lifted
section.  It is in the dense torus over $t_c\neq0$.  Each rational
toric character $\chi^m$ on the cover consequently pulls back to a
nonzero meromorphic function germ at $0$.  Its order is additive in
$m\in\operatorname{Hom}((\C^*)^2,\C^*)$, and thus equals
$\langle m,\nu\rangle$ for a unique $\nu\in\Z^2$.  Dividing by
$t_c^{\langle m,\nu\rangle}$ leaves a holomorphic unit.  For a basis of
the character lattice these units define
$u:\Delta\to(\C^*)^2$.

Over a punctured fibre, a lift of a torus translation to its
$(\C^*)^2$ cover is multiplication by the lifted value at zero.
Hence $\widehat h$ has the asserted multiplier $t_c^\nu u(t_c)$.
Conversely, this multiplier and its inverse extend on the periodic
model: $t_c^\nu$ shifts the fan by $\nu$ and the unit factor is a
holomorphic torus action.  They commute with the deck action and
therefore descend.  Finally, changing the lift by the deck element
indexed by $\eta$ multiplies the multiplier by
$t_c^{B^\kappa_{0,r}\eta}$ times a holomorphic unit.  Its valuation
changes by $B^\kappa_{0,r}\eta$, as claimed.
\end{proof}

\begin{proposition}[Translations across the local completions]

There are exact sequences
\begin{equation}\label{eq:translation-vector-bundle}
 0\longrightarrow\mathcal O_{B_\kappa}
 \longrightarrow\mathcal V_\kappa
 \longrightarrow\mathcal O_{B_\kappa}(-1)
 \longrightarrow0,
\end{equation}
\begin{equation}\label{eq:completion-translation-exponential}
 0\longrightarrow j_*\mathbb L_{\kappa,r}
 \longrightarrow\mathcal V_\kappa
 \xrightarrow{\exp}
 \mathcal T^{0,\mathrm{comp}}_{\kappa,r}
 \longrightarrow0,
\end{equation}
and
\begin{equation}\label{eq:completion-translation-components}
 0\longrightarrow\mathcal T^{0,\mathrm{comp}}_{\kappa,r}
 \longrightarrow\mathcal T^{\mathrm{comp}}_{\kappa,r}
 \longrightarrow(i_{p_{\kappa,0}})_*(\mathbb Z/r)
 \longrightarrow0.
\end{equation}
The first sequence splits, and therefore
\begin{equation}\label{eq:H1-translation-vector-bundle}
 H^1(B_\kappa,\mathcal V_\kappa)=0.
\end{equation}
At the three types of points the kernels and component groups are
\[
\begin{array}{c|c|c}
 q&(j_*\mathbb L_{\kappa,r})_q&
 \mathcal T^{\mathrm{comp}}_{\kappa,r,q}/
 \mathcal T^{0,\mathrm{comp}}_{\kappa,r,q}\\ \hline
 q\in B_\kappa^\circ&\Lambda&0\\
 q=p_{\kappa,j},\ j=1,2&\Lambda^{A_j}&0\\
 q=p_{\kappa,0}&\Lambda_{\mathrm{van}}&
 \Lambda_{\mathrm{van}}/\operatorname{im}(\mathsf M^\kappa_{0,r}-I)
 \simeq\mathbb Z/r.
\end{array}
\]
\end{proposition}

\begin{proof}
The automorphy matrices $R_j^\kappa$ are lower triangular and fix the
second coordinate vector.  Hence $\mathcal V^{\mathrm{orb}}_\kappa$ has
a trivial line subbundle, and its quotient is the automorphic line from
Lemma~\ref{lem:homogeneous-bundle}.  Taking invariant coarse pushforward
is exact for the finite root stabilizers, so the quotient on $B_\kappa$
is $\mathcal O_{B_\kappa}(-1)$ and
\eqref{eq:translation-vector-bundle} follows.  Since
\[
 \operatorname{Ext}^1(\mathcal O(-1),\mathcal O)
 =H^1(B_\kappa,\mathcal O(1))=0,
\]
the sequence splits.  Together with
$H^1(\mathcal O)=H^1(\mathcal O(-1))=0$, this gives
\eqref{eq:H1-translation-vector-bundle}.

We verify the exponential sequence on stalks.  On the smooth locus it is
the usual sequence for a complex torus.  Near a finite point, work on the
root disc with coordinate $s$.  A vertical vector field is a holomorphic
map $a(s)\in\mathbb C^2$ satisfying the equivariance relation determined
by $R_j^\kappa$.  If its exponential is the identity on the torus, then
on the punctured disc
\[
 a(s)=\Pi_r^\kappa(s)\lambda
\]
for a locally constant $\lambda\in\Lambda$; connectedness makes
$\lambda$ constant and equivariance gives $A_j\lambda=\lambda$.  Thus the
kernel is $\Lambda^{A_j}=(j_*\mathbb L_{\kappa,r})_{p_{\kappa,j}}$.
Conversely, let $\tau(s)$ be an equivariant translation of the root-disc
model.  Since the root disc is simply connected, $\tau$ has a holomorphic
lift $a(s)$ to the vertical vector bundle.  The equivariance defect
\[
 a(\zeta_m^{-1}s)-R_j^\kappa(s)a(s)
\]
is a constant lattice-valued one-cocycle for $C_m$.  Its class lies in
$H^1(C_m,\Lambda)=\ker N_A/(A-I)\Lambda$.  For the order-four,
order-three, and order-six matrices one has, respectively,
\[
\begin{array}{c|c}
 A&\ker N_A=(A-I)\Lambda\\ \hline
 A^B_{1,r}&\mathbb Z\langle u+w,-2u+\delta\rangle\\
 A^C_{1,r}&\mathbb Z\langle u+w,-3u+\delta\rangle\\
 A_{2,r}&\mathbb Z\langle u,w\rangle.
\end{array}
\]
Thus this cyclic cohomology group vanishes in all three finite local
models.  After changing the lift by a lattice-valued zero-cochain, $a(s)$
is equivariant.  This proves surjectivity of the exponential map and also
shows that the local translation group has no extra components at either
root point.

At the cusp, logarithmic vertical fields exponentiate to the dense torus
$(\mathbb C^*)^2$ acting on every toric chart.  Their kernel is the
cocharacter lattice
\[
 \ker(\mathsf M^\kappa_{0,r}-I)=\Lambda_{\mathrm{van}}
 =(j_*\mathbb L_{\kappa,r})_{p_{\kappa,0}}.
\]
Lemma~\ref{lem:toroidal-translation-valuation} associates to every
completion-preserving translation a valuation class
$[\nu]\in\mathbb Z^2/B^\kappa_{0,r}\mathbb Z^2$.  The unit factor
$u(t_c)$ exponentiates a logarithmic vertical field and therefore belongs
to the identity component.  Conversely, periodicity of the fan makes the
monomial factor $t_c^\nu$ extend after shifting the normalization
component by $\nu$.  Thus the valuation gives
\[
 \pi_0\bigl(\mathcal T^{\mathrm{comp}}_{\kappa,r,p_{\kappa,0}}\bigr)
 \cong\mathbb Z^2/B^\kappa_{0,r}\mathbb Z^2
 \cong\Lambda_{\mathrm{van}}/
 \operatorname{im}(\mathsf M^\kappa_{0,r}-I)
 \cong\mathbb Z/r.
\]
This proves the last two exact sequences and the stalk descriptions.
\end{proof}

\subsubsection{The obstruction to patching local conjugacies}

\begin{lemma}[The integral collar-obstruction group]

There is a canonical identification
\begin{equation}\label{eq:translation-obstruction-group}
 H^2(B_\kappa,j_*\mathbb L_{\kappa,r})
 \cong
 \frac{\Lambda}{(A_1-I)\Lambda+(A_2-I)\Lambda}.
\end{equation}
For both realized families the quotient is infinite cyclic, detected by
the primitive invariant functional $\chi$.
\end{lemma}

\begin{proof}
Put $\mathcal Q=j_*\mathbb L/j_!\mathbb L$.  This is a sum of
skyscraper sheaves at the three special points, so
$H^a(B_\kappa,\mathcal Q)=0$ for every $a>0$.  The exact segment
\[
 H^1(B_\kappa,\mathcal Q)\longrightarrow
 H^2(B_\kappa,j_!\mathbb L)\longrightarrow
 H^2(B_\kappa,j_*\mathbb L)\longrightarrow
 H^2(B_\kappa,\mathcal Q)
\]
therefore has zero outer terms.  Since
$H^2(B_\kappa,j_!\mathbb L)=H_c^2(B_\kappa^\circ,\mathbb L)$, it gives
\[
 H^2(B_\kappa,j_*\mathbb L_{\kappa,r})
 \cong H_c^2(B_\kappa^\circ,\mathbb L_{\kappa,r}).
\]
Use the oriented pair-of-pants cellular model for compactly supported
cohomology \cite[Chapter~8]{KashiwaraSchapira}, with the clockwise paths
and forward transports $\mathsf P_1,\mathsf P_2$ fixed in
\eqref{eq:deck-forward-distinction}.  Its top differential, in oriented
cell trivializations, is
\begin{equation}\label{eq:translation-cellular-complex}
 \Lambda\oplus\Lambda\xrightarrow{\ d_c^1\ }\Lambda,
 \qquad d_c^1(a,b)=(\mathsf P_1-I)a+(\mathsf P_2-I)b.
\end{equation}
The integral source change $(a,b)=(-\mathsf D_1x,-\mathsf D_2y)$
turns this differential into $(A_1-I)x+(A_2-I)y$.  Equivalently,
the top cohomology is the common coinvariant lattice by
\eqref{eq:inverse-integral-image}.  This proves
\eqref{eq:translation-obstruction-group}.  The Smith normal
form of $[(A_1-I)\ (A_2-I)]$ is
$\operatorname{diag}(1,1,1,0)$ for both families.  The functional
$\chi$ annihilates the image and is primitive, so it identifies the
cokernel with $\mathbb Z$.
\end{proof}

\begin{proposition}[Patching completion-preserving translations]
\label{lem:torus-translation-globalization}
Let $X,X'$ be two completed fibrations of fixed type $\kappa$ and index
$r$ over $B_\kappa=\mathbb P^1$.  Write
$\mathcal T^0=\mathcal T^{0,\mathrm{comp}}_{\kappa,r}$ and
$\mathcal T=\mathcal T^{\mathrm{comp}}_{\kappa,r}$ for the translation
sheaves of the target $X'$.  Choose local isomorphisms
$\Phi_a:X|_{U_a}\to X'|_{U_a}$ on the pair of pants, the two root discs,
and the cusp disc, with the same flat integral linear part $L$.

Fix a reference cusp comparison $\Phi_0^{\mathrm{ref}}$ with linear part
$L$, and require the chosen cusp map $\Phi_0$ to induce the same
identification of normalization-component labels as
$\Phi_0^{\mathrm{ref}}$.  More precisely, the translation
$\Phi_0\circ(\Phi_0^{\mathrm{ref}})^{-1}$ of the target completion has
zero valuation class in
$\mathbb Z^2/B_{0,r}^\kappa\mathbb Z^2\simeq\mathbb Z/r$, using its
fixed toric marking and Lemma~\ref{lem:toroidal-translation-valuation}.
The reference comparison thus fixes the component identification that
zero-component corrections must preserve.  Every overlap mismatch
$\Phi_b\circ\Phi_a^{-1}$ is a fibre translation.
On the three clockwise oriented collars let
\[
 [\omega_i]\in\Lambda/(\mathsf P_i-I)\Lambda
\]
be the winding of the translation mismatch: continue a local logarithm
along the collar, return its endpoint by $\mathsf P_i$, and subtract
its initial value.  At the cusp use $\mathsf P_0$ from
\eqref{eq:clockwise-cusp-transport}.  All classes are transported to
the fixed reference fibre.  The local maps glue after correction by
sections of the full sheaf $\mathcal T$ if and only if
\begin{equation}\label{eq:collar-winding-obstruction}
 [\omega_0+\omega_1+\omega_2]=0
 \quad\text{in }H^2(B_\kappa,j_*\mathbb L_{\kappa,r}).
\end{equation}
When this condition holds, the corrections can be chosen in
$\mathcal T^0$, preserving the prescribed cusp component identification.
\end{proposition}

\begin{proof}
Let $\tau_{ab}$ be the target translation mismatches, in additive
notation.  Equality of the linear parts makes them translations, and
the overlap identities make them a cocycle.  All overlaps lie in the
smooth locus, where $\mathcal T^0=\mathcal T$, so this is a
$\mathcal T^0$-valued cocycle.  The reference cusp comparison fixes the
component identification to be preserved by $\mathcal T^0$ corrections.
On simply connected sectors refining each collar choose local
holomorphic logarithms.  Their integral jumps, after returning to the
initial marking, represent $[\omega_i]$; changing a lift by $\eta$
changes its representative by $(\mathsf P_i-I)\eta$.  The connecting
morphism for \eqref{eq:completion-translation-exponential} is the
integral two-cocycle of these logarithms.  Evaluation on the oriented
pair-of-pants two-cell is the sum of its boundary jumps, namely
\eqref{eq:collar-winding-obstruction}.  This fixes the sign of the
identification with the cellular quotient
\eqref{eq:translation-cellular-complex}.

The long exact sequence contains
\[
 H^1(B_\kappa,\mathcal V_\kappa)\longrightarrow
 H^1(B_\kappa,\mathcal T^0)\xrightarrow{\partial}
 H^2(B_\kappa,j_*\mathbb L_{\kappa,r}).
\]
Its left term is zero by \eqref{eq:H1-translation-vector-bundle}, so
$\partial$ is injective.  The obstruction vanishes exactly when
$[\tau]=0$.  A trivial translation torsor has a global section; written
in the original local trivializations, this gives sections
$\sigma_a\in\mathcal T^0(U_a)$ satisfying
\[
 \tau_{ab}=\sigma_b-\sigma_a.
\]
Composing the local maps with the corresponding translations removes
the mismatches.  These sections extend across the prescribed local
completions by the definition of $\mathcal T^0$, and analytic gluing
gives the global biholomorphism.  Local logarithms were used only on a
refinement to compute the connecting class; no single-valued logarithm
of $\sigma_a$ on an entire original piece is required.

It remains to check necessity when corrections in the full sheaf
$\mathcal T$ are allowed.  The injection $\partial$ above and
\eqref{eq:translation-obstruction-group} give
\[
 H^1(B_\kappa,\mathcal T^0)\hookrightarrow
 H^2(B_\kappa,j_*\mathbb L_{\kappa,r})\simeq\mathbb Z.
\]
Thus $H^1(B_\kappa,\mathcal T^0)$ is torsion-free.  The component
sequence \eqref{eq:completion-translation-components} gives the exact
segment
\[
 \mathbb Z/r\xrightarrow{\partial_{\mathrm{comp}}}
 H^1(B_\kappa,\mathcal T^0)\xrightarrow{\iota_*}
 H^1(B_\kappa,\mathcal T),
\]
where $\mathbb Z/r$ is the group of global sections of the
point-supported component sheaf.  The image of
$\partial_{\mathrm{comp}}$ is annihilated by $r$, hence is zero in the
torsion-free middle group.  Therefore $\iota_*$ is injective.  If
$\mathcal T$-valued corrections glue the local maps, then
$\iota_*[\tau]=0$, so $[\tau]=0$ and
\eqref{eq:collar-winding-obstruction} follows.  This proves necessity
for the full sheaf, while sufficiency was proved using $\mathcal T^0$.
The comparison uses the torsion-free obstruction group on
$B_\kappa=\mathbb P^1$; no such comparison is assumed for the arbitrary
base curves of Section~\ref{sec:extensions}.
\end{proof}

\subsubsection{The exact period}

\begin{lemma}[Root-disc conjugacies for a period shift]
\label{lem:family-root-conjugacy}
Put $n=6k$ and $M_n=I+nE_{41}$.  For the two finite points the changes of
affine vectors and compatible conjugating data are
\[
\begin{array}{c|c|c|c|c}
\kappa,j&M_nv_j-v_j&m_j&\lambda_j&t_j\\ \hline
B,1&6k\delta&4&-3kw&\dfrac{3k}{2r}\gamma\\[1mm]
C,1&6k\delta&3&2k\delta&0\\
B,2&-6k\delta&6&-k\delta&0\\
C,2&-6k\delta&6&-k\delta&0.
\end{array}
\]
They satisfy
\begin{equation}\label{eq:root-conjugacy-equation}
 \frac{M_nv_j-v_j}{m_j}=(I-A_j)t_j+\lambda_j.
\end{equation}
For $c'=c-n/r$, write $\Pi=\Pi^\kappa_{r,c}$ and
$\Pi'=\Pi^\kappa_{r,c'}$.  The column identity $\Pi'M_n=\Pi$ holds.
In the direction from parameter $c$ to parameter $c'$, the root-disc
isomorphism is
\begin{equation}\label{eq:directed-root-conjugacy}
 \Psi_j(s,[u])=\bigl(s,[u-\Pi'(s)t_j]\bigr).
\end{equation}
It is equivariant on torus quotients and extends holomorphically over
$s=0$.  Its inverse uses the opposite translation.
\end{lemma}

\begin{proof}
The displayed identities follow directly from the matrices:
\[
 (I-A^B_{1,r})\frac{3k}{2r}\gamma-3kw=\frac{3k}{2}\delta,
 \qquad
 N^C_1(2k\delta)=6k\delta,
 \qquad
 N_2(-k\delta)=-6k\delta.
\]
The last period column gives $\Pi'M_n=\Pi$ directly, and the complex
linear actions $R_j^\kappa(s)$ are the same at both parameters.  Put
$d_j=M_nv_j-v_j$ and let $g(s)=\zeta_{m_j}^{-1}s$.  On the complex
cover the difference between $\Psi_j\Phi_c$ and $\Phi_{c'}\Psi_j$ is
\[
 \Pi'(gs)\left(\frac{d_j}{m_j}-t_j+A_jt_j\right)
 =\Pi'(gs)\lambda_j.
\]
This is an integral period by \eqref{eq:root-conjugacy-equation}, so
the two maps agree on the root-disc torus quotient.  The section
$\Pi'(s)t_j$ and its negative are holomorphic on the full root disc;
hence \eqref{eq:directed-root-conjugacy} and its inverse extend there.
\end{proof}

\begin{proposition}[Exact arithmetic parameter quotient]
\label{prop:parameter-translations}
For every $n\in\mathbb Z$ there is an integral re-marking identity
\begin{equation}\label{eq:period-remarking-identity}
 \Pi^\kappa_{r,c-n/r}(I+nE_{41})=\Pi^\kappa_{r,c}.
\end{equation}
This re-marking preserves the homology marking only when $n=0$.
If $6\mid n$, there is an unmarked biholomorphism inducing this
integral re-marking,
\[
 X^\kappa_{r,c}\simeq X^\kappa_{r,c-n/r}.
\]
For all $c,c'\in\mathcal U_\kappa$,
\[
 X^\kappa_{r,c}\simeq X^\kappa_{r,c'}
 \quad\Longleftrightarrow\quad
 c-c'\in(6/r)\mathbb Z,
\]
whereas a homology-marked isomorphism, and in particular a strictly
marked isomorphism, forces $c=c'$.
\end{proposition}

\begin{proof}
Right multiplication in \eqref{eq:period-remarking-identity} adds $n$
times the fourth period column to the first.  For the cusp extension,
let $E_{21}^{(2)}$ be the $2\times2$ elementary matrix.  Both explicit
cusp decompositions give
\begin{equation}\label{eq:cusp-integral-parameter-shift}
 C_r^\kappa(t_c,c-n/r)=C_r^\kappa(t_c,c)-nE_{21}^{(2)}.
\end{equation}
For every $\lambda\in\Z^2$ the exponential multipliers of these two
twist matrices are identical.  The periodic toric deck actions are
therefore identical, and the identity in toric coordinates extends
the re-marking across the cusp.  Its integral linear part is
$M_n=I+nE_{41}$; it acts trivially on $\Lambda_{\mathrm{van}}$ and on
$\Lambda/\Lambda_{\mathrm{van}}$, and has no component shift.
For $n=6k$, Lemma~\ref{lem:family-root-conjugacy} gives the two
root-disc maps \eqref{eq:directed-root-conjugacy}.  We compute their
collar windings in the same direction.  Set $d_j=M_nv_j-v_j$,
$\mathsf D_j=A_j$, $\mathsf P_j=A_j^{-1}$, and
$\ell(s)=\log(s)/(2\pi i)$.  Composing the source collar map, the
root-disc map, and the inverse target collar map gives
\begin{equation}\label{eq:collar-conjugacy-logarithm}
 \phi_{c'}\circ\Psi_j\circ\Theta_c(s,[u])
 =\bigl(s,[u+\Pi'(s)b_j(s)]\bigr),\qquad
 b_j(s)=-t_j-\ell(s)d_j.
\end{equation}
Here $\Theta_c$ and $\phi_{c'}$ are the maps of
Lemma~\ref{lem:finite-collar}.  As $\mathsf D_jd_j=d_j$ and
$\ell(gs)=\ell(s)-1/m_j$, the deck-equivariance defect is
\begin{equation}\label{eq:collar-winding-computation}
 b_j(gs)-\mathsf D_jb_j(s)
 =(A_j-I)t_j+\frac{d_j}{m_j}=\lambda_j.
\end{equation}
This defect is expressed in the endpoint marking.  Returning the
logarithm to the initial fibre gives the clockwise winding
\begin{equation}\label{eq:forward-collar-winding}
 \omega_j=\mathsf P_jb_j(gs)-b_j(s)=\mathsf P_j\lambda_j.
\end{equation}
Thus chosen representatives of the two kinds of data are
\[
\begin{array}{c|cc|cc}
 &\lambda_1&\lambda_2&\omega_1&\omega_2\\ \hline
 \text{Type B}&-3kw&-k\delta&3ku&-k\delta\\
 \text{Type C}&2k\delta&-k\delta&2k\delta&-k\delta.
\end{array}
\]
In the common quotient by $(A_j-I)\Lambda=(\mathsf P_j-I)\Lambda$,
$[\omega_j]=[\lambda_j]$, since
$\mathsf P_j\lambda_j-\lambda_j=(\mathsf P_j-I)\lambda_j$.
Changing $b_j$ by $\eta\in\Lambda$ changes $\lambda_j$ by
$(I-A_j)\eta$ and $\omega_j$ by $(\mathsf P_j-I)\eta$.
Take the toric identity above as both the chosen cusp map and the
reference cusp comparison in Proposition~\ref{lem:torus-translation-globalization}.
Its relative valuation is zero and $\omega_0=0$.
All local maps have the same flat integral linear part $M_n$.
Each $\omega_j$ is annihilated by $\chi$, so the oriented sum in
\eqref{eq:collar-winding-obstruction} vanishes in the global coinvariant
lattice \eqref{eq:translation-obstruction-group}.
Proposition~\ref{lem:torus-translation-globalization} now corrects the
local maps by completion-preserving translations and glues them to a
global biholomorphism with linear part $M_n$.

Conversely, Theorem~\ref{thm:algebraic-dimension-one} makes every
biholomorphism between admissible members fibre-preserving and trivial
on the base.
Lemmas~\ref{lem:fibrewise-affine-form} and
\ref{lem:monodromy-centralizer} give linear part
$M=\varepsilon I+nE_{41}$.  Comparing
\[
 \Pi^\kappa_{r,c'}M=R\Pi^\kappa_{r,c}
\]
and using the last two period columns gives $R=\varepsilon I_2$; the
first column then gives
\[
 \varepsilon r(c-c')=n.
\]
If $\varepsilon=-1$, the order-six affine-vector difference is
$2\gamma-n\delta$, which cannot lie in
$N_2\Lambda=6\mathbb Z\gamma\oplus6\mathbb Z\delta$.  Hence
$\varepsilon=1$.  The order-six difference is then $-n\delta$, so the
same norm-lattice condition gives $6\mid n$.
\end{proof}

\begin{corollary}[The connected automorphism group]\label{cor:connected-automorphisms}
For every admissible parameter,
\[
 H^0(X^\kappa_{r,c},T_{X^\kappa_{r,c}})=\C v_\delta,\qquad
 \Aut^0(X^\kappa_{r,c})\simeq\C/\Z\simeq\C^*.
\]
Here $v_\delta$ is the global vector field whose flow on the smooth
locus is translation by $a e_2$, $a\in\C$.
\end{corollary}

\begin{proof}
Every automorphism preserves the base by
Theorem~\ref{thm:algebraic-dimension-one}.  A global holomorphic vector
field is therefore vertical, and on each smooth compact torus fibre
it is translation invariant.  Its coefficient is a section of the
vertical translation bundle.  At a finite point the free root-disc
cover pulls the vector field back holomorphically.  At the cusp,
evaluation on the extended zero section, which lies in the smooth
part of the central fibre, gives holomorphic coefficients in
$x_1\partial_{x_1},x_2\partial_{x_2}$.  These coefficients are
independent of the fibre coordinate on the punctured disc, hence
extend the translation field on every toric chart.  Conversely such
sections define global vector fields.  Thus
\[
 H^0(X^\kappa_{r,c},T_{X^\kappa_{r,c}})
 \simeq H^0(\mathbb P^1,\mathcal V_\kappa)=\C,
\]
where the last equality follows from
$\mathcal V_\kappa\simeq\mathcal O\oplus\mathcal O(-1)$.
The trivial subbundle is generated by $e_2=\Pi\delta$.
Its translations extend across the finite quotients and act in the
second multiplicative coordinate at the cusp, giving a holomorphic
$\C$-action on $X^\kappa_{r,c}$.

If translation by $ae_2$ is the identity, then on a simply connected
smooth base chart $ae_2=\Pi(z)\lambda(z)$ for an integral lattice
vector.  Real-linear independence of the periods makes $\lambda(z)$
unique and locally constant.  Continuation therefore fixes it under
both monodromies.  Their common invariant lattice is
$\Lambda^{A_1,A_2}=\Z\delta$: the order-six matrix first forces the
$u,w$ coordinates to vanish, and the first matrix then forces the
$\gamma$ coordinate to vanish.  Hence $a\in\Z$; conversely every
integer is a period.  The image of $\C/\Z$ has the full one-dimensional
Lie algebra of $\Aut^0(X^\kappa_{r,c})$, so it is an open subgroup and
therefore equals that connected group.  Finally $a\mapsto e^{2\pi ia}$
identifies $\C/\Z$ with $\C^*$.
\end{proof}

For the Kodaira--Spencer statements the target $B_\kappa$ and its
three special values are fixed.  The simultaneous family is locally
trivial as a map near its critical points: the finite root charts have
$t_j=s_j^{m_j}$, and the toroidal charts have $t_0=z_0\cdots z_k$;
these equations are independent of the affine parameter.  We use the
following direct infinitesimal observation.

\begin{lemma}[Forgetting an equisingular map]
\label{lem:forgetful-KS}
Let $f:X\to\Pone$ be a proper surjective holomorphic map from a connected
compact complex manifold, with connected fibres.  Suppose that
$q_i\in X$ are critical points with three distinct images
$p_i=f(q_i)$, $i=0,1,2$.  Consider first-order deformations of the map to
the fixed target that are locally trivial as maps near each $q_i$.
The tangent map that forgets the map and retains only the source
deformation has zero kernel.
\end{lemma}

\begin{proof}
Work over $\C[\epsilon]/(\epsilon^2)$.  A deformation in the kernel has
trivial source; choose a trivialization reducing to the identity on $X$.
Its variation as a map is then a global section
$\sigma\in H^0(X,f^*T_{\Pone})$.  Connected fibres, properness, and
Stein factorization give $f_*\mathcal O_X=\mathcal O_{\Pone}$.  The
projection formula therefore identifies this space with
$H^0(\Pone,T_{\Pone})$, so $\sigma=f^*v$ for a vector field $v$ on
$\Pone$.

Local triviality as a map means that, near $q_i$, the chosen global
source trivialization differs from a map-trivializing one by a local
vector field $\xi_i$.  Consequently $\sigma=df(\xi_i)$ there, up to
an immaterial sign.  Since $q_i$ is critical for a map to a curve,
$df_{q_i}=0$, and hence $v(p_i)=0$.  A vector field on $\Pone$ vanishing
at three distinct points is zero, because
\[
 H^0\bigl(\Pone,T_{\Pone}(-p_0-p_1-p_2)\bigr)
 =H^0(\Pone,\mathcal O_{\Pone}(-1))=0.
\]
Thus $\sigma=0$, and the global source trivialization also trivializes
the map.  The pair deformation is zero.
\end{proof}

\begin{theorem}[Deformation and separation in moduli]
For each $r\geq1$ and each $\kappa\in\{B,C\}$, the fixed oriented smooth
six-manifold underlying $X^\kappa_{r,c}$ carries uncountably many pairwise
non-biholomorphic complex structures obtained from the affine period
parameter.  The unmarked isomorphism classes within this family form
the orbit set
\[
 \mathcal U_\kappa/(6/r)\mathbb Z.
\]
Each half-open vertical strip
\[
 \{c\in\mathcal U_\kappa:a\leq\operatorname{Re}c<a+6/r\}
\]
is a fundamental domain.  Both the Kodaira--Spencer class of the family
of pairs $(X^\kappa_{r,c},f^\kappa_{r,c})$ and the absolute
Kodaira--Spencer class are nonzero at every admissible parameter.
\end{theorem}

\begin{proof}
Proposition~\ref{prop:simultaneous-parameter-family} fixes the underlying
oriented diffeomorphism type, while
Proposition~\ref{prop:parameter-translations} identifies the isomorphism
relation on the entire half-plane with translation by $(6/r)\mathbb Z$.
A half-open vertical strip of width $6/r$ is therefore a fundamental
domain and contains an uncountable pairwise distinct family.

At a fixed smooth base point the derivative of the marked period matrix is
\[
 \frac{d}{dc}\Pi^\kappa_{r,c}
 =r\begin{pmatrix}0&0&0&0\\1&0&0&0\end{pmatrix}.
\]
A tangent vector produced by changing the complex basis of $\mathbb C^2$
has the form $A\Pi^\kappa_{r,c}$.  Its last two columns are the columns of
$A$, so it cannot equal the displayed derivative.  Hence the family of
pairs has nonzero Kodaira--Spencer class.  At each finite special
value the local equation $t_j=s_j^{m_j}$ supplies critical points, and
at the cusp any double or triple point is critical.  The simultaneous
root and toroidal charts make the family locally trivial as a map near
these points.  Thus Lemma~\ref{lem:forgetful-KS} applies to its
first-order variation and shows that the absolute Kodaira--Spencer class
is nonzero at every admissible parameter.
\end{proof}

\section{Topology of the compact threefolds}\label{sec:topology}

Fix an admissible affine parameter $c$.  Throughout this section we write
$Y_r=Y_{r,c}$ and $Z_r=Z_{r,c}$; Proposition~\ref{prop:simultaneous-parameter-family}
shows that the groups computed below are independent of $c$.  Put
$X_r^B=Y_r$ and $X_r^C=Z_r$.

Van Kampen computes the fundamental group from the three boundary
fillings.  The normalization complex gives the integral cohomology of
the cusp; two successive Mayer--Vietoris attachments then compute the
integral homology of the total space.  The geometric bases and orientation
conventions are recorded in Appendix~\ref{app:integral-mv-ledger}.

\subsection{Coinvariants and fundamental group}
For either type, the matrices in Appendix~\ref{app:lift-calculations}
give
\[
 \Lambda/\bigl((A_{1,r}^\kappa-I)\Lambda+
 (A_{2,r}^\kappa-I)\Lambda\bigr)\simeq\Z,
\]
with quotient map the primitive invariant functional $\chi$.

\begin{lemma}[Invariant closure of the vanishing lattice]\label{lem:normal-closure}
For $\kappa=B,C$, the smallest subgroup containing
$\Lambda_{\mathrm{van}}=\Z\langle w,\delta\rangle$ and invariant under
$A_{1,r}^\kappa,A_{2,r}^\kappa$ is
$\ker\chi=\Z\langle u,w,\delta\rangle$.
\end{lemma}

\begin{proof}
The functional $\chi$ is invariant.  Conversely,
$A^B_{1,r}w=u-\delta$ and $A^C_{1,r}w=u-w-\delta$.
Thus the invariant closure of $w,\delta$ also contains $u$.
\end{proof}

\begin{theorem}[Fundamental group of the two completed families]
\label{thm:van-kampen-reduction}
Let $\rho_j$ be the zero-section lifts of the clockwise meridians
with the left-to-right path product of
\eqref{eq:path-deck-transport-product}.  Then
\[
 \pi_1(J_r^\kappa)=\Lambda\rtimes F(\rho_1,\rho_2),\qquad
 \rho_jt_\lambda\rho_j^{-1}=t_{\mathsf D_{j,r}^\kappa\lambda},\qquad
 \rho_0=(\rho_1\rho_2)^{-1},
\]
where $\mathsf D_{j,r}^\kappa=A_{j,r}^\kappa$ is the conjugation
matrix and $\mathsf P_{j,r}^\kappa=(\mathsf D_{j,r}^\kappa)^{-1}$
is the forward return.  In particular, the conjugation matrix at the
cusp is $(\mathsf D_1\mathsf D_2)^{-1}=\mathsf D_0$, consistently
with Proposition~\ref{prop:cusp-pi1}.
Let $v_j$ be the invariant vectors of the finite fillings.  Suppose the
cusp collar identifies its meridian with $t_{\xi_0}\rho_0$, where
$\xi_0\in\Lambda$ is its winding class.  Put
$\ell_j=\chi(v_j)$ and $\ell_0=\chi(\xi_0)$.
Then
\begin{equation}\label{eq:unified-pi1}
 \pi_1(X_r^\kappa)=
 \left\langle z,x,y\ \middle|\ z\text{ central},\
 xy=z^{\ell_0},\ x^{m_1}=z^{\ell_1},\ y^{m_2}=z^{\ell_2}
 \right\rangle,
\end{equation}
where $(m_1,m_2)=(4,6)$ for Type B and $(3,6)$ for Type C.
If $\gcd(m_1,\ell_1)=\gcd(m_2,\ell_2)=1$, this group is cyclic of order
\[
 |p|,\qquad
 p=m_1m_2\ell_0-m_2\ell_1-m_1\ell_2.
\]
For the canonical choices $(\ell_0,\ell_1,\ell_2)=(0,1,-1)$,
\[
 \pi_1(Y_r)\simeq\Z/2,\qquad \pi_1(Z_r)\simeq\Z/3.
\]
\end{theorem}

\begin{proof}
Proposition~\ref{prop:finite-filling-pi1} gives the full kernel at
each finite filling: attaching it imposes
$\rho_j^{m_j}=t_{v_j}$.
Proposition~\ref{prop:cusp-pi1} gives the full cusp kernel:
attaching it kills $\Lambda_{\mathrm{van}}$ and
$t_{\xi_0}\rho_0$, or equivalently imposes
$\rho_1\rho_2=t_{\xi_0}$.
By Lemma~\ref{lem:normal-closure}, the normal closure of the vanishing
lattice kills $\ker\chi$.  The remaining fibre group is
$\Z\langle z\rangle$, with $t_\lambda=z^{\chi(\lambda)}$;
both finite monodromies act trivially on it.  Van Kampen gives
\eqref{eq:unified-pi1}, with no further relations because the three
boundary kernels have been specified completely.

The relation $xy=z^{\ell_0}$ makes the group abelian.  Its relation
matrix is
\[
 R=\begin{pmatrix}
 -\ell_1&m_1&0\\-\ell_2&0&m_2\\-\ell_0&1&1
 \end{pmatrix},\qquad \det R=-p.
\]
Its entries have greatest common divisor one.  Its $2\times2$ minors
include $m_1$ and $m_1\ell_0-\ell_1$, which are coprime under the
stated hypothesis.  Hence the first two Smith invariants are one.
For Type B the coprimality conditions give
$p\equiv\pm2\pmod{12}$; for Type C they give
$p\equiv3\pmod6$.  In particular the determinant is nonzero.
The canonical choices give $p=-2$ in Type B and $p=-3$ in Type C.
\end{proof}

\begin{corollary}[Rank-one Seifert relation with several finite points]
If the same rank-one coinvariant reduction has finite fillings of multiplicities $m_1,\ldots,m_n$ and projected translation classes $\ell_1,\ldots,\ell_n$, then the abelianized attaching group is the cokernel of
\[
 \begin{pmatrix}
 -\ell_1&m_1&0&\cdots&0\\
 -\ell_2&0&m_2&\cdots&0\\
 \vdots&\vdots&&\ddots&\vdots\\
 -\ell_n&0&0&\cdots&m_n\\
 -\ell_0&1&1&\cdots&1
 \end{pmatrix}.
\]
If this matrix has full rank, its order is
\[
 \left|\left(\prod_{i=1}^n m_i\right)\ell_0
 -\sum_{i=1}^n\ell_i\prod_{j\neq i}m_j\right|.
\]
This is the abelianized analogue of the preceding two-filling presentation.
\end{corollary}

\subsection{Euler characteristic}

Every smooth torus fibre and every free finite quotient of a torus has Euler characteristic zero.  Hence the Euler number is concentrated at the cusp.

\begin{proposition}
For the two geometric families,
\[
 e(Y_r)=e(W_r)=2r,
 \qquad
 e(Z_r)=e(W^C_r)=2r.
\]
Thus members with distinct indices $r$ are pairwise non-homeomorphic
within either family.  Moreover no $Y_r$ is homeomorphic to any $Z_s$,
because their fundamental groups are $\Z/2$ and $\Z/3$.
\end{proposition}

\subsection{The toroidal neighbourhood and controlled collapse}

The closed cusp neighbourhoods $N^\kappa_{0,r}(\rho)$ and their boundaries
$M^\kappa_{0,r}(\rho)$ were fixed in
\eqref{eq:closed-cusp-tube}.  We compare the level tube with its central fibre by first verifying the
Whitney--Thom hypotheses on the quotient model and then applying a
neighbourhood-retraction theorem.

\begin{lemma}[Whitney stratification and the Thom condition]
\label{lem:cusp-thom-stratification}
For $\kappa=B,C$, choose $0<\rho_0<\epsilon$ and put
\[
 \mathcal N^\kappa_{0,1}(<\rho_0)
 =(f_1^\kappa)^{-1}(|t_c|<\rho_0).
\]
Stratify this open tube by $t_c\neq0$ and the connected components of
the smooth-branch, double-branch, and triple-branch strata of
$W^\kappa_1$.  This is a Whitney \textup{(B)} stratification, and
\[
 f:\mathcal N^\kappa_{0,1}(<\rho_0)\longrightarrow\Delta_{\rho_0}
\]
is a proper Thom stratified map for the stratification
$\Delta_{\rho_0}^*\sqcup\{0\}$ of the open disc.  The same statement
holds on every finite cyclic cover
$\mathcal N^\kappa_{0,r}(<\rho_0)$.
The closed tubes $N^\kappa_{0,r}(\rho)$, $\rho<\rho_0$, will be used
as compact manifolds with boundary.
\end{lemma}

\begin{proof}
Work first on the ordered-branch toric cover.  Near a point where
$k+1$ branches meet there are holomorphic coordinates in which
\[
 t_c=z_0\cdots z_k,
 \qquad k=0,1,2,
\]
and the strata are the products of coordinate-hyperplane strata with a
smooth factor.  Coordinate subspace stratifications are Whitney
\textup{(B)}, and this property is preserved under products.

We verify the Thom condition directly.  Let $x_n$ lie in the open
stratum, tend to a point $x$ of a central stratum $S_I$, and suppose that
$\ker(df_{x_n})$ tends to a complex hyperplane $H$.  If
$v\in T_xS_I$, choose $i\in I$ and define $v_n$ by keeping all components
of $v$ except
\[
 (v_n)_i=-z_i(x_n)\sum_{j\ne i}\frac{v_j}{z_j(x_n)}.
\]
For $j\in I\setminus\{i\}$ one has $v_j=0$, so $v_n\to v$; moreover
$df_{x_n}(v_n)=0$ by the logarithmic identity
$df/f=\sum_jdz_j/z_j$.  Hence $T_xS_I\subset H$, which is Thom's
$a_f$ condition.

The unit-twisted toric transition maps preserve $t_c$ and permute the
ordered local branches, so the local stratifications and the Thom
condition glue on the ordered cover.  The deck action is free and acts by
stratified local biholomorphisms.  Whitney \textup{(B)} and the Thom
condition are local properties, and therefore descend to the index-one
quotient and to its finite covers.  Properness over the open disc
follows by restriction from Theorem~\ref{thm:master-cusp}.
\end{proof}

\begin{lemma}[Level-tube collapse compatible with finite covers]
\label{lem:cusp-neighbourhood-collapse}
For each $\kappa=B,C$ and $r\geq1$, after decreasing $\rho$ if necessary,
$W_r^\kappa$ is a strong deformation retract of
$N^\kappa_{0,r}(\rho)$.  These retractions commute with the
cyclic covering maps
$N^\kappa_{0,r}(\rho)\to N^\kappa_{0,1}(\rho)$.
\end{lemma}

\begin{proof}
We first treat $r=1$.  Choose $\rho_0>\rho$ so that the open tube
\[
 X_0=(f_1^\kappa)^{-1}(|t_c|<\rho_0)
\]
is contained in the toroidal model.  By
Lemma~\ref{lem:cusp-thom-stratification}, $X_0$ is a Whitney stratified
space and $W_1^\kappa$ is a closed union of strata.  The neighbourhood
retraction theorem of \cite[Theorem~1.1]{PflaumWilkin}, with trivial
group, ambient manifold $X_0$, and closed stratified subspace
$W_1^\kappa$, applies using the control data of \cite{MatherControl}.
It gives an open neighbourhood
$U\subset X_0$ of $W_1^\kappa$ and a stratified strong deformation
retraction
\[
 H:U\times[0,1]\longrightarrow U
\]
onto $W_1^\kappa$.

The closed level tubes form a neighbourhood basis of the central fibre.
Indeed, if no $N^\kappa_{0,1}(\rho')$ were contained in $U$, a sequence
outside $U$ with $|t_c|\to0$ would have, by properness over a closed
subdisc, a limit point in $W_1^\kappa\subset U$.  Hence, after decreasing
$\rho$, we may assume
$N^\kappa_{0,1}(\rho)\subset U$.
The set
\[
 K=H\bigl(N^\kappa_{0,1}(\rho)\times[0,1]\bigr)
\]
is compact in $X_0$, so it is contained in
$N^\kappa_{0,1}(R)$ for some $\rho<R<\rho_0$.

Over the closed annulus $\rho\le |t_c|\le R$, the map $f_1^\kappa$ is a
proper submersion.  Choose an Ehresmann connection tangent to the extended zero section and
compatible with the fixed homology marking, and use horizontal transport
along radial segments at fixed argument.  This defines a continuous
radial compression
\[
 q_{\rho,R}:N^\kappa_{0,1}(R)\longrightarrow
 N^\kappa_{0,1}(\rho)
\]
which is the identity on $N^\kappa_{0,1}(\rho)$ and preserves the fibre
marking and the zero-section meridian.  The homotopy
\[
 G(x,t)=q_{\rho,R}\bigl(H(x,t)\bigr),
 \qquad x\in N^\kappa_{0,1}(\rho),
\]
is therefore a strong deformation retraction onto $W_1^\kappa$: at
$t=0$ it is the identity, at $t=1$ it takes values in $W_1^\kappa$, and
it fixes $W_1^\kappa$ pointwise.

For general $r$, lift $H$ through the finite covering of toroidal
neighbourhoods and pull back the chosen Ehresmann connection.  The lifted
homotopy fixes $W_r^\kappa$ pointwise by uniqueness of path lifting, and
the lifted radial compression gives the same formula for
$N^\kappa_{0,r}(\rho)$.  The construction is consequently compatible
with all cyclic covering maps.
\end{proof}

\subsection{The normalization complex of the central fibre}

The index-one cusp is globally irreducible and nonnormal: its normalization
is a single copy of $dP_6$ and opposite toric boundary curves are paired.
We form the normalization complex on the periodic toric cover, where
the branches are globally labelled, and then descend it through the
lattice quotient.

Let $\widetilde W_\infty$ be the central fibre of the periodic toric
cover from Section~\ref{sec:cusp}.  Its components $D_v$ are indexed by
$v\in\Z^2$.  Fix the lexicographic order on $\Z^2$, which is preserved
by every lattice translation.  Define the ordered branch spaces by
\[
 W_r^{\langle p\rangle}
 =\left(\coprod_{v_0<\cdots<v_p}
       D_{v_0}\cap\cdots\cap D_{v_p}\right)/\Gamma_r,
 \qquad \Gamma_r=B_{0,r}\Z^2,
\]
where empty intersections are omitted.  Deleting the $i$th vertex gives
face maps $d_i$, since translations preserve both the order and the face
maps.  In particular $W_r^{\langle0\rangle}$ is the normalization of
$W_r$.  Branches, rather than global component names, distinguish the
summands when several branches descend to the same normalization
component.

The deck group has no stabilizer on a nonempty finite ordered simplex.
There are $r$ vertex orbits, $3r$ edge orbits, and $2r$ triangle orbits;
hence
\[
 W_r^{\langle0\rangle}=\coprod_r dP_6,\qquad
 W_r^{\langle1\rangle}=\coprod_{3r}\Pone,\qquad
 W_r^{\langle2\rangle}=\coprod_{2r}\{\mathrm{pt}\}.
\]
The construction applies to both Type B and Type C with their respective
subgroups $\Gamma_r$.  The numbers of strata agree, but their incidence
maps retain the subgroup-dependent vertex labels.

Stalkwise the augmented complex below has the following groups:
\[
\begin{array}{c|cccc}
\text{point type}&\Z_{W_r}&\epsilon_{0*}\Z&\epsilon_{1*}\Z&\epsilon_{2*}\Z\\ \hline
\text{smooth branch}&\Z&\Z&0&0\\
\text{double point}&\Z&\Z^2&\Z&0\\
\text{triple point}&\Z&\Z^3&\Z^3&\Z.
\end{array}
\]
At a double point the maps are $1\mapsto(1,1)$ and
$(a,b)\mapsto b-a$; at a triple point they are the alternating face maps
of the ordered $2$-simplex.  Orienting a double curve in the opposite
direction changes the corresponding target generator by a sign.

\begin{proposition}[Normalization resolution]
\label{prop:branch-normalization-resolution}
Let $\epsilon_p:W_r^{\langle p\rangle}\to W_r$ be the natural maps.  With
incidence differential
\[
 \partial_{\mathrm{br}}
 =\sum_{i=0}^{p+1}(-1)^i d_i^*,
\]
the augmented complex of sheaves
\begin{equation}\label{eq:branch-resolution}
 0\longrightarrow\Z_{W_r}
 \longrightarrow\epsilon_{0*}\Z_{W_r^{\langle0\rangle}}
 \xrightarrow{\partial_{\mathrm{br}}}
 \epsilon_{1*}\Z_{W_r^{\langle1\rangle}}
 \xrightarrow{\partial_{\mathrm{br}}}
 \epsilon_{2*}\Z_{W_r^{\langle2\rangle}}
 \longrightarrow0
\end{equation}
is exact.  It is the descent of the ordinary normalization complex of the
simple-normal-crossings divisor $\widetilde W_\infty$.  In particular it
remains valid when $r=1$, where the three pairs of boundary curves belong
to the same normalization component.

Its hypercohomology computes $H^*(W_r;\Z)$.
\end{proposition}

\begin{proof}
Everything is first constructed upstairs.  At a point where exactly
$k+1$ branches meet, the stalk of the augmented ordered-branch complex is
the augmented simplicial cochain complex of a $k$-simplex.  It is exact.
Each analytic toric chart meets only the central divisors corresponding
to its rays; in a maximal chart there are at most three local branches.
Thus the restriction of the augmented branch complex to that chart is
a finite complex, with the stalks just described.  These restrictions
agree on common face charts and define an exact complex of sheaves on
$\widetilde W_\infty$.  The action of
$\Gamma_r$ is free on the set of labelled branches and preserves all face
maps and incidence signs.  To check the quotient complex, choose a lift of
a point $x\in W_r$.  The branches through the lift form the same finite
simplex as the branches through $x$, and the stalk of the quotient complex
is its augmented simplicial cochain complex.  Thus exactness holds
stalkwise downstairs and yields \eqref{eq:branch-resolution}.  This
construction does not require global irreducible-component labels on
$W_r$; the labels live on the ordered normalization fibre products.

Applying derived global sections to this exact resolution gives
$R\Gamma(W_r,\Z)$ and the additive normalization spectral sequence.
\end{proof}

\subsection{Integral cohomology of the cusp fibre}

The ordered-branch strata are
\[
 W_r^{\langle0\rangle}=\coprod_{r}dP_6,
 \qquad
 W_r^{\langle1\rangle}=\coprod_{3r}\Pone,
 \qquad
 W_r^{\langle2\rangle}=\coprod_{2r}\{\mathrm{pt}\},
\]
and the branch nerve is the two-torus with its $A_2$ triangulation.

\begin{theorem}[Cohomology of the toroidal fibre]\label{thm:cusp-cohomology}
For either $\kappa=B,C$, the integral cohomology of $W_r=W_r^\kappa$
is torsion-free and
\[
 H^q(W_r;\Z)\cong
 \begin{cases}
 \Z,&q=0,\\
 \Z^2,&q=1,\\
 \Z^{r+3},&q=2,\\
 \Z^2,&q=3,\\
 \Z^r,&q=4,\\
 0,&\text{otherwise}.
 \end{cases}
\]
\end{theorem}

\begin{proof}
Use the spectral sequence of the exact ordered-branch resolution
\eqref{eq:branch-resolution}:
\[
 E_1^{p,q}=H^q(W_r^{\langle p\rangle};\Z)
 \Longrightarrow H^{p+q}(W_r;\Z).
\]
The $q=0$ row is the cellular cochain complex of the dual two-torus and contributes $\Z,\Z^2,\Z$ in degrees $0,1,2$.  The $q=4$ row is $\Z^r$ in column zero.

For the degree-two row, put $\ell=1$ for Type B and $\ell=2$ for
Type C.  The following integral calculation treats both values.
The component of a vertex $(a,b)\in\Z^2$ is labelled by
$j=a-\ell b\pmod r$, since
\[
 \Z^2/B^\kappa_{0,r}\Z^2\longrightarrow\Z/r,
 \qquad [(a,b)]\longmapsto a-\ell b
\]
is an isomorphism.  In the cyclic order of rays
\[
 e_1,\ e_2,\ e_2-e_1,\ -e_1,\ -e_2,\ e_1-e_2,
\]
mark the six toric boundary curves of each $dP_6$ by
\[
 E_1,\ H-E_1-E_2,\ E_2,\ H-E_2-E_3,\ E_3,\ H-E_1-E_3.
\]
A class $a_jH-b_{1j}E_1-b_{2j}E_2-b_{3j}E_3$ has restriction degrees
$b_{1j},a_j-b_{1j}-b_{2j},b_{2j},a_j-b_{2j}-b_{3j},b_{3j},
a_j-b_{1j}-b_{3j}$ on these curves.  Pairing opposite boundary curves
on the actual adjacent components gives the restriction-difference map
$\rho_r^\kappa:\Z^{4r}\to\Z^{3r}$:
\begin{equation}\label{eq:geometric-cusp-restrictions}
 \begin{aligned}
 X_j&=b_{1j}-a_{j+1}+b_{2,j+1}+b_{3,j+1},\\
 Y_j&=a_j-b_{1j}-b_{2j}-b_{3,j-\ell},\\
 Z_j&=b_{2j}-a_{j-\ell-1}+b_{1,j-\ell-1}+b_{3,j-\ell-1}.
 \end{aligned}
\end{equation}
All indices are cyclic.  These formulas use the three directed edge
shifts $1,-\ell,-\ell-1$; when a shift is zero modulo $r$, the row is
the corresponding self-identification.

The image is contained in the kernel of the surjective map
\begin{equation}\label{eq:cusp-image-two-relations}
 \Z^{3r}\longrightarrow\Z^2,\qquad
 (X,Y,Z)\longmapsto
 \left(\sum_j(X_j+Y_j),\ \sum_j(Y_j+Z_j)\right).
\end{equation}
We prove the reverse inclusion over $\Z$.  Let $S$ be cyclic shift,
$(Sb)_j=b_{j+1}$, and put $T=S^{-1}$.  Use the second line of
\eqref{eq:geometric-cusp-restrictions} to set
$a=Y+b_1+b_2+T^\ell b_3$.  The other two equations become
\[
 \begin{aligned}
 X+SY&=(I-S)b_1+S(I-T^\ell)b_3,\\
 Z+T^{\ell+1}Y&=(I-T^{\ell+1})b_2
                  +T^{\ell+1}(I-T^\ell)b_3.
 \end{aligned}
\]
An integer cyclic vector lies in the image of $I-T$ if and only if its
coordinate sum is zero.  The second relation in
\eqref{eq:cusp-image-two-relations} therefore gives an integer vector
$h$ with $Z+T^{\ell+1}Y=(I-T)h$.  Choose
\[
 b_2=h,\qquad b_3=-T^{-\ell}h.
\]
The second equation now holds because
\[
 (I-T^{\ell+1})-T(I-T^\ell)=I-T.
\]
The first equation reduces to an integer cyclic difference equation for
$b_1$ whose right-hand side has coordinate sum zero, by the first
relation in \eqref{eq:cusp-image-two-relations}.  Solving it and then
defining $a$ gives an integral preimage of $(X,Y,Z)$.
Consequently the image of $\rho_r^\kappa$ is exactly that kernel and
is saturated.  Thus
\[
 \ker\rho_r^\kappa\cong\Z^{r+2},\qquad
 \coker\rho_r^\kappa\cong\Z^2.
\]
The spectral sequence has columns only $0,1,2$.  Its $d_2$ maps have
zero targets in odd cohomological rows, and every $d_k$ with $k\ge3$
leaves the column range.  The contributions are free abelian, so the
remaining additive extensions split.  Combining the three rows gives
the stated groups.
\end{proof}

\subsection{Integral homology}
\label{subsec:complete-integral-homology}

We determine the integral homology using the same geometric
decomposition as in the fundamental-group calculation.  Appendix~\ref{app:integral-mv-ledger} records the geometric cycles, orientation conventions, finite-quotient relation matrices, and every integral map used below.  Write
\[
 X_r^B=Y_r,\qquad X_r^C=Z_r,
\]
and denote the two reduced finite fibres by $S_1^\kappa,S_2^\kappa$.
Let $F$ be a smooth fibre and set
\[
 \mathcal E=(\gamma^*\wedge u^*,\gamma^*\wedge w^*,
 \gamma^*\wedge\delta^*,u^*\wedge w^*,
 u^*\wedge\delta^*,w^*\wedge\delta^*)
\]
for the standard basis of $H^2(F;\mathbb Z)$.

\begin{lemma}[Integral cohomology bases of the finite fibres]
\label{lem:finite-H2-bases}
For each finite quotient $p_j:F\to S_j^\kappa$, the group
$H^2(S_j^\kappa;\mathbb Z)$ is free of rank two.  There are integral
bases $a_{j,1},a_{j,2}$ for which the pullback matrices, with respect to
$\mathcal E$, are
\[
 P^B_1(r)=
 \begin{pmatrix}
 -r&2\\ r&-2\\0&4\\1&0\\0&0\\0&0
 \end{pmatrix},\qquad
 P^C_1(r)=
 \begin{pmatrix}
 -r&1\\ r&-1\\0&3\\1&0\\0&0\\0&0
 \end{pmatrix},
 \tag{\ref{lem:finite-H2-bases}.1}
\]
and, for the common order-six fibre,
\[
 P_2=
 \begin{pmatrix}
 0&0\\0&0\\0&-6\\1&0\\0&0\\0&0
 \end{pmatrix}.
 \tag{\ref{lem:finite-H2-bases}.2}
\]
With the marked orientations of Appendix~\ref{app:mv-orientations},
the intersection matrices on the two quotient surfaces are respectively
\[
 \begin{pmatrix}0&1\\1&0\end{pmatrix},\qquad
 \begin{pmatrix}0&-1\\-1&0\end{pmatrix}.
\]
\end{lemma}

\begin{proof}
The affine fundamental group of a free cyclic quotient has presentation
\[
 \Gamma_{A,v,m}=\langle \Lambda,a\mid
 [\Lambda,\Lambda]=1,\ a\lambda a^{-1}=A\lambda,\ a^m=v\rangle.
\]
For the four actions occurring here, the five-generator abelianization
matrices have Smith form $\operatorname{diag}(1,1,1,0,0)$, and hence
$H_1(S_j^\kappa;\mathbb Z)\simeq\mathbb Z^2$.  Since the quotient is a
closed oriented four-manifold with Euler characteristic zero, Poincar\'e
duality gives $b_2=2$, and the universal coefficient theorem shows that
$H^2$ is free.

Let $I=H^2(F;\Z)^T$, where $T=(A^{-1})^{\mathsf T}$, and let
$L=p_j^*H^2(S_j^\kappa;\Z)\subset I$.  Transfer shows that $L$ has full
rank in $I$ and that the pullback is injective.  Orient $F$ by its
marked lattice and the quotient surface so that the cover has positive
degree, as in Appendix~\ref{app:mv-orientations}.  Poincar\'e duality
makes the intersection lattice of $S_j^\kappa$ unimodular.  Hence the
absolute discriminant of $L$ is $m_j^2$.

The integral invariant bases in Appendix~\ref{app:finite-quotient-h2}
have intersection matrices $2U$, $3U$, and $U$ for the Type B first,
Type C first, and common order-six fibres, respectively, where
$U=\begin{psmallmatrix}0&1\\1&0\end{psmallmatrix}$.  The lattice-index
identity $|\det L|=[I:L]^2|\det I|$ gives
\[
 [I:L]=2,\qquad 1,\qquad 6,
\]
respectively.

To identify which sublattices occur, use the invariant circle $\delta$.
The square $(t,s)\mapsto[t\delta+s v_j/m_j]$ in $S_j^\kappa$ closes
to an oriented two-torus, because $A_j\delta=\delta$ and the affine
generator identifies its $s$-edges.  For $\alpha=p_j^*a$, integration
of the invariant constant two-form representing $\alpha$ gives
\begin{equation}\label{eq:finite-pullback-sweep-test}
 \frac{\alpha(\delta,v_j)}{m_j}
 =\langle a,s_{j,\delta}\rangle\in\Z.
\end{equation}
In the Type B invariant basis $(\omega_{B,1},\omega_{B,2})$ below,
this condition on $a\omega_{B,1}+b\omega_{B,2}$ is $b\in2\Z$.
The resulting candidate sublattice has index two and therefore equals
$L$.  In Type C the index is one, so $L=I$.  At the order-six fibre,
in the basis $(u^*\wedge w^*,\gamma^*\wedge\delta^*)$, the condition
is $b\in6\Z$, again giving the exact index.  These are precisely the
columns displayed in the statement.  Dividing their wedge pairings on
$F$ by $m_j$ gives the displayed quotient intersection matrices.
\end{proof}

Let
$X_r^{\kappa,\circ}=X_r^\kappa\setminus
\operatorname{int}N^\kappa_{0,r}(\rho)$.
Equivalently, it is obtained by attaching the two finite logarithmic
transforms but not the cusp.  It is
the union of two pieces retracting to $S_1^\kappa,S_2^\kappa$, with
intersection homotopy equivalent to $F$.  Fix the target homology bases by
\[
 x_{j,1}=p_{j*}T_{uw},\qquad x_{j,2}=-s_{j,\delta},
\]
and orient Mayer--Vietoris so that
$\alpha_2^\kappa=(p_{1*},-p_{2*})$.  In these bases the map
\[
 \alpha_2^\kappa:H_2(F;\mathbb Z)\longrightarrow
 H_2(S_1^\kappa;\mathbb Z)\oplus H_2(S_2^\kappa;\mathbb Z)
\]
is exactly
\begin{equation*}\label{eq:finite-attachment-matrix-B}
 A_B(r)=
 \begin{pmatrix}
 -r&r&0&1&0&0\\
 2&-2&4&0&0&0\\
 0&0&0&-1&0&0\\
 0&0&6&0&0&0
 \end{pmatrix},
 \tag{\ref*{subsec:complete-integral-homology}.1}
\end{equation*}
\begin{equation*}\label{eq:finite-attachment-matrix-C}
 A_C(r)=
 \begin{pmatrix}
 -r&r&0&1&0&0\\
 1&-1&3&0&0&0\\
 0&0&0&-1&0&0\\
 0&0&6&0&0&0
 \end{pmatrix}.
 \tag{\ref*{subsec:complete-integral-homology}.2}
\end{equation*}
The degree-one maps, in suitable bases of the two quotient surfaces, are
\[
 \begin{array}{c|c|c}
 &p_{1*}&p_{2*}\\ \hline
 B&\begin{pmatrix}4&0&0&0\\0&1&-1&2\end{pmatrix}
  &\begin{pmatrix}-6&0&0&0\\0&0&0&1\end{pmatrix}\\[3mm]
 C&\begin{pmatrix}3&0&0&0\\0&1&-1&3\end{pmatrix}
  &\begin{pmatrix}-6&0&0&0\\0&0&0&1\end{pmatrix}.
 \end{array}
 \tag{\ref{subsec:complete-integral-homology}.3}
\]
Consequently the kernel of
$\alpha_1=(p_{1*},-p_{2*})$ is the primitive line
\[
 \ker\alpha_1=\mathbb Z(u+w).
 \tag{\ref{subsec:complete-integral-homology}.4}
\]
The quotient presentations and the chosen bases of $H_1(S_j^\kappa)$ are derived in Appendix~\ref{app:finite-quotient-h1}.

\begin{lemma}[Kernel of the boundary inclusion at the toroidal fibre]
\label{lem:cusp-sweeps-integral}
Let $M^\kappa_{0,r}(\rho)=\partial N^\kappa_{0,r}(\rho)$ and let $\xi_\nu\in H_2(M^\kappa_{0,r}(\rho);\mathbb Z)$
be the torus obtained by sweeping the circle in the invariant direction
$\nu\in\Lambda_{\mathrm{van}}=\mathbb Z\langle w,\delta\rangle$ along the
clockwise cusp meridian through the extended zero section, with
circle-first product orientation.  Write $j$ for the inclusion of this
boundary into $X_r^{\kappa,\circ}$.  Then
\begin{equation*}\label{eq:integral-cusp-sweep-kernel}
 K_2:=\ker\bigl(H_2(M^\kappa_{0,r}(\rho);\mathbb Z)\to
 H_2(N^\kappa_{0,r}(\rho);\mathbb Z)\bigr)
 =\mathbb Z\xi_w\oplus\mathbb Z\xi_\delta.
 \tag{\ref*{lem:cusp-sweeps-integral}.1}
\end{equation*}
With the cuts and the Mayer--Vietoris orientation
$(i_{1*},-i_{2*})$ specified in Appendix~\ref{app:cusp-boundary-ledger},
\[
 \partial(j_*\xi_w)=-(u+w),\qquad
 \partial(j_*\xi_\delta)=0.
 \tag{\ref{lem:cusp-sweeps-integral}.2}
\]
Modulo $\operatorname{im}\alpha_2^\kappa$, the class $-j_*\xi_\delta$
is represented by
\[
 b=(0,-1,0,-1)^{\mathsf T}
 \in H_2(S_1^\kappa)\oplus H_2(S_2^\kappa).
 \tag{\ref{lem:cusp-sweeps-integral}.3}
\]
\end{lemma}

\begin{proof}
Let
\[
 \mathcal E_2=(\gamma\wedge u,\gamma\wedge w,
 \gamma\wedge\delta,u\wedge w,u\wedge\delta,w\wedge\delta)
\]
be the basis of $H_2(F;\mathbb Z)$.  Direct exterior-square calculation
gives
\[
 \operatorname{SNF}\bigl(\wedge^2\mathsf P^\kappa_{0,r}-I\bigr)
 =\operatorname{diag}(1,1,0,0,0,0)
 \qquad(\kappa=B,C).
 \tag{\ref{lem:cusp-sweeps-integral}.4}
\]
The matrices of the forward exterior-square maps and their unit minors
are displayed in Appendix~\ref{app:cusp-boundary-ledger}.  Write
$\mathsf P_0=\mathsf P^\kappa_{0,r}$.  The Wang sequence has the exact
segment
\[
 0\longrightarrow\operatorname{coker}(\wedge^2\mathsf P_0-I)
 \longrightarrow H_2(M^\kappa_{0,r}(\rho);\mathbb Z)
 \longrightarrow\ker(\mathsf P_0-I)\longrightarrow0.
\]
Its outer groups are free, so its middle group is torsion-free.
Orient the Wang cut projection by circle-first product orientation
along the clockwise meridian.  Its restriction to the sweep tori is
\[
 \xi_w\longmapsto w,
 \qquad \xi_\delta\longmapsto\delta.
\]
Thus $L=\mathbb Z\xi_w\oplus\mathbb Z\xi_\delta$ is a primitive direct
summand of $H_2(M^\kappa_{0,r}(\rho);\mathbb Z)$.

We give explicit integral fillings of the two sweeps.  The angular
torus $(S^1)^2$ acts on the periodic toric cover by multiplication in
$x_1,x_2$.  This action preserves $t_c$, commutes with the deck action,
and descends to $N=N^\kappa_{0,r}(\rho)$.  Let
$e:\overline\Delta_\rho\to N$ be the restriction of the extended zero section from
Proposition~\ref{prop:cusp-pi1}.  For $\nu=w,\delta$, the map
\begin{equation}\label{eq:explicit-sweep-filling}
 F_\nu:S^1\times\overline\Delta_\rho\longrightarrow N,\qquad
 F_\nu(e^{2\pi i\theta},t_c)
 =\nu(e^{2\pi i\theta})\cdot e(t_c)
\end{equation}
is a continuous map of a solid torus whose boundary lies in $M=\partial N$.
Give $S^1$ its positive orientation and $\overline\Delta_\rho$ its complex
orientation.  The product-chain identity
\[
 \partial([S^1]\times[\overline\Delta_\rho])
 =-[S^1]\times[\partial\overline\Delta_\rho]
\]
makes its boundary the circle-first sweep along the clockwise meridian.
After triangulation, \eqref{eq:explicit-sweep-filling} therefore gives an
integral relative three-cycle $\mathcal B_\nu$ with
$\partial\mathcal B_\nu=\xi_\nu$.
No embeddedness of the solid-torus map is required.  Thus $L\subseteq K_2$.

The Wang projection above splits integrally on $L$, so
$H_2(M;\mathbb Z)/L$ is free.  On the other hand,
Lemma~\ref{lem:cusp-neighbourhood-collapse} and Poincar\'e--Lefschetz
duality give
\[
 H_3(N,M;\mathbb Z)\simeq H^3(N;\mathbb Z)
 \simeq H^3(W_r;\mathbb Z)\simeq\mathbb Z^2.
\]
Exactness bounds the rank of $K_2$ by two.  Since $L$ already has rank
two, $K_2/L$ is finite and embeds in the free group $H_2(M;\mathbb Z)/L$.
It is zero, proving \eqref{eq:integral-cusp-sweep-kernel}.

For the inclusion into $X_r^{\kappa,\circ}$, traverse $a_1$ first and
$a_2$ second.  The forward matrices give
$\mathsf P_1w=-u$ and $\mathsf P_2(-u)=w$.  A circle-first cylinder
$C_j(\nu)$ over $a_j$ has boundary
$c_\nu-c_{\mathsf P_j\nu}$.  Hence the sweep over $a_1a_2$ is
$C_1(w)+C_2(-u)$; reversing it represents the clockwise cusp sweep.
The first piece of that reversed cycle has boundary $c_{-u}-c_w$,
whose integral class is $-(u+w)$.  With the stated Mayer--Vietoris
convention this proves the displayed connecting image, including its
sign.  It is a primitive generator.  Appendix~\ref{app:cusp-boundary-ledger}
records the two chains, their endpoint identifications, and their
comparison with the cusp boundary loop.

Since $\delta$ is fixed by both monodromies, its circle action sweeps
the zero section over the entire pair of pants.  The boundary of this
three-chain gives
\begin{equation}\label{eq:oriented-delta-sweep-relation}
 j_*\xi_\delta=-s_{1,\delta}-s_{2,\delta}
 \quad\text{in }H_2(X_r^{\kappa,\circ};\mathbb Z).
\end{equation}
Here the finite sweeps are mapped into the respective pieces.
For $a\in H^2(S_j^\kappa;\mathbb Z)$ their evaluations are
\[
 \langle a,s_{j,\delta}\rangle
 =\frac1{m_j}(p_j^*a)(\delta,v_j).
\]
The bases in Lemma~\ref{lem:finite-H2-bases} give $(0,-1)$ for each
$s_{j,\delta}$.  Hence $j_*\xi_\delta$ is represented by
$(0,1,0,1)^{\mathsf T}=-b$, and the chosen source generator
$-\xi_\delta$ gives the displayed column $b$.
Also \eqref{eq:oriented-delta-sweep-relation} places this class in the
image of the two finite pieces, so its connecting boundary is zero.
\end{proof}

\begin{proposition}[The cokernel of the cusp attachment]
\label{prop:open-integral-torsion}
Put $(a_B,p_B)=(2,2)$ and $(a_C,p_C)=(1,3)$.
Then
\begin{equation}\label{eq:reduced-torsion-presentation}
 H_2(X_r^{\kappa,\circ};\Z)/j_*K_2
 \simeq G_{\kappa,r}:=
 \operatorname{coker}\begin{pmatrix}r&0\\-a_\kappa&p_\kappa\end{pmatrix}.
\end{equation}
In particular,
\[
 G_{B,r}\simeq\Z/r\oplus\Z/2,\qquad
 G_{C,r}\simeq\Z/(3r).
\]
\end{proposition}

\begin{proof}
The Mayer--Vietoris sequence for the two finite pieces gives
\[
 0\longrightarrow\coker\alpha_2^\kappa
 \longrightarrow H_2(X_r^{\kappa,\circ};\Z)
 \longrightarrow\ker\alpha_1\longrightarrow0.
\]
By Lemma~\ref{lem:cusp-sweeps-integral}, $j_*\xi_w$ maps to a
primitive generator of $\ker\alpha_1$, while $-j_*\xi_\delta$ gives
the relation $b=(0,-1,0,-1)^{\mathsf T}$ in the left-hand group.
Thus the desired group is $\coker D_\kappa(r)$, where
$D_\kappa(r)=[A_\kappa(r)\mid b]$ is recorded with its geometric bases
in Appendix~\ref{app:mv-augmented-matrices}.

The columns $T_{uw}$ and $-\xi_\delta$ give
$x_3=x_1$ and $x_4=-x_2$.  Eliminating them uses unit coefficients.
The two remaining relations are
\[
 rx_1=a_\kappa x_2,\qquad (6-m_1)x_2=p_\kappa x_2=0,
\]
which proves \eqref{eq:reduced-torsion-presentation} over $\Z$.
For Type B, $2x_2=0$ makes the first relation $rx_1=0$.
For Type C, $x_2=rx_1$ leaves the single relation $3rx_1=0$.
These give the two asserted groups for every positive integer $r$.
\end{proof}

\begin{lemma}[The free relative contribution]
\label{lem:relative-cusp-free}
Let $N^\kappa_{0,r}(\rho)$ and $M^\kappa_{0,r}(\rho)$ be the standard
closed cusp tube and its oriented boundary.  In the excision
sequence for the pair
$(X_r^\kappa,X_r^{\kappa,\circ})$, the group
\[
 \ker\bigl(H_2(N^\kappa_{0,r}(\rho),M^\kappa_{0,r}(\rho);\mathbb Z)
 \longrightarrow H_1(X_r^{\kappa,\circ};\mathbb Z)\bigr)
\]
is free of rank $r-1$.
\end{lemma}

\begin{proof}
Lemma~\ref{lem:cusp-neighbourhood-collapse} gives
$N^\kappa_{0,r}(\rho)\simeq W_r^\kappa$.  Poincar\'e--Lefschetz duality and
Theorem~\ref{thm:cusp-cohomology} therefore give
\begin{equation*}\label{eq:relative-cusp-second-homology}
 H_2(N^\kappa_{0,r}(\rho),M^\kappa_{0,r}(\rho);\mathbb Z)
 \simeq H^4(N^\kappa_{0,r}(\rho);\mathbb Z)
 \simeq\mathbb Z^r.
 \tag{\ref*{lem:relative-cusp-free}.1}
\end{equation*}
Orient the extended zero-section disc oppositely to its complex
orientation, so that its relative class has boundary the clockwise
cusp meridian $a_0$.  Before the cusp is attached, abelianization gives
\[
 H_1(X_r^{B,\circ};\mathbb Z)\simeq\mathbb Z\oplus\mathbb Z/2,
 \qquad
 H_1(X_r^{C,\circ};\mathbb Z)\simeq\mathbb Z\oplus\mathbb Z/3.
\]
The primitive homomorphisms to the free quotient can be chosen as
\[
 \Phi_B(z,x,y)=12z+3x-2y,
 \qquad
 \Phi_C(z,x,y)=6z+2x-y.
\]
For the canonical collar data, $a_0=-(x+y)$ and
$\Phi_B(a_0)=\Phi_C(a_0)=-1$.  We also identify the full image of the
relative boundary map.  Proposition~\ref{prop:cusp-pi1} shows that the
kernel of $H_1(M^\kappa_{0,r})\to H_1(N^\kappa_{0,r})$ is generated
by the meridian and the image of $\Lambda_{\mathrm{van}}$.  In
$H_1(X_r^{\kappa,\circ})$, the joint monodromy images kill
$\ker\chi$, hence in particular $\Lambda_{\mathrm{van}}$.  The image
of the relative boundary map is therefore contained in $\Z[a_0]$;
the extended section disc shows that it equals $\Z[a_0]$.
Since $\Phi_\kappa(a_0)=-1$, this is a free direct summand of
$H_1(X_r^{\kappa,\circ})$.  The kernel is consequently a direct
summand of \eqref{eq:relative-cusp-second-homology}, free of rank $r-1$.
\end{proof}

\begin{theorem}[Integral homology]
\label{thm:complete-integral-homology}
For every $r\ge1$,
\[
 H_i(Y_r;\mathbb Z)\simeq
 \begin{cases}
 \mathbb Z,&i=0,6,\\
 \mathbb Z/2,&i=1,\\
 \mathbb Z^{r-1}\oplus\mathbb Z/r\oplus\mathbb Z/2,&i=2,\\
 \mathbb Z/r\oplus\mathbb Z/2,&i=3,\\
 \mathbb Z^{r-1}\oplus\mathbb Z/2,&i=4,\\
 0,&i=5,
 \end{cases}
 \tag{\ref{thm:complete-integral-homology}.1}
\]
and
\[
 H_i(Z_r;\mathbb Z)\simeq
 \begin{cases}
 \mathbb Z,&i=0,6,\\
 \mathbb Z/3,&i=1,\\
 \mathbb Z^{r-1}\oplus\mathbb Z/(3r),&i=2,\\
 \mathbb Z/(3r),&i=3,\\
 \mathbb Z^{r-1}\oplus\mathbb Z/3,&i=4,\\
 0,&i=5.
 \end{cases}
 \tag{\ref{thm:complete-integral-homology}.2}
\]
In particular, the Type B middle torsion is cyclic of order $2r$ exactly
when $r$ is odd; for even $r$ it is the noncyclic group
$\mathbb Z/r\oplus\mathbb Z/2$.
\end{theorem}

\begin{proof}
\medskip\noindent\emph{The middle homology.}
Excision for the pair
$(X_r^\kappa,X_r^{\kappa,\circ})$ gives the exact segment
\[
 H_3(N^\kappa_{0,r}(\rho),M^\kappa_{0,r}(\rho))\longrightarrow
 H_2(X_r^{\kappa,\circ})\longrightarrow H_2(X_r^\kappa)
 \longrightarrow H_2(N^\kappa_{0,r}(\rho),M^\kappa_{0,r}(\rho))
 \longrightarrow H_1(X_r^{\kappa,\circ}).
 \tag{\ref{thm:complete-integral-homology}.3}
\]
The image of the first arrow is $j_*K_2$ by the long exact sequence of the
cusp pair.  Proposition~\ref{prop:open-integral-torsion} and
Lemma~\ref{lem:relative-cusp-free} therefore yield a short exact sequence
\[
 0\longrightarrow G_{\kappa,r}\longrightarrow
 H_2(X_r^\kappa;\mathbb Z)\longrightarrow\mathbb Z^{r-1}
 \longrightarrow0,
\]
where
\[
 G_{B,r}=\mathbb Z/r\oplus\mathbb Z/2,
 \qquad G_{C,r}=\mathbb Z/(3r).
\]
The sequence splits because its quotient is free, proving the formulas in
degree two.  Appendix~\ref{app:relative-free-ledger} records the
complete boundary map after splitting off its free image $\Z[a_0]$.

\medskip\noindent\emph{The remaining degrees.}
The fundamental-group computation gives the degree-one groups.  The
free rank of $H_2$ is $r-1$; Poincar\'e duality gives the same free rank in
degree four.  Since $e(X_r^\kappa)=2r$, the Euler characteristic identity
forces $b_3=0$.  Thus $H_3$ is finite.  The universal coefficient theorem
now gives
$H^3\simeq\operatorname{Ext}(H_2,\mathbb Z)$, and Poincar\'e duality gives
$H_3\simeq H^3$.  Hence $H_3$ has the same isomorphism type as the finite
summand of $H_2$.  Finally,
\[
 H_4\simeq\mathbb Z^{r-1}\oplus\operatorname{Ext}(H_1,\mathbb Z),
 \qquad H_5=0.
\]
Tensoring these groups with $\mathbb Q$ gives the rational cohomology
stated in Corollary~\ref{thm:rational-cohomology-total} below.
\end{proof}

\begin{corollary}[Rational cohomology]
\label{thm:rational-cohomology-total}
For every $r\ge1$,
\[
 H^q(Y_r;\mathbb Q)\simeq H^q(Z_r;\mathbb Q)\simeq
 \begin{cases}
 \mathbb Q,&q=0,6,\\
 \mathbb Q^{r-1},&q=2,4,\\
 0,&q=1,3,5.
 \end{cases}
\]
In particular, $Y_1$ and $Z_1$ are non-K\"ahler rational homology
six-spheres.
\end{corollary}

\begin{proof}
Tensor Theorem~\ref{thm:complete-integral-homology} with $\mathbb Q$
and apply the universal coefficient theorem.  Non-K\"ahlerness was
proved in Theorem~\ref{thm:algebraic-dimension-one}.
\end{proof}

\begin{remark}[Torsion linking]\label{thm:torsion-linking}
For the complex orientation, the torsion linking pairing
\[
 \operatorname{Tor}H_2(X_r^\kappa;\Z)\times
 \operatorname{Tor}H_3(X_r^\kappa;\Z)\longrightarrow\Q/\Z
\]
is perfect by Poincar\'e duality and the universal coefficient theorem.
In the generators of \eqref{eq:reduced-torsion-presentation} and their
linking-dual degree-three generators it is the standard evaluation
pairing: $\operatorname{diag}(1/r,1/2)$ for Type B and $1/(3r)$ for
Type C, omitting the trivial order-one factor.  This specifies dual
bases, not an independently chosen geometric basis of $H_3$.
\end{remark}

\subsection{A simple-connectivity obstruction}

Let $g(x)=Ax+v/m$ act on $\Lambda_\R/\Lambda$, where $A^m=I$
and $v\in\Lambda^A$.  For each prime $p\mid m$, set
\[
 B_p=A^{m/p},\qquad N_{B_p}=I+B_p+\cdots+B_p^{p-1}.
\]
The cyclic action is free if and only if
\begin{equation}\label{eq:prime-subgroup-freeness}
 v\notin N_{B_p}\Lambda\qquad\text{for every prime }p\mid m.
\end{equation}
Indeed, a fixed point of $g^{m/p}$ gives
$(B_p-I)x+v/p=\lambda\in\Lambda$, and application of $N_{B_p}$ gives
$v=N_{B_p}\lambda$.  Conversely this equality implies
$\lambda-v/p\in\ker N_{B_p}=\im(B_p-I)$ over $\R$, so the fixed-point
equation is solvable.  Every nontrivial stabilizer in a finite cyclic
group contains a subgroup of prime order, proving the criterion.

For a rank-one coinvariant completion with two finite multiplicities, the
relation determinant is
\[
 p=m_1m_2\ell_0-m_2\ell_1-m_1\ell_2.
\]
For the three monodromy types it becomes
\[
 p_{34}=12\ell_0-4\ell_1-3\ell_2,\quad
 p_{46}=24\ell_0-6\ell_1-4\ell_2,\quad
 p_{36}=18\ell_0-6\ell_1-3\ell_2.
\]

\begin{corollary}[Simple-connectivity obstruction]

Within the rank-one coinvariant setting considered here,
every Type B completion satisfies $2\mid p_{46}$ and every Type C
completion satisfies $3\mid p_{36}$.  Hence neither type has a simply
connected completion.  For the standard Type B matrices $A_{j,r}$ of
the constructed families, freeness further implies
$p_{46}\equiv2\pmod4$.  The canonical Type A data attain $|p_{34}|=1$.
\end{corollary}

\begin{proof}
The displayed determinant formulas immediately give the two
divisibilities.  A simply connected completion would have trivial
abelianized attaching group, which requires determinant $\pm1$; if the
determinant vanishes, that group is infinite.

For the standard order-four Type B matrix, put
$v_0=(1,-r,-r,0)^{\mathsf T}$.  Direct calculation gives
\[
 \Lambda^{A_{1,r}}=\Z v_0\oplus\Z\delta,\qquad
 (I+A_{1,r}^2)\Lambda=2\Z v_0\oplus\Z\delta.
\]
Write a filling vector as $v=\ell_1v_0+b\delta$.  Criterion
\eqref{eq:prime-subgroup-freeness} for the unique prime $p=2$ forces
$\ell_1$ to be odd.  Therefore
$p_{46}\equiv-2\ell_1\equiv2\pmod4$.
For Type A, $(\ell_0,\ell_1,\ell_2)=(0,1,-1)$ gives $p_{34}=-1$.
\end{proof}

\section{Extensions and additional invariants}\label{sec:extensions}

We first prove that the invariant tensors of the Type B/C monodromies
are unchanged under finite-index restriction.  This leads to
compact examples over arbitrary curves and additional families obtained
by translation gluing.  We then extend the local periodic quotient
construction to arbitrary dimension and compute the additive
Bockstein spectral sequences of the original threefolds.

\subsection{Rigidity under finite-index restriction}

Let $\Lambda=\Z\gamma\oplus\Z u\oplus\Z w\oplus\Z\delta$ be the
homology lattice and put
\[
 \mathcal G_{\kappa,r}=\langle A^\kappa_{1,r},A_{2,r}\rangle\subset GL(\Lambda).
\]
This deck-matrix group is also generated by the forward transports
$\mathsf P_j=A_j^{-1}$.  Thus the finite-index condition below applies
to the geometric local system as well.

\begin{theorem}[Finite-index invariant tensors]\label{thm:finite-index-invariants}
Let $r\geq1$ and $\kappa=B,C$.  For every finite-index subgroup
$H\subset\mathcal G_{\kappa,r}$,
\begin{align}
 (\textstyle\bigwedge^2\Lambda^*)^H
   &=\Z\widetilde q_{\kappa,r},&
 (\textstyle\bigwedge^2\Lambda^*_{\R})^H
   &=\R\widetilde q_{\kappa,r},\label{eq:finite-index-H2}\\
 \Lambda^H&=\Z\delta,&(\Lambda^*)^H&=\Z\gamma^*,\label{eq:finite-index-H1}\\
 \operatorname{Cent}_{GL_4(\Z)}(H)
   &=\{\varepsilon I+nE_{41}:\varepsilon=\pm1,\ n\in\Z\}.
   &&\label{eq:finite-index-centralizer}
\end{align}
Thus finite-index restriction preserves the invariant alternating
line and the integral centralizer.
\end{theorem}

\begin{proof}
Set $s=s_\kappa\in\{2,3\}$ and $h=s-1$.  Direct multiplication of the
homology matrices in Appendix~\ref{app:monodromy-data} gives
\[
 N=A^\kappa_{1,r}A_{2,r}-I
 =\begin{pmatrix}
 0&0&0&0\\0&0&0&0\\-sr&-h&0&0\\-r&-1&0&0
 \end{pmatrix},\qquad N^2=0.
\]
Put $N_i=A_2^iNA_2^{-i}$ for $i=0,1,2$.  These conjugates are
\begin{equation}\label{eq:three-cusp-conjugates}
 N_0=N,\quad
 N_1=\begin{pmatrix}0&0&0&0\\sr&0&h&0\\0&0&0&0\\-r&0&-1&0\end{pmatrix},
 \quad
 N_2=\begin{pmatrix}0&0&0&0\\sr&-h&h&0\\sr&-h&h&0\\-r&1&-1&0\end{pmatrix}.
\end{equation}
For each $i$, some positive power of $I+N_i$ belongs to $H$, because
$H$ has finite index.  No normality of $H$ is needed.

Represent an alternating form by its skew matrix $Q$, with independent
entries $(a,b,c,d,e,f)$ in positions $(12,13,14,23,24,34)$.
If $Q$ is $H$-invariant, then for some $m_i>0$ and every integer $j$,
\[
 (I+jm_iN_i)^{\mathsf T}Q(I+jm_iN_i)=Q.
\]
Comparing coefficients of this polynomial in $j$ gives
$N_i^{\mathsf T}Q+QN_i=0$.  Five entries of these three equations are
\begin{equation}\label{eq:five-invariant-equations}
 \begin{split}
 -hb-c+srd+re&=0,\qquad rf=0,\qquad re=0,\\
 ha-c+srd+rf&=0,\\
 -ha-hb+c-srd+re&=0.
 \end{split}
\end{equation}
Since $r,h\neq0$, the equations force $e=f=0$, $a=-b$, then
$c=srd$ and $a=b=0$.  The remaining form is
$d(sr\,\gamma^*\wedge\delta^*+u^*\wedge w^*)$.
This form is invariant under both original generators by
Lemma~\ref{lem:invariant-class-all-parameters}.  Over $\Z$ its
coefficient $d$ is an integer, proving both statements in
\eqref{eq:finite-index-H2}.

The kernels of $N_0,N_1,N_2$ over $\Q$ are, respectively,
\[
 \langle w,\delta\rangle,\qquad
 \langle u,\delta\rangle,\qquad
 \langle u+w,\delta\rangle.
\]
Their intersection is $\Q\delta$ and the sum of their images is
$\Q\langle u,w,\delta\rangle$.  Intersecting with the integral
lattice or its dual proves \eqref{eq:finite-index-H1}.

Finally, a matrix commuting with $H$ commutes with all $N_i$ and
preserves their kernels.  Preservation of the three planes above forces
it to have the form
\[
 L=\begin{pmatrix}
 a_0&0&0&0\\p&a&0&0\\q&0&a&0\\e&f&g&d
 \end{pmatrix}.
\]
The $(4,2)$ entry for $[L,N_0]$, the $(4,3)$ entry for $[L,N_1]$, and
the $(4,2)$ entry for $[L,N_2]$ give
\[
 a-d=hg=-hf=-h(f+g).
\]
Hence $a=d$ and $f=g=0$.  The $(4,1)$ entries for $N_0,N_1$ give
$p=q=-r(a_0-a)$.  The $(3,1)$ entry for $N_0$ then gives
$r(s-h)(a_0-a)=0$, so $a_0=a$ and $p=q=0$.
Thus $L=aI+eE_{41}$.  Conversely these matrices commute with the
original generators.  Integrality and $\det L=a^4=\pm1$ give exactly
\eqref{eq:finite-index-centralizer}.
\end{proof}

\subsection{Compact examples over arbitrary curves}

\begin{lemma}[An invariant-cone criterion]\label{lem:invariant-cone-criterion}
Let $f:X\to C$ be a proper surjective holomorphic map from a connected
compact complex manifold to a compact Riemann surface, with connected
fibres.  Assume that over $C^\circ=C\setminus S$, for a finite set $S$,
it is a holomorphic submersion with complex $n$-torus fibres.
At each smooth fibre suppose that the real monodromy-invariant space
in $H^2$ contains no nonzero class with a translation-invariant
semipositive $(1,1)$ representative.  Then every divisor is vertical,
$f^*\C(C)$ is the whole meromorphic function field of $X$, and $f$ is
the algebraic reduction.  In particular $a(X)=1$, and $X$ lies outside
Fujiki's class $\mathcal C$.
\end{lemma}

\begin{proof}
Restrictions of global classes are monodromy-invariant.  On an
$n$-torus the translation average of the integration current of a
nonzero effective divisor is a nonzero invariant semipositive form:
its pairing with an invariant K\"ahler form $\omega^{n-1}$ is
strictly positive.  Thus a horizontal divisor contradicts the
hypothesis on a general smooth fibre.  A meromorphic function then
has vertical zero and pole divisors, is constant on a general fibre,
and descends to $C$ by the graph argument in
Theorem~\ref{thm:algebraic-dimension-one}.  A compact Riemann surface
has meromorphic function field of transcendence degree one, proving
the assertions on algebraic reduction.

If $X$ were in class $\mathcal C$, it would carry a K\"ahler current
with analytic singularities; see \cite[Theorems~3.2 and~3.4]{DemaillyPaun}.  Restrict to a smooth fibre
not contained in the singular locus and average the restricted current
under translations.  The result is a positive definite
invariant form representing the restriction of a global real class,
contradicting the stated semipositivity condition.  This is the same
restriction argument as in Corollary~\ref{cor:non-Fujiki} and is
independent of the dimension of the fibre.
\end{proof}

\begin{theorem}[Finite base change away from the special values]
\label{thm:arbitrary-base-curve}
Let $b:C\to\Pone$ be a finite holomorphic map of degree $d$ from a
connected compact Riemann surface, unramified over $\{0,1,\infty\}$.
For either type, $r\geq1$, and $c\in\mathcal U_\kappa$, put
\begin{equation}\label{eq:base-changed-threefold}
 X^\kappa_{r,c;b}=X^\kappa_{r,c}\times_{\Pone}C,
 \qquad f_b:X^\kappa_{r,c;b}\longrightarrow C.
\end{equation}
This is a smooth compact connected complex threefold.  There are $d$
multiple fibres of each prescribed multiplicity and $d$ toroidal
fibres, each with $r$ normalization components isomorphic to $dP_6$.
Moreover
\begin{equation}\label{eq:base-change-geometric-conclusions}
 e(X^\kappa_{r,c;b})=2rd,\qquad
 a(X^\kappa_{r,c;b})=1,\qquad
 X^\kappa_{r,c;b}\notin\mathcal C.
\end{equation}
All divisors are vertical, $f_b$ is the algebraic reduction, and
$\Aut^0(X^\kappa_{r,c;b})\simeq\C^*$.

For fixed $b,\kappa,r$, varying $c$ gives a proper holomorphic
submersion over $\mathcal U_\kappa$ on a fixed oriented smooth
six-manifold.  Over the identity of $C$ the exact equivalence is
\begin{equation}\label{eq:base-change-parameter-period}
 X^\kappa_{r,c;b}\simeq_C X^\kappa_{r,c';b}
 \quad\Longleftrightarrow\quad c-c'\in(6/r)\Z.
\end{equation}
There are also uncountably many unmarked biholomorphism classes in
this family.
\end{theorem}

\begin{proof}
Over the three special values, $b$ is a local biholomorphism, so the
original local completions are simply repeated.  At a branch point
lying over the smooth locus, coordinates for the original submersion
and the base map have the form $f(z_1,z_2,t)=t$ and $t=v^e$.
Their fibre product has coordinates $(z_1,z_2,v)$ and is smooth.
This also proves smoothness jointly with the parameter $c$.
Properness follows by base change; compactness and connectedness
follow from compactness of $C$ and connectedness of the fibres.
The same local argument applied to
Proposition~\ref{prop:simultaneous-parameter-family} gives a proper
holomorphic submersion in $c$, and Ehresmann's theorem applies.
The finite map to the original threefold is branched along the smooth
fibres above the branch values of $b$.

Let $R$ be the finite set of branch values in
$\Pone\setminus\{0,1,\infty\}$.  Away from $R$ the map of bases is a
covering of degree $d$, so its fundamental-group image has index $d$.
Filling back the points of $R$ gives a surjection on base fundamental
groups; the image of the covering subgroup therefore still has finite
index, at most $d$.  Filling their preimages in $C$ gives precisely
the monodromy image of the pulled-back torus family.  Its image in
$\mathcal G_{\kappa,r}$ consequently has finite index.
Theorem~\ref{thm:finite-index-invariants} and
Lemma~\ref{lem:invariant-class-all-parameters} show that its invariant
real line still has signature $(1,1)$ on every smooth fibre.  Neither
sign of a nonzero multiple is semipositive.  Apply
Lemma~\ref{lem:invariant-cone-criterion} to prove all assertions on
divisors and algebraic reduction, including exclusion from class
$\mathcal C$.

All smooth torus fibres and all reduced bielliptic fibres have Euler
number zero.  The Euler characteristic of a multiple fibre is that
of its reduction.  Stratifying the base and using additivity of
compactly supported Euler characteristic, only the $d$ cusp fibres
contribute; each contributes $2r$.  This proves the first formula
in \eqref{eq:base-change-geometric-conclusions}.

A biholomorphism over $C$ has flat integral linear part commuting
with the finite-index monodromy image.  Equation
\eqref{eq:finite-index-centralizer} gives $L=\varepsilon I+nE_{41}$.
At a marked smooth fibre the period identity forces
$R=\varepsilon I_2$ and $\varepsilon r(c-c')=n$ exactly as in
Proposition~\ref{prop:parameter-translations}.  There is an unchanged
order-six root model above each point of $b^{-1}(1)$.  Its norm
lattice is $6\Z\gamma\oplus6\Z\delta$, forcing $\varepsilon=1$ and
$6\mid n$.  Conversely the original isomorphisms for $6\mid n$ pull
back to $C$.  This proves \eqref{eq:base-change-parameter-period}.

For unmarked isomorphisms, the algebraic reduction forces a base
automorphism preserving the three distinguished sets of special
values.  This automorphism group is finite: it preserves the finite
set $b^{-1}\{0,1,\infty\}$ of at least three points, and the
pointwise stabilizer is finite for every genus.  For a fixed such
automorphism and a fixed integral homology map, at most one target
parameter $c'$ can occur for a fixed $c$.  Indeed, at a smooth base
point the two possible target period matrices would have the same
kernel after the same integral change of marking.  They would differ
by a complex-linear left factor; their last two columns are both
$I_2$, making that factor the identity and forcing the two parameters
to coincide.  There are only countably many integral maps, so every
unmarked isomorphism class meets the parameter half-plane in a
countable set.  The half-plane is uncountable.

Finally the vertical translation bundle is
$b^*\mathcal V_\kappa\simeq\mathcal O_C\oplus b^*\mathcal O_{\Pone}(-1)$.
This follows locally at the unchanged root and cusp charts and by
ordinary pullback at all smooth charts, including branch points.
The second summand has negative degree $-d$, hence no sections.
An element of $\Aut^0$ fixes the base and is a translation on the
smooth fibres; as in Corollary~\ref{cor:connected-automorphisms}, its
vector fields are exactly the sections of this bundle.  Their
one-dimensional flow has kernel $\Z\delta$ by
\eqref{eq:finite-index-H1}.  Its connected automorphism group is
therefore $\C/\Z\simeq\C^*$.
\end{proof}

\begin{corollary}[Every compact curve occurs]\label{cor:every-base-curve}
For every connected compact Riemann surface $C$, both types admit
compact threefolds with algebraic reduction onto $C$ as in
Theorem~\ref{thm:arbitrary-base-curve}.  For a fixed choice of $b$ and
$r$, one oriented smooth six-manifold carries uncountably many of
these complex structures.
\end{corollary}

\begin{proof}
Choose a nonconstant meromorphic function on $C$, giving a finite
map to $\Pone$.  Its branch-value set is finite.  Postcomposing with
a generic M\"obius transformation moves that set away from
$0,1,\infty$, so the preceding theorem applies.
\end{proof}

\paragraph{An explicit higher-genus family.}
For $g\geq1$, take the smooth projective completion of
\begin{equation}\label{eq:hyperelliptic-base-example}
 C_g:\quad y^2=\prod_{j=2}^{2g+3}(t-j),\qquad b_g(t,y)=t.
\end{equation}
There are $2g+2$ simple branch points, none at $0$ or $1$;
the even degree gives two unramified points above infinity.
Riemann--Hurwitz gives genus $g$ and $\deg b_g=2$.
Thus \eqref{eq:base-changed-threefold} gives examples over $C_g$
with Euler number $4r$ and two copies of each special-fibre type.
The translation-gluing space below has dimension $2g+1$ for this
example.

\subsection{Translation twists and the additional holomorphic obstruction}

In this subsection the objects are \emph{homology-marked} fibrations,
not strictly marked ones: no zero section is specified.  On the smooth
locus they are torsors under the fixed family of tori.  We also fix
identifications of the component data at every toroidal fibre;
comparison translations must have zero valuation class there.  Thus
the allowed local changes of trivialization are sections of the sheaf
$\mathcal T^0$ of completion-preserving translations.
With all differences interpreted in this fixed translation sheaf,
two overlap cocycles $h_{ab},h'_{ab}$ are equivalent when
\begin{equation}\label{eq:translation-torsor-equivalence}
 h'_{ab}=h_{ab}+\sigma_b-\sigma_a,
 \qquad \sigma_a\in\mathcal T^0(U_a).
\end{equation}
This is the change of local origins by the translations $\sigma_a$.
It does not require these local origins to glue to a specified section.

The vanishing $H^1(\Pone,\mathcal V_\kappa)=0$ is essential to the
sufficiency of the integer winding test used earlier.  On another curve, the obstruction also has a holomorphic term.

\begin{proposition}[Translation gluing on a curve]\label{prop:general-translation-gluing}
Let $C$ be a compact Riemann surface.  Suppose the fixed local
completions have an exact exponential sequence of abelian sheaves
\[
 0\longrightarrow\mathcal L\longrightarrow\mathcal V
 \xrightarrow{\exp}\mathcal T^0\longrightarrow0,
\]
where $\mathcal V$ is a holomorphic vector bundle and $\mathcal T^0$
is the identity-component sheaf of completion-preserving
translations.  There is an exact sequence
\begin{equation}\label{eq:general-translation-obstruction}
 0\longrightarrow
 \frac{H^1(C,\mathcal V)}{\operatorname{im}H^1(C,\mathcal L)}
 \longrightarrow H^1(C,\mathcal T^0)
 \xrightarrow{\partial}H^2(C,\mathcal L)
 \longrightarrow0.
\end{equation}
For local conjugacies with common linear part and matched component
data, the mismatch glues precisely when its class in
$H^1(C,\mathcal T^0)$ vanishes.  A zero integral obstruction
$\partial[\tau]$ is necessary but is sufficient only when the remaining
class in the left term of \eqref{eq:general-translation-obstruction}
is zero.
\end{proposition}

\begin{proof}
The long exact sequence contains
\[
 H^1(C,\mathcal L)\longrightarrow H^1(C,\mathcal V)
 \longrightarrow H^1(C,\mathcal T^0)
 \xrightarrow{\partial}H^2(C,\mathcal L)
 \longrightarrow H^2(C,\mathcal V)=0.
\]
The last equality is coherent-cohomology vanishing above the dimension
of a curve.  Taking the cokernel of the first arrow gives
\eqref{eq:general-translation-obstruction}.  Translation mismatches
are a \v{C}ech one-cocycle; correcting the local maps is exactly expressing
it as a coboundary.  When its image under $\partial$ vanishes, its
lifts to $H^1(C,\mathcal V)$ differ by the displayed integral image.
This proves the gluing criterion.
\end{proof}

\begin{theorem}[Locally trivial twists after base change]
\label{thm:base-change-translation-twists}
Fix $b,C,\kappa,r,c$ as in
Theorem~\ref{thm:arbitrary-base-curve}, and let $g$ be the genus of $C$.
For the pulled-back completions put
$\mathcal L_b=j_{C*}\mathbb L_b$ and let $\mathcal V_b$ be their
translation bundle.  Then
\begin{equation}\label{eq:translation-twist-dimension}
 \mathcal V_b\simeq\mathcal O_C\oplus b^*\mathcal O_{\Pone}(-1),
 \qquad h^1(C,\mathcal V_b)=2g+d-1.
\end{equation}
The homology-marked translation twists with the fixed component
identifications above and zero integral obstruction are classified by
\begin{equation}\label{eq:zero-obstruction-twists}
 H^1(C,\mathcal V_b)/\operatorname{im}H^1(C,\mathcal L_b).
\end{equation}
They are realized by holomorphic gluings with unchanged local models.
A choice of vector-space coordinates on $H^1(C,\mathcal V_b)$ gives
a proper holomorphic deformation of compact smooth threefolds,
locally trivial over the base charts.  Every member has algebraic
dimension one, Euler number $2rd$, and lies outside $\mathcal C$.
If $2g+d-1>0$, there are uncountably many inequivalent twists under
homology-marked translation isomorphisms preserving those component
identifications.

Here $2g+d-1$ is the dimension of the additive space
$H^1(C,\mathcal V_b)$.  The reference fibration supplies the zero class in
\eqref{eq:zero-obstruction-twists}; the quotient is taken using
\eqref{eq:translation-torsor-equivalence} and need not be Hausdorff.
\end{theorem}

\begin{proof}
At the special points the base change is unramified, so the original
exponential sequences pull back without alteration.  At all other
points, including branch points, the fibration is smooth and the
ordinary torus exponential sequence gives the same conclusion.
Thus the translation sequence is exact with $\mathcal V_b=b^*\mathcal
V_\kappa$, proving the splitting in
\eqref{eq:translation-twist-dimension}.  Riemann--Roch \cite{Forster} and the
negative degree of $b^*\mathcal O(-1)$ give
\[
 h^1(C,\mathcal O_C)=g,\qquad
 h^1(C,b^*\mathcal O(-1))=g+d-1.
\]
Apply Proposition~\ref{prop:general-translation-gluing}.

For realization choose a finite Stein cover of $C$ and holomorphic
additive \v{C}ech cocycles $v^{(1)}_{ab},\ldots,v^{(q)}_{ab}$ representing
a basis of $H^1(C,\mathcal V_b)$, where $q=2g+d-1$.
Modify the original transitions on overlaps by
\[
 \exp\left(\sum_{\ell=1}^q z_\ell v^{(\ell)}_{ab}\right),
 \qquad z\in\C^q.
\]
The additive cocycle condition and commutativity of fibre translations
give the exact multiplicative cocycle condition.  The exponential
lands in completion-preserving translations by the stalkwise sequence,
so the modified charts descend to a holomorphic family.  Locally over
$C$, this family is the product of an original smooth local total
space with $\C^q$.  Its total space is therefore smooth, and its map
to $C\times\C^q$ is proper, since properness is local on that target.
Projection to $\C^q$ is a holomorphic submersion by this local product
description.  Since $C$ is compact, the projection is also proper.
Ehresmann's theorem gives a single oriented diffeomorphism type.

These transitions are translations, so the homology monodromy, fibre
complex structures, and all local special fibres remain unchanged.
The invariant-cone argument and the Euler calculation in
Theorem~\ref{thm:arbitrary-base-curve} apply to every twist.  Finally
$\mathcal L_b$ is a constructible integral sheaf on a finitely
triangulated compact curve; its stalks are finitely generated, and
its cohomology is finitely generated.  Its image in the complex vector
space $H^1(C,\mathcal V_b)$ is therefore countable.  If that vector
space has positive dimension its quotient by this image is
uncountable.  Exactness identifies this quotient with the zero-obstruction
translation-twist classes.
\end{proof}

\subsection{Periodic toroidal quotients in arbitrary dimension}

Let $n\geq1$ and let $\mathcal P$ be a locally finite triangulation of
$\R^n$ with vertex set $\Z^n$, invariant under all translations in
$\Z^n$.  Assume every maximal simplex is unimodular.  In particular
there are finitely many simplex orbits.  Coning the simplices at
height one in $\Z^n\oplus\Z$ gives a smooth fan and a separated
analytic toric $(n+1)$-fold $Y_{\mathcal P}$ with height character $t$.
For the polarized fan construction, starting with a positive-definite
symmetric integral bilinear form, see \cite{Mumford} and
\cite[Section~3.1]{EngelFortmanSchreieder}.  Here the holomorphic twist
need not be symmetric and the tori need not carry a chosen polarization;
the proof below establishes the analytic quotient and its properness
directly from the periodic fan and the logarithmic coordinate bounds.

\begin{theorem}[Higher-dimensional periodic quotient]
\label{thm:higher-dimensional-periodic}
Let $S$ be a complex manifold, let $D(t,s)$ be a holomorphic
$n\times n$ matrix on $\Delta_{\epsilon_0}\times S$, and let
$\Gamma=B\Z^n\subset\Z^n$ have finite index $d=|\det B|$.
For every relatively compact open $S'\Subset S$, and sufficiently
small $\epsilon>0$, the action
\begin{equation}\label{eq:n-dimensional-deck-action}
 (x,t,s)\longmapsto
 \bigl(e^{2\pi iD(t,s)\nu}t^\nu x,t,s\bigr),\qquad\nu\in\Gamma,
\end{equation}
is free and properly discontinuous on
$(Y_{\mathcal P})_{<\epsilon}\times S'$.
Its quotient is smooth and Hausdorff, proper over
$\Delta_\epsilon\times S'$, and a holomorphic submersion over $S'$.
The local equation of its map to the disc is
$t=z_0\cdots z_k$ with $0\leq k\leq n$.  The relative canonical
bundle for the smooth projection to $S'$ is trivial.

The smooth fibres are compact complex $n$-tori.  The normalization of
the central fibre has exactly $d$ components, each the smooth complete
toric $n$-fold defined by the star of a vertex of $\mathcal P$.
The quotient to the model with $\Gamma=\Z^n$ is finite \'etale of
degree $d$ and group $\Z^n/\Gamma$.  If $f_n(\mathcal P/\Gamma)$
is the number of maximal simplex orbits, then
\begin{equation}\label{eq:n-dimensional-cusp-euler}
 e(X_0)=f_n(\mathcal P/\Gamma).
\end{equation}
In the period order $[I_n\mid(\tau I_n+D)B]$, where
$t=e^{2\pi i\tau}$, the counterclockwise forward monodromy is
\begin{equation}\label{eq:n-dimensional-monodromy}
 P=\begin{pmatrix}I_n&B\\0&I_n\end{pmatrix},\qquad
 (P-I)^2=0,\quad\rk(P-I)=n,\quad
 [\operatorname{sat}(P-I)\Lambda:(P-I)\Lambda]=d.
\end{equation}
The component interpretation of $d$ is specific to this vertex-lattice
model, as in Remark~\ref{rem:model-dependent-components}.
\end{theorem}

\begin{proof}
A cone over a maximal unimodular simplex has $n+1$ primitive generators
forming a lattice basis.  Its affine chart is $\C^{n+1}$ and the
height character is the product of its coordinates.  Intersections
are charts for common faces, so gluing gives a separated smooth
analytic space.  Local finiteness of the triangulation gives a finite star for every
nonzero cone.  The fan of a
vertex divisor is its star modulo the height-one ray; it is complete
because the triangulation fills a neighbourhood of the vertex in
$\R^n$.  Thus each such divisor is compact.  Translation of simplices
and the holomorphic torus multiplier extend
\eqref{eq:n-dimensional-deck-action} on all charts.

Here is the uniform properness estimate.  For $t\neq0$ put
\[
 y=-\frac{\log|x|}{\log|t|},\qquad
 L=I_n-\frac{2\pi\operatorname{Im}D(t,s)}{\log|t|}.
\]
On a compact parameter set, $D$ is bounded and $L$ tends uniformly to
$I_n$ as $t\to0$.  Choose $\epsilon<1$ so that
$2\pi\sup\|\operatorname{Im}D\|/|\log\epsilon|<1/2$; then
$L,L^{-1}$ are uniformly bounded.  The action changes $y$ to $y-L\nu$.
For a maximal simplex with vertices $v_0,\ldots,v_n$, let $m_i$ be
the basis dual to $(v_i,1)$ and put
$\ell_i(q)=\langle m_i,(q,1)\rangle$.  Its exact chart identities are
\begin{equation}\label{eq:n-dimensional-tropical-chart}
 \log|z_i|=\log|t|\,\ell_i(-y),\qquad
 \sum_{i=0}^n\ell_i=1,\qquad t=\prod_{i=0}^n z_i.
\end{equation}
On a compact chart portion with $|z_i|\leq R$ and $R\geq1$ these imply
\[
 -\frac{\log R}{|\log|t||}\leq\ell_i(-y)
 \leq1+\frac{n\log R}{|\log|t||}.
\]
Hence $y$ is bounded on compact subsets even when $t\to0$.
For two compact sets, $y-L\nu$ can stay in the second bounded set
only for finitely many lattice vectors $\nu$.  On the central fibre,
a point in a toric orbit for a simplex $\sigma$ is sent to the orbit
for $\sigma+\nu$.  Finitely many chart cones meet each compact set,
so again only finitely many translations occur.  This proves proper
discontinuity uniformly on compact parameter sets.  Central freeness
follows because no nonzero translation preserves a bounded simplex;
off the central fibre, a fixed point would give $L\nu=0$ and hence
$\nu=0$.

To prove properness of the quotient, first use the full lattice.
Translate a point so that $L^{-1}y$ lies in $[0,1]^n$.  Its new
$-y$ lies in a fixed bounded subset of $\R^n$, covered by finitely
many simplices.  In a simplex containing it, all the $\ell_i(-y)$
are nonnegative and at most one.  Formula
\eqref{eq:n-dimensional-tropical-chart} gives $|z_i|\leq1$.
The corresponding finite union of closed polydiscs over a closed
subdisc is compact and meets every orbit off the central fibre.
Add one complete vertex divisor for the central orbits.  For
$\Gamma$, add finitely many translates and vertex representatives.
This yields compact sets meeting every orbit over compact subsets
of the base.  Their images prove quotient properness.  Freeness and
proper discontinuity give a smooth Hausdorff quotient and the local
product equations.  The projection to $S'$ is locally the projection
of a product of a smooth toric chart with $S'$.

Writing $t=e^{2\pi i\tau}$ and $x=e^{2\pi iz}$ identifies a punctured
fibre with
\[
 \C^n/\bigl(\Z^n+(\tau I_n+D(t,s))B\Z^n\bigr).
\]
The imaginary part of $\tau I_n+D$ is invertible for small $|t|$;
these periods are a real lattice of rank $2n$.  Thus the fibre is a
compact complex torus.  In the stated period order, counterclockwise
continuation $\tau\mapsto\tau+1$ gives
\eqref{eq:n-dimensional-monodromy}; clockwise continuation gives its
inverse, with $-B$ in the upper-right block.  The image of either
monodromy minus the identity is $B\Z^n$ in the first period block,
with saturation $\Z^n$.

The relative $(n+1)$-form
\begin{equation}\label{eq:local-canonical-volume}
 \Omega=\frac{dx_1}{x_1}\wedge\cdots\wedge
          \frac{dx_n}{x_n}\wedge dt
\end{equation}
is nowhere zero on every toric chart.  Indeed, the height-one
unimodular basis gives
$\Omega=\pm dz_0\wedge\cdots\wedge dz_n$ there.
Under a deck transformation each $d\log x_i$ changes by a multiple
of $dt$ modulo forms from $S'$.  Thus \eqref{eq:local-canonical-volume}
is invariant as a form relative to $S'$ and descends, proving the
canonical-bundle assertion.

Vertex orbits are indexed by $\Z^n/\Gamma$, and no vertex has a
nontrivial stabilizer.  This proves the normalization assertion,
allowing self-identifications along boundary strata.  In the central
fibre a simplex of dimension $k$ contributes the torus orbit
$(\C^*)^{n-k}$.  There are finitely many orbit strata after quotient;
their compactly supported Euler numbers vanish unless $k=n$.
Every maximal-simplex orbit contributes one, proving
\eqref{eq:n-dimensional-cusp-euler}.  Finally, inclusion of the two
free deck groups gives the finite \'etale quotient of degree $d$.
\end{proof}

\begin{corollary}[The standard grid family]\label{cor:kuhn-euler}
For the standard triangulation whose simplices are
\[
 \operatorname{conv}\left(v,\ v+e_{\pi(1)},\
 v+e_{\pi(1)}+e_{\pi(2)},\ \ldots,\ v+\sum_{i=1}^n e_i\right),
 \quad v\in\Z^n,\quad\pi\in\mathfrak S_n,
\]
the central Euler number is $n!d$.  For $n=2$ this triangulation is
integrally equivalent, by reflection of one coordinate, to the
periodic $A_2$ triangulation used earlier.  Thus the formula recovers
$2d$, while $n=1$ gives the $d$-component nodal genus-one model.
\end{corollary}

\begin{proof}
The successive edge differences are a permutation of the standard
basis, proving unimodularity.  Every unit cube has $n!$ maximal
simplices, with compatible triangulations on its faces.  There are
$d$ cube orbits modulo $\Gamma$, so
$f_n(\mathcal P/\Gamma)=n!d$.  The remaining statements follow from
the displayed simplices and Theorem~\ref{thm:higher-dimensional-periodic}.
\end{proof}

\subsection{Mod-prime cohomology and Bockstein spectral sequences}
\label{subsec:bockstein-data}

We deduce the mod-prime cohomology and Bockstein spectral sequences
from the integral homology, using the coefficient-sequence convention
of \cite[Section~3.E]{Hatcher}.

\begin{proposition}[The additive Bockstein spectral sequences]
\label{prop:all-bocksteins}
For the original Type B/C threefolds over $\Pone$, put
\[
 p_B=2,\quad p_C=3,\qquad
 G_{B,r}=\Z/r\oplus\Z/2,\quad G_{C,r}=\Z/(3r).
\]
For a prime $\ell$, define $a_\ell=1$ when $\ell=p_\kappa$ and zero
otherwise.  Let $b_{\ell,j}$ be the multiplicity of $\ell^j$ among
the elementary divisors of the $\ell$-primary part of $G_{\kappa,r}$:
\begin{align}
 b_{\ell,j}^{B}&=\mathbf1_{v_\ell(r)=j}
                  +\mathbf1_{\ell=2,\,j=1},&
 b_{\ell,j}^{C}&=\mathbf1_{v_\ell(3r)=j},\qquad j\geq1.
 \label{eq:bockstein-multiplicities}
\end{align}
Set $b_\ell=\sum_{j\geq1}b_{\ell,j}$.  In degrees $0,\ldots,6$,
\begin{equation}\label{eq:mod-prime-betti}
 \dim_{\mathbb F_\ell}H^q(X^\kappa_{r,c};\mathbb F_\ell)
 =\bigl(1,\ a_\ell,\ r-1+a_\ell+b_\ell,\ 2b_\ell,
     \ r-1+a_\ell+b_\ell,\ a_\ell,\ 1\bigr)_q.
\end{equation}
Let $(E_j,d_j)$ be the singly graded Bockstein spectral sequence from
$0\to\Z\xrightarrow{\ell}\Z\to\mathbb F_\ell\to0$, with
$E_1=H^*(X;\mathbb F_\ell)$ and $d_j$ of degree one.  Its differential
ranks in source degrees $0,\ldots,5$ are
\begin{equation}\label{eq:bockstein-page-ranks}
 \rk(d_j:E_j^q\to E_j^{q+1})
 =\bigl(0,\ a_\ell\mathbf1_{j=1},\ b_{\ell,j},\ b_{\ell,j},
        \ a_\ell\mathbf1_{j=1},\ 0\bigr)_q.
\end{equation}
These formulas determine the entire additive spectral sequence up to
isomorphism; its stable dimensions are
$(1,0,r-1,0,r-1,0,1)$.  They do not specify a preferred geometric
basis or the multiplicative structure.
\end{proposition}

\begin{proof}
The integral universal coefficient theorem and
Theorem~\ref{thm:complete-integral-homology} give
\begin{equation}\label{eq:integral-cohomology-for-bockstein}
 H^q(X;\Z)\simeq
 \begin{cases}
 \Z,&q=0,6,\\
 0,&q=1,\\
 \Z^{r-1}\oplus\Z/p_\kappa,&q=2,\\
 G_{\kappa,r},&q=3,\\
 \Z^{r-1}\oplus G_{\kappa,r},&q=4,\\
 \Z/p_\kappa,&q=5.
 \end{cases}
\end{equation}
The coefficient sequence
\[
 0\to H^q(X;\Z)\otimes\mathbb F_\ell
 \to H^q(X;\mathbb F_\ell)
 \to H^{q+1}(X;\Z)[\ell]\to0
\]
then gives \eqref{eq:mod-prime-betti}.

For completeness, the elementary-complex argument also fixes all
higher pages.  A finite free integral cochain complex decomposes,
by choosing Smith bases in its groups of cocycles and their
boundaries, into free one-term complexes and two-term complexes
$\Z\xrightarrow{m}\Z$, up to chain isomorphism after splitting off
contractible unit summands.  To see compatibility with adjacent
maps, the quotient of each cochain group by its cocycles is its
next boundary group, hence free; choose a splitting, and then Smith
diagonalize the inclusion of boundaries into cocycles.  These choices
give the stated two-term decomposition without changing the next
map.

A two-term summand in degrees $(q,q+1)$ with $v_\ell(m)=j>0$
contributes one copy of $\mathbb F_\ell$ in each of those degrees.
Its Bockstein differentials are zero before page $j$, and on page
$j$ the differential is multiplication by the nonzero residue of
$m/\ell^j$.  Both copies disappear on the next page.  If $\ell$
does not divide $m$, the summand is absent from $E_1$; a free
one-term summand survives forever.  Apply this description to the
torsion summands of \eqref{eq:integral-cohomology-for-bockstein}.
It gives \eqref{eq:bockstein-page-ranks}, and also proves the assertion
about the complete additive spectral sequence.
\end{proof}

\begin{remark}[First versus higher Bocksteins]
For $Y_2$ at $\ell=2$, the middle group is $(\Z/2)^2$, so the two
middle differential ranks on page one are both two.  For $Y_4$,
they are one on page one and one on page two.  Thus divisibility by
$\ell$ alone gives the mod-$\ell$ dimensions but does not give the
first Bockstein rank.  At $r=1$ and $\ell=p_\kappa$, the dimensions
are $(1,1,2,2,2,1,1)$, and the page-one ranks in degrees one through
four are all one; the stable page has only degrees zero and six.
\end{remark}

\paragraph{Further questions.}
The constructions leave open the integral cup-product rings, torsion
linking pairings in independently specified geometric bases, and the
full Kuranishi spaces.  The higher-dimensional quotient theorem is
local over a disc; constructing global higher-dimensional families
requires additional period and gluing data.

\appendix
\section{Monodromy matrices and lattice calculations}\label{app:monodromy-data}

\subsection{The three integral lift calculations}\label{app:lift-calculations}

Write $d_1=(x,y)^{\mathsf T}$ and $d_2=(a,b)^{\mathsf T}$.  The group law
\eqref{eq:jacobi-law} gives the central coordinates listed in the table in
Section~\ref{sec:lifts}.  After imposing the two finite-order relations, the matrices $P-I$,
where $P=T_1T_2$, are as follows; the central entry is denoted by $*$.
\[
\begin{array}{c|c}
A&
\begin{pmatrix}
0&s(a-y)&s(b+x-y)&*\\
0&0&1&-a+b+x\\
0&0&0&-a+y\\
0&0&0&0
\end{pmatrix}\\[5ex]
B&
\begin{pmatrix}
0&s(a-y)&s(a+b+x-y)&*\\
0&0&1&b+x\\
0&0&0&-a+y\\
0&0&0&0
\end{pmatrix}\\[5ex]
C&
\begin{pmatrix}
0&s(a-y)&s(a+b+x-2y)&*\\
0&0&2&-a+b+x\\
0&0&0&-a+y\\
0&0&0&0
\end{pmatrix}.
\end{array}
\]
Thus $(P-I)^2=0$ exactly when $a=y$.  On that locus, in the quotient and
saturated-image bases used in the proof, the induced matrices are
\[
 \overline N_A=\begin{pmatrix}sp_A&q_A\\1&p_A\end{pmatrix},
 \qquad
 \overline N_B=\begin{pmatrix}sp_B&q_B\\1&p_B\end{pmatrix},
 \qquad
 \overline N_C=\begin{pmatrix}sp_C&q_C\\2&p_C\end{pmatrix},
\]
where
\begin{align*}
 p_A&=b+x-y,
 &q_A&=\frac{s}{6}(-3b^2-6by+2x^2-8xy+5y^2),\\
 p_B&=b+x,
 &q_B&=\frac{s}{2}(-2b^2-4by+x^2-2xy-y^2),\\
 p_C&=b+x-y,
 &q_C&=\frac{s}{3}(-3b^2-6by+x^2-4xy+y^2).
\end{align*}
Their determinants are
\[
 \det\overline N_A=\frac{s}{6}(3b+2x-y)^2,
 \quad
 \det\overline N_B=\frac{s}{2}(2b+x+y)^2,
 \quad
 \det\overline N_C=\frac{s}{3}(3b+x+y)^2.
\]

The integral Heisenberg move lattices and their primitive quotient
coordinates are computed in Lemma~\ref{lem:square-zero-cohomology}.

For later use, the standard cusp-index-one Type B matrices are
\begin{equation}\label{eq:typeB-standard-matrices}
 T_1=
 \begin{pmatrix}
 1&-2&0&1\\
 0&0&1&1\\
 0&-1&0&0\\
 0&0&0&1
 \end{pmatrix},
 \qquad
 T_2=
 \begin{pmatrix}
 1&0&0&0\\
 0&0&-1&0\\
 0&1&1&0\\
 0&0&0&1
 \end{pmatrix}.
\end{equation}
The corresponding Type C representative is obtained by replacing the
first row of $T_1$ by $(1,-3,0,1)$ and the middle $2\times2$ block by the
order-three Type C block displayed in Theorem~\ref{thm:core-classification}.

The dual deck matrices are
\[
 A_{1,r}=(T_{1,r}^{-1})^{\mathsf T}
 =\begin{pmatrix}
 1&0&0&0\\
 0&0&1&0\\
 -2r&-1&0&0\\
 -r&0&-1&1
 \end{pmatrix},
\]
\[
 A_{2,r}=(T_{2,r}^{-1})^{\mathsf T}
 =\begin{pmatrix}
 1&0&0&0\\
 0&1&-1&0\\
 0&1&0&0\\
 0&0&0&1
 \end{pmatrix}.
\]
They fix the vectors used in Section~\ref{sec:finite}.  The combined image lattice
\[
 (A_{1,r}-I)\Lambda+(A_{2,r}-I)\Lambda
\]
is saturated of rank three; its Smith normal form is $\diag(1,1,1,0)$.

For the Type C family,
\[
 A^C_{1,r}=((T^C_{1,r})^{-1})^{\mathsf T}
 =\begin{pmatrix}
 1&0&0&0\\
 0&0&1&0\\
 -3r&-1&-1&0\\
 -r&0&-1&1
 \end{pmatrix},
 \qquad
 A^C_{2,r}=A_{2,r}.
\]
They fix the vectors in \eqref{eq:typeC-filling-vectors}; their joint
coinvariant image is again saturated of rank three with Smith normal form
$\diag(1,1,1,0)$.  The cusp deck matrix is
\eqref{eq:typeC-cusp-monodromy}, and its inverse is the clockwise
forward transport \eqref{eq:clockwise-cusp-transport}.

\section{Geometric Mayer--Vietoris data}
\label{app:integral-mv-ledger}

This appendix fixes the geometric generators and signs used in the integral
Mayer--Vietoris calculation of Section~\ref{subsec:complete-integral-homology}.
The matrices represent inclusion and boundary maps in these bases.

\subsection{Orientations and basic cycles}
\label{app:mv-orientations}

Write
\[
 F=\Lambda_{\mathbb R}/\Lambda,
 \qquad
 \Lambda=\mathbb Z\langle\gamma,u,w,\delta\rangle,
\]
and use the marked orientation of the smooth fibre defined by
\[
 \operatorname{vol}_F
 =\gamma^*\wedge u^*\wedge w^*\wedge\delta^*.
\]
For $\Pi=[Z\mid I]$, the real period matrix in the standard complex
coordinate order $(\Re z_1,\Im z_1,\Re z_2,\Im z_2)$ has determinant
$-\det(\Im Z)$.  This is negative on the chosen admissible half-plane.
Thus the marked orientation is opposite to the complex orientation.
Orient each finite quotient so that its covering map preserves the
marked orientation; the intersection matrices below use this convention.
Using the complex orientations instead negates those matrices.  The
coordinate cycles and their evaluation pairings are fixed independently
of this choice.  In Remark~\ref{thm:torsion-linking}, the total
space carries its complex orientation and the degree-three generators
are defined by the resulting Poincar\'e duality and the universal coefficient theorem.
For two ordered basis vectors $\alpha,\beta\in\Lambda$, let
$T_{\alpha\beta}\subset F$ be the oriented coordinate two-torus with
first circle $\alpha$ and second circle $\beta$.  The ordered homology basis
is
\[
 \mathcal E_2=(T_{\gamma u},T_{\gamma w},T_{\gamma\delta},
 T_{uw},T_{u\delta},T_{w\delta}).
\]
At a finite point, $a_j$ is one clockwise coarse meridian; its root
lift is $s(q)=s_0e^{-2\pi iq/m_j}$, $0\leq q\leq1$.  The endpoint
identification uses $\mathsf D_j^{-1}$, so its forward homology return
is $\mathsf P_j=A_j^{-1}$, as computed in
\eqref{eq:clockwise-forward-period}.  Based path products are traversed
from left to right: $a_1a_2$ traverses $a_1$ then $a_2$, has endpoint
deck transformation $g_1g_2$, and has forward return
$\mathsf P_2\mathsf P_1$.  The clockwise cusp loop is
$a_0=(a_1a_2)^{-1}$, with forward return $\mathsf P_0$ and deck matrix
$\mathsf D_0=\mathsf M_0$ from \eqref{eq:clockwise-cusp-transport}.

Let $p_j:F\to S_j^\kappa$ be the free cyclic quotient of the central
torus.  For an invariant direction $\nu$, the finite sweep $s_{j,\nu}$
has circle-first orientation $(\nu,a_j)$; the cusp sweep $\xi_\nu$
has orientation $(\nu,a_0)$.  The closed tube has the
outward-normal-first boundary orientation.  Conjugation in fundamental
group presentations uses $\mathsf D_j$, not $\mathsf P_j$, by the
mapping-torus square in Proposition~\ref{prop:finite-filling-pi1}.

For every two-piece Mayer--Vietoris decomposition we use the convention
$(i_{1*},-i_{2*})$.  Thus the maps from the reference fibre to the two
finite pieces are
\[
 \alpha_q^\kappa=(p_{1*},-p_{2*}).
\]
The chain computations below give
$j_*\xi_\delta=-s_{1,\delta}-s_{2,\delta}$.  Accordingly we choose
$-\xi_\delta$ as the last source generator of the augmented presentation;
in the finite-fibre bases its column is
\[
 b=(0,-1,0,-1)^{\mathsf T}.
\]
Using $\xi_\delta$ instead multiplies this entire column by $-1$.
Equivalently, it right-multiplies the augmented matrix by the unimodular
matrix $\operatorname{diag}(1,1,1,1,1,1,-1)$, so its cokernel is unchanged.

\subsection{The finite quotient groups in degree one}
\label{app:finite-quotient-h1}

Let $a$ also denote the deck generator in the affine fundamental-group
presentation
\[
 \Gamma_{A,v,m}=\langle\Lambda,a\mid
 a\lambda a^{-1}=A\lambda,\ a^m=v\rangle.
\]
In the ordered generators $(\gamma,u,w,\delta,a)$, relation matrices for the
three finite local models are
\[
 R_{B,1}(r)=
 \begin{pmatrix}
 0&0&2r&r&0\\
 0&1&1&0&0\\
 0&-1&1&1&0\\
 0&0&0&0&0\\
 -1&r&r&0&4
 \end{pmatrix},
 \qquad
 R_{C,1}(r)=
 \begin{pmatrix}
 0&0&3r&r&0\\
 0&1&1&0&0\\
 0&-1&2&1&0\\
 0&0&0&0&0\\
 -1&r&r&0&3
 \end{pmatrix},
\]
\[
 R_2=
 \begin{pmatrix}
 0&0&0&0&0\\
 0&0&-1&0&0\\
 0&1&1&0&0\\
 0&0&0&0&0\\
 1&0&0&0&6
 \end{pmatrix}.
\]
The first four rows are the columns of $I-A$ and the last row is the
relation $ma-v=0$.

For the Type B first fibre the relations reduce to
\[
 u+w=0,
 \qquad 2w+\delta=0,
 \qquad \gamma=4a_1.
\]
For the Type C first fibre they reduce to
\[
 u+w=0,
 \qquad 3w+\delta=0,
 \qquad \gamma=3a_1.
\]
At the common order-six fibre they reduce to
\[
 u=w=0,
 \qquad \gamma=-6a_2.
\]
Consequently all three $H_1$ groups are free of rank two.  Choose
\[
 (y_{1,1},y_{1,2})=([a_1],[u]),
 \qquad
 (y_{2,1},y_{2,2})=([a_2],[\delta]).
\]
The maps induced by the quotient covers are then
\[
 p_{1*}^B=
 \begin{pmatrix}4&0&0&0\\0&1&-1&2\end{pmatrix},
 \qquad
 p_{1*}^C=
 \begin{pmatrix}3&0&0&0\\0&1&-1&3\end{pmatrix},
\]
\[
 p_{2*}=
 \begin{pmatrix}-6&0&0&0\\0&0&0&1\end{pmatrix}.
\]
Solving $(p_{1*},-p_{2*})z=0$ gives
\[
 \ker\alpha_1^\kappa=\mathbb Z(u+w)
\]
for both families.  The generator is primitive.

\subsection{Integral pullback lattices in degree two}
\label{app:finite-quotient-h2}

Let $T=(A^{-1})^{\mathsf T}$ be the contragredient deck matrix and
$I=H^2(F;\Z)^T$.  This is also the fixed group of the forward
cohomology transport $T^{-1}$.  Solving $(\bigwedge^2T-I)\alpha=0$ in the ordered
basis $\mathcal E$ gives the following integral bases.  For the first
finite fibre, with $s_B=2$ and $s_C=3$, the coefficient equations are
\[
 x_{u\delta}=x_{w\delta}=0,\qquad
 x_{\gamma w}=-x_{\gamma u},\qquad
 x_{\gamma\delta}=s_\kappa(x_{\gamma u}+r x_{uw}).
\]
Thus the free integer coordinates are $a=x_{uw}$ and
$b=x_{\gamma u}+r x_{uw}$, and
\[
 \omega_{\kappa,1}=-r\gamma^*\wedge u^*
             +r\gamma^*\wedge w^*+u^*\wedge w^*,
\]
\[
 \omega_{\kappa,2}=\gamma^*\wedge u^*
             -\gamma^*\wedge w^*+s_\kappa\gamma^*\wedge\delta^*
\]
form an integral basis of $I$.  At the common order-six fibre, the
coefficients $x_{\gamma u},x_{\gamma w},x_{u\delta},x_{w\delta}$ vanish,
and an integral basis is
\[
 \omega_{2,1}=u^*\wedge w^*,\qquad
 \omega_{2,2}=\gamma^*\wedge\delta^*.
\]

The squares of these basis elements vanish, and their mutual products
in the marked orientation are
\[
 \omega_{B,1}\smile\omega_{B,2}=2\operatorname{vol}_F,\qquad
 \omega_{C,1}\smile\omega_{C,2}=3\operatorname{vol}_F,\qquad
 \omega_{2,1}\smile\omega_{2,2}=\operatorname{vol}_F.
\]
The invariant-lattice discriminants therefore have absolute values
$4,9,1$.  The cover degrees $4,3,6$ and unimodularity downstairs give
indices $2,1,6$, respectively, for the pullback lattices.

For $\alpha=a\omega_{\kappa,1}+b\omega_{\kappa,2}$, the geometric
sweep integral \eqref{eq:finite-pullback-sweep-test} is
\[
 \frac{\alpha(\delta,v_1)}{m_1}
 =\begin{cases}-b/2,&\kappa=B,\\-b,&\kappa=C.\end{cases}
\]
For $\alpha=a\omega_{2,1}+b\omega_{2,2}$ it is $b/6$, since
$v_2=-\gamma$.  Accordingly the pullback lattices are
\[
 \Z\omega_{B,1}\oplus2\Z\omega_{B,2},\qquad
 \Z\omega_{C,1}\oplus\Z\omega_{C,2},\qquad
 \Z\omega_{2,1}\oplus6\Z\omega_{2,2}.
\]
The integrality conditions give containment, and the calculated indices
give equality.  Choosing $-6\omega_{2,2}$ as the second generator in
the last lattice gives exactly $P_1^B(r),P_1^C(r),P_2$ from
Lemma~\ref{lem:finite-H2-bases}.  After division by the covering degrees,
the intersection matrices are
\[
 \begin{pmatrix}0&1\\1&0\end{pmatrix}
 \quad\text{at the first finite fibre},\qquad
 \begin{pmatrix}0&-1\\-1&0\end{pmatrix}
 \quad\text{at the order-six fibre}.
\]

Let $a_{j,1},a_{j,2}$ denote the quotient cohomology bases with these
pullbacks.  Define
\[
 x_{j,1}=p_{j*}T_{uw},
 \qquad
 x_{j,2}=-s_{j,\delta}.
\]
The sweep pairing is
\begin{equation}\label{eq:finite-sweep-pairing-appendix}
 \langle a,s_{j,\nu}\rangle
 =\frac1{m_j}(p_j^*a)(\nu,v_j).
\end{equation}
Indeed, on the root cover the $m_j$ lifted strips form the oriented
parallelogram spanned by $\nu$ and $v_j$.  Formula
\eqref{eq:finite-sweep-pairing-appendix} gives
\[
 \langle a_{j,k},x_{j,\ell}\rangle=\delta_{k\ell}.
\]
Thus the matrices of $p_{j*}$ in degree two are obtained by evaluating the
columns of the pullback matrices on the coordinate tori.  With
$\alpha_2=(p_{1*},-p_{2*})$, this gives exactly
$A_B(r)$ and $A_C(r)$ in
\eqref{eq:finite-attachment-matrix-B}--
\eqref{eq:finite-attachment-matrix-C}.

\subsection{The toroidal boundary and the two sweep classes}
\label{app:cusp-boundary-ledger}

The boundary $M^\kappa_{0,r}(\rho)$ is the mapping torus with clockwise
forward return $\mathsf P^\kappa_{0,r}$.  In the ordered basis
\[
 (\gamma\wedge u,\gamma\wedge w,\gamma\wedge\delta,
 u\wedge w,u\wedge\delta,w\wedge\delta)
\]
the matrices $C_\kappa(r)=\bigwedge^2\mathsf P^\kappa_{0,r}-I$ are
\[
 C_B(r)=
 \begin{pmatrix}
0&0&0&0&0&0\\-1&0&0&0&0&0\\-1&0&0&0&0&0\\
2r&0&0&0&0&0\\r&0&0&0&0&0\\r&r&-2r&1&-1&0
\end{pmatrix},
\]
\[
 C_C(r)=
 \begin{pmatrix}
0&0&0&0&0&0\\-2&0&0&0&0&0\\-1&0&0&0&0&0\\
3r&0&0&0&0&0\\r&0&0&0&0&0\\r&r&-3r&1&-2&0
\end{pmatrix}.
\]
Both have rank two.  For $C_B(r)$ the minor in rows $2,6$ and columns
$1,4$ is $-1$; for $C_C(r)$ the minor in rows $3,6$ and columns $1,4$
is $-1$.  Hence their images are saturated and
\[
 \operatorname{SNF}C_B(r)=\operatorname{SNF}C_C(r)
 =\operatorname{diag}(1,1,0,0,0,0).
\]
These are the forward-transport matrices, not the exterior squares of
the deck matrices.  They have the same image lattices as the latter:
\[
 \bigwedge^2\mathsf P_0-I
 =(\bigwedge^2\mathsf D_0-I)
   \bigl(- (\bigwedge^2\mathsf D_0)^{-1}\bigr).
\]
The right factor is unimodular.  In the Wang sequence oriented by the
clockwise, circle-first cut, $\xi_w\mapsto w$ and
$\xi_\delta\mapsto\delta$.
These minors give the saturation used in
Lemma~\ref{lem:cusp-sweeps-integral}; the integral fillings are the
maps \eqref{eq:explicit-sweep-filling}.  We now specify the chains for
the other boundary inclusion.

Choose a based pair of pants $P$ with boundary loops
$a_0=(a_1a_2)^{-1}$, $a_1$, and $a_2$.  Cut along arcs from a reference
base point to the two finite boundaries.  Thickening the two resulting
lobes and attaching the finite fillings gives pieces $V_1,V_2$ of
$X_r^{\kappa,\circ}$; each retracts onto $S_j^\kappa$, and their
intersection retracts onto the reference fibre $F$.
On a cut lobe use the flat real trivialization of the torus bundle.
Transport $c_\nu(\theta)=[\theta\nu]$ along the based clockwise path
$a_j$, including its incoming and outgoing basing arcs.  Along the
root arc the flat coordinate remains $\nu$; at its endpoint $(gs_0,\nu)$
the quotient identifies it with $(s_0,A_j^{-1}\nu)$.
Both ends are therefore in the same reference fibre, and the terminal
circle is $c_{\mathsf P_j\nu}$, not $c_{A_j\nu}$.
The product boundary of the circle-first cylinder is
\begin{equation}\label{eq:finite-cylinder-boundary}
 \partial C_j(\nu)=c_\nu-c_{\mathsf P_j\nu},\qquad
 \mathsf P_j=A_j^{-1}.
\end{equation}
In particular, $\mathsf P_2w=u+w$, whereas $\mathsf P_1w=-u$ and
$\mathsf P_2(-u)=w$ for both types.

Since $a_1a_2$ traverses $a_1$ before $a_2$, its $w$-sweep has
successive fibre circles $c_w,c_{-u},c_w$.  Its two cylinders are
$C_1(w)$ and $C_2(-u)$, whose ends cancel as integral chains.
Reversing this traversal gives the cusp sweep.  Its two pieces are
\begin{equation}\label{eq:cusp-cut-cylinders}
 \begin{aligned}
 z_1&=-C_1(w),&z_2&=-C_2(-u),\\
 \partial z_1&=c_{-u}-c_w,&\partial z_2&=c_w-c_{-u}.
 \end{aligned}
\end{equation}
Here $z_j$ lies in $V_j$.  A based homotopy of the cusp boundary to
$(a_1a_2)^{-1}$, applied to its flat circle family through the zero
section, supplies a prism three-chain between the two representatives.
Subdivision realizes these cylinders and prisms as integral singular
chains.  With $\alpha=(i_{1*},-i_{2*})$ and
$\partial_{\rm MV}[z_1+z_2]=[\partial z_1]$, we obtain
\begin{equation}\label{eq:forward-MV-sweep-boundary}
 \partial_{\rm MV}(j_*\xi_w)=-(u+w).
\end{equation}
The vector has coefficients $-1,-1$ in the $(u,w)$ positions and is
primitive.  No rescaling of the Wang generator $\xi_w\mapsto w$ has
been made.

For $\delta$, the invariant circle acts on the open torus family.
Sweeping its zero section over $P$ gives a map
$S^1\times P\to X_r^{\kappa,\circ}$ and a three-chain whose boundary
is the negative sum of the three circle-first boundary sweeps.
Under the finite collar map, the zero section on
$s(t)=s_0e^{-2\pi it/m_j}$ has coordinate
\[
 -\ell(s_0)\Pi(s(t))v_j+\frac{t}{m_j}\Pi(s(t))v_j.
\]
The first term is a holomorphic equivariant translation on the root
disc and can be removed by a translation isotopy.  Retraction to the
central affine quotient therefore identifies its $\delta$-sweep with
\begin{equation}\label{eq:finite-delta-sweep-square}
 (\theta,t)\longmapsto
 \left[\theta\delta+\frac{t}{m_j}v_j\right]\in S_j^\kappa,
 \qquad(\theta,t)\in[0,1]^2.
\end{equation}
The $\theta$-edges agree by the period $\delta$, and the $t$-edges
agree by the affine generator, since $A_j\delta=\delta$.
This is one coarse meridian, not $m_j$ meridians.  Its pairing with
$a\in H^2(S_j^\kappa;\mathbb Z)$ is
$(p_j^*a)(\delta,v_j)/m_j$, explaining the factor in
\eqref{eq:finite-sweep-pairing-appendix}.
The three boundary sweeps hence give
$j_*\xi_\delta=-s_{1,\delta}-s_{2,\delta}$ with the same signs as
\eqref{eq:oriented-delta-sweep-relation}.
In the stated bases the complete inclusion data are
\[
 \partial(j_*\xi_w)=-(u+w),\qquad \partial(j_*\xi_\delta)=0,\qquad
 j_*\xi_\delta\equiv(0,1,0,1)^{\mathsf T}=-b\pmod{\im\alpha_2}.
\]

\subsection{The augmented Mayer--Vietoris matrices}
\label{app:mv-augmented-matrices}
The source basis is
\[
 (T_{\gamma u},T_{\gamma w},T_{\gamma\delta},T_{uw},
 T_{u\delta},T_{w\delta},-\xi_\delta),
\]
and the target basis is $(x_{1,1},x_{1,2},x_{2,1},x_{2,2})$.
The first six columns are the inclusion map $\alpha_2^\kappa$;
the last is the representative $b$ of $-j_*\xi_\delta$ obtained from
\eqref{eq:finite-delta-sweep-square}.
For $(a_\kappa,m_1)=(2,4)$ or $(1,3)$ the two matrices are
\begin{equation}\label{eq:augmented-MV-matrix}
 D_\kappa(r)=
 \begin{pmatrix}
 -r&r&0&1&0&0&0\\
 a_\kappa&-a_\kappa&m_1&0&0&0&-1\\
 0&0&0&-1&0&0&0\\
 0&0&6&0&0&0&-1
 \end{pmatrix}.
\end{equation}
Eliminating $x_3,x_4$ as in
Proposition~\ref{prop:open-integral-torsion} gives
\eqref{eq:reduced-torsion-presentation}.  Its Smith invariants, with
the two eliminated unit factors restored, are
\[
 \operatorname{SNF}D_B(r)=
 \begin{cases}
 \operatorname{diag}(1,1,1,2r),&r\text{ odd},\\
 \operatorname{diag}(1,1,2,r),&r\text{ even},
 \end{cases}
 \qquad
 \operatorname{SNF}D_C(r)=\operatorname{diag}(1,1,1,3r).
\]
In particular, the Type B group at $r=2$ is $(\Z/2)^2$, whereas the
Type C group at that index is $\Z/6$.

\subsection{The free relative contribution}
\label{app:relative-free-ledger}

The controlled collapse gives
\[
 H_2(N^\kappa_{0,r},M^\kappa_{0,r};\mathbb Z)
 \cong H^4(W_r;\mathbb Z)\cong\mathbb Z^r.
\]
Let $d_0$ be the relative class of the extended zero-section disc with
the negative of its complex orientation, and complete it to a basis
$d_0,d_1,\ldots,d_{r-1}$.  Its boundary is the clockwise cusp meridian
$a_0$, consistently with Lemma~\ref{lem:relative-cusp-free}.  The homomorphisms
\[
 \Phi_B(z,x,y)=12z+3x-2y,
 \qquad
 \Phi_C(z,x,y)=6z+2x-y
\]
induce primitive maps from $H_1(X_r^{\kappa,\circ})$ to its free quotient,
and for the canonical collar data
\[
 \Phi_B(a_0)=\Phi_C(a_0)=-1.
\]
By Lemma~\ref{lem:relative-cusp-free}, the full boundary image is
exactly $\Z[a_0]$.  The equation $\Phi_\kappa(a_0)=-1$ shows both
that $d_0$ is primitive and that
\[
 H_1(X_r^{\kappa,\circ})
 =\Z[a_0]\oplus\operatorname{Tor}H_1(X_r^{\kappa,\circ}).
\]
For $i>0$, write $\partial d_i=n_i[a_0]$ and replace $d_i$ by
$d_i-n_id_0$.  In these domain and target bases the complete map is
\[
 \Z^r\longrightarrow\Z\oplus\Z/p_\kappa,\qquad
 (n_0,\ldots,n_{r-1})\longmapsto(n_0,\overline0),
 \qquad p_B=2,\quad p_C=3.
\]
Its projection to the free quotient has row $(1,0,\ldots,0)$, its
torsion coordinate is zero, and its kernel is $\Z^{r-1}$.  The exact sequence for the
second Mayer--Vietoris attachment becomes
\[
 0\longrightarrow G_{\kappa,r}
 \longrightarrow H_2(X_r^\kappa;\mathbb Z)
 \longrightarrow\mathbb Z^{r-1}\longrightarrow0,
\]
where
\[
 G_{B,r}=\mathbb Z/r\oplus\mathbb Z/2,
 \qquad
 G_{C,r}=\mathbb Z/(3r).
\]
The quotient is free, so the sequence splits.  Poincar\'e duality and the
universal coefficient theorem then give the remaining homology groups in
Theorem~\ref{thm:complete-integral-homology}.

\section{N\'eron--Severi groups of fixed fibres}\label{app:fibre-NS}
Fix $\kappa\in\{B,C\}$ and $r\geq1$.  The parameter domain and marking
are those of Sections~\ref{sec:period} and \ref{sec:affine-moduli}.

\begin{lemma}[Independence of the first two period coordinates]

For either realized monodromy type, the holomorphic functions
\[
 1,\qquad \tau_\kappa,\qquad \mu_\kappa
\]
on the orbifold universal cover are linearly independent over $\C$.
Consequently there is a smooth base point $t_*$ for which
$1,\tau_\kappa(t_*),\mu_\kappa(t_*)$ are linearly independent over $\Q$.
\end{lemma}

\begin{proof}
Suppose $A\mu_\kappa+B\tau_\kappa+C=0$ identically.  The second elliptic
generator has the common transformation law
\[
 \tau_\kappa\longmapsto-\frac1{\tau_\kappa+1},
 \qquad
 \mu_\kappa\longmapsto\frac{\mu_\kappa}{\tau_\kappa+1}.
\]
Applying the relation after this transformation and multiplying by
$\tau_\kappa+1$ gives
\[
 A\mu_\kappa+C\tau_\kappa+(C-B)=0.
\]
Subtracting the original relation yields
$(C-B)\tau_\kappa-B=0$.  Since $\tau_\kappa$ is nonconstant, one obtains
$B=C=0$, and then $A=0$ because $\mu_\kappa$ is not identically zero.

For every nonzero $(a,b,c)\in\Z^3$, the holomorphic function
$a\mu_\kappa+b\tau_\kappa+c$ is therefore nonzero and has a discrete zero
set.  The complement of the countable union of those zero sets contains a
smooth point $t_*$ with the stated rational independence.
\end{proof}

Fix such a point $t_*$ and abbreviate
\[
 \tau_*=\tau_\kappa(t_*),\qquad
 \mu_*=\mu_\kappa(t_*),\qquad
 \beta_*=\beta_{\kappa,0}(t_*).
\]
We use $\lambda_{\kappa,r}=s_\kappa r$ and
$\widetilde q_{\kappa,r}$ from \eqref{eq:distinguished-NS-class}.

\begin{proposition}[Very general N\'eron--Severi group]
\label{prop:very-general-NS}
There is a countable subset
$\Sigma_{\kappa,r}\subset\mathcal U_\kappa$ such that, for every
$c\notin\Sigma_{\kappa,r}$, the fixed smooth torus fibre over $t_*$
satisfies
\[
 \NS(F^\kappa_{r,c,t_*})=\Z\widetilde q_{\kappa,r}.
\]
The Hermitian form associated with $\widetilde q_{\kappa,r}$ has
signature $(1,1)$.  Consequently $F^\kappa_{r,c,t_*}$ contains no
nonzero effective divisor and has algebraic dimension zero.
\end{proposition}

\begin{proof}
Write the marked period matrix at $t_*$ as
\[
 \Pi_c=[Z_c\mid I_2],\qquad
 Z_c=\begin{pmatrix}
 \lambda_{\kappa,r}\mu_*&\tau_*\\
 r(\beta_*+c)&\mu_*
 \end{pmatrix},
\]
and put
\[
 K_c=\begin{pmatrix}I_2\\-Z_c\end{pmatrix}.
\]
For an integral cohomological alternating form
$E=(e_{ij})\in\bigwedge^2\Lambda^*$, the Appell--Humbert type-$(1,1)$
condition is
\[
 K_c^{\mathsf T}EK_c=0.
\]
When $E$ is nondegenerate this is equivalently
$\Pi_cE^{-1}\Pi_c^{\mathsf T}=0$; these are the covariant and
contravariant forms of the same Riemann relation.

Writing the six upper-triangular entries of $E$ as
$e_{12},e_{13},e_{14},e_{23},e_{24},e_{34}$, the unique scalar equation is
\begin{align}\label{eq:hodge-locus-affine}
 \Psi_E(c)={}&e_{12}-e_{13}\tau_*-e_{14}\mu_*
 +\lambda_{\kappa,r}e_{23}\mu_*\\
 &+r e_{24}(\beta_*+c)
 -r e_{34}\tau_*(\beta_*+c)
 +\lambda_{\kappa,r}e_{34}\mu_*^2=0.\notag
\end{align}
If this equation is identically zero as a function of $c$, its linear
coefficient gives $e_{24}-\tau_*e_{34}=0$.  The rational independence of
$1$ and $\tau_*$ forces $e_{24}=e_{34}=0$.  The remaining equation and
the independence of $1,\tau_*,\mu_*$ give
\[
 e_{12}=e_{13}=0,\qquad e_{14}=\lambda_{\kappa,r}e_{23}.
\]
Thus the integral classes whose entire parameter line is a Hodge locus
are exactly the multiples of \eqref{eq:distinguished-NS-class}.  Every
other integral class contributes at most one parameter value through
\eqref{eq:hodge-locus-affine}; removing the resulting countable set gives
the asserted N\'eron--Severi group.

Lemma~\ref{lem:invariant-class-all-parameters} gives signature $(1,1)$
for $\widetilde q_{\kappa,r}$ at every admissible parameter.
Since every integral $(1,1)$ class is now a multiple of this generator,
Lemma~\ref{lem:effective-torus-class} excludes every nonzero effective
divisor.  A nonconstant meromorphic function on a compact torus would
have a nonzero polar divisor, so the algebraic dimension is zero.
\end{proof}

\section{Nearby cycles and local invariant cycles}\label{app:nearby-cycles}
Fix $\kappa\in\{B,C\}$ and $r\geq1$, and suppress $\kappa$ from the notation.
This appendix describes specialization at the toroidal fibre using
the nearby-cycle formalism of \cite{SGA7,Dimca} and the ordered branch
spaces of Proposition~\ref{prop:branch-normalization-resolution}.
Write $T_i$ for the contragredient deck action on fibre cohomology,
with exterior powers understood in degree $q$.  The clockwise forward
cohomology transport is $T_i^{-1}$; its invariant subspace is the same.
Accordingly the invariant-cycle targets below are independent of this
choice of inverse, whereas all geometric sweeps use the clockwise
forward convention of \eqref{eq:clockwise-forward-period}.

\subsection{The nearby-cycle sheaves}

\begin{proposition}[Nearby-cycle sheaves on the ordered branch complex]
\label{prop:branch-nearby-descent}
Let $\psi_r=R\psi_f\Z$ be the nearby-cycle complex of the toroidal
neighbourhood $N^\kappa_{0,r}(\rho)$.  Its cohomology sheaves descend from
the ordered-branch cover and satisfy
\[
 \mathcal H^0\psi_r=\Z_{W_r},
 \qquad
 \mathcal H^2\psi_r=\epsilon_{2*}\Z_{W_r^{\langle2\rangle}}.
\]
The sheaf $\mathcal H^1\psi_r=\mathcal R^1_r$ fits into
\begin{equation}\label{eq:global-V-resolution}
 0\longrightarrow\mathcal R^1_r\longrightarrow
 \epsilon_{1*}\Z_{W_r^{\langle1\rangle}}
 \xrightarrow{\delta_{\mathrm{tr}}}
 \epsilon_{2*}\Z_{W_r^{\langle2\rangle}}
 \longrightarrow0.
\end{equation}
At either oriented triple point the stalk map is
\[
 (n_1,n_2,n_3)\longmapsto n_1-n_2+n_3.
\]
For $r=1$, the induced map on global sections is
\[
 \Z^3\longrightarrow\Z^2,
 \qquad(a,b,c)\longmapsto(s,s),
 \quad s=a-b+c,
\]
and consequently
\begin{equation}\label{eq:V-cohomology}
 H^a(W_1,\mathcal R^1_1)\cong
 \begin{cases}
 \Z^2,&a=0,\\
 \Z,&a=1,\\
 \Z^3,&a=2,\\
 0,&a\ge3.
 \end{cases}
\end{equation}
\end{proposition}

\begin{proof}
On the periodic cover the local equation at a point with ordered
branch set $I$ is $\prod_{i\in I}z_i=t_c$.  Its Milnor fibre retracts
onto the angular torus whose homology and cohomology lattices are
\[
 H_1(M_I;\Z)=\ker\!\left(\Z^I\xrightarrow{\sum}\Z\right),
 \qquad
 Q_I:=H^1(M_I;\Z)=\Z^I/\Z(1,\ldots,1).
\]
The second identity follows by dualizing the first exact sequence.
Consequently
\[
 (R^q\psi\Z)_x\cong\bigwedge^q Q_I.
\]
For $J\subset I$, generization to the stratum with branch set $J$ is
induced by the coordinate restriction $Q_I\to Q_J$ and its exterior
powers.  This map is well defined because a diagonal vector restricts
to a diagonal vector.

For three ordered branches $0,1,2$, the map to the three double-branch
stalks is
\[
 [a_0,a_1,a_2]\longmapsto
 (a_1-a_0,\ a_2-a_0,\ a_2-a_1).
\]
It identifies $Q_I$ integrally with the kernel of
$(n_1,n_2,n_3)\mapsto n_1-n_2+n_3$: given a vector in that kernel,
take $[a_0,a_1,a_2]=[0,n_1,n_2]$.  Thus the stalk sequence is exact,
with no finite-index change of lattice.  Also $\bigwedge^2Q_I\cong\Z$
at a triple point and is zero on the other strata.

These coordinate restrictions are equivariant under the ordered deck
translations, so they give \eqref{eq:global-V-resolution} on $W_r$.
For $r=1$ each of the three double curves is a copy of $\Pone$ joining
the same two triple points.  Choose the incidence generators so that
the alternating boundary at both triple points is $(1,-1,1)$; changing
one triple-point generator by a sign makes these choices compatible.
The map on global sections is then $(a,b,c)\mapsto(s,s)$ with
$s=a-b+c$.  Its kernel is $\Z^2$ and its cokernel is $\Z$.
The long exact sequence, together with $H^2(\Pone;\Z)=\Z$, gives
\eqref{eq:V-cohomology}.
\end{proof}

\subsection{Specialization at the toroidal fibre}

\begin{lemma}[Supported specialization and the angular torus]
\label{lem:rational-supported-specialization}
Let $K\simeq(S^1)^2$ be a compact angular torus orbit in the smooth
stratum of the index-one fibre $W_1$, and let $K_t\subset F_t$ be its
parallel transport to a nearby fibre.  There is a commutative diagram
\[
\begin{CD}
 H^1(K;\Q) @>{\operatorname{Th}_0}>>
 H^3(W_1,W_1\setminus K;\Q) @>>> H^3(W_1;\Q)\\
 @V{\operatorname{PT}}V{\simeq}V
 @V{\operatorname{sp}_K}V{\simeq}V
 @VV{\operatorname{sp}}V\\
 H^1(K_t;\Q) @>{\operatorname{Th}_t}>>
 H^3(F_t,F_t\setminus K_t;\Q) @>>> H^3(F_t;\Q).
\end{CD}
\]
The upper Gysin map and the lower Gysin map are injective of rank two; the
lower image is primitive in integral cohomology.  Consequently
specialization $H^3(W_1;\Q)\to H^3(F_t;\Q)$ is injective.
\end{lemma}

\begin{proof}
Choose a relatively compact tubular neighbourhood $U_0$ of $K$ in
the smooth stratum of $W_1$, invariant under the compact angular torus.
Average a horizontal lift under that torus.  After shrinking the base
disc, its transport gives an equivariant product tube around $K$ in
the submersion locus,
with slices $U_0$ and $U_t$ and transported compact torus $K_t$.
Fix a simply connected sector in the punctured disc for this transport
and for the global specialization map.

On a holomorphic submersion neighbourhood the nearby-cycle complex is
the constant sheaf in degree zero, and the canonical morphism
$\mathbb Z_{W_1}\to R\psi_f\mathbb Z$ restricts to the ordinary
identification supplied by the product tube.  Apply cohomology with
supports in $K$ to this morphism.  Excision in the tube and its slice
identifies the resulting local map with
\[
 H^3(U_0,U_0\setminus K;\mathbb Z)
 \longrightarrow H^3(U_t,U_t\setminus K_t;\mathbb Z).
\]
Forgetting supports is natural for that same sheaf morphism.  Under
the proper nearby-fibre identification
$\mathbb H^3(W_1,R\psi_f\mathbb Z)\simeq H^3(F_t;\mathbb Z)$,
its global map is precisely specialization.  Thus the local map is
compatible with the right-hand vertical arrow in the displayed diagram,
not merely with an abstract isomorphism of the two local groups.
Transport preserves the oriented normal bundles of $K$ and $K_t$,
so it also identifies their Thom classes.  Excision followed by
forgetting supports gives the asserted diagram, already over $\mathbb Z$;
tensoring with $\mathbb Q$ gives its stated form.

For the lower row, the lattice of $K_t$ is the primitive direct summand
$\Z\langle w,\delta\rangle\subset\Lambda$.  A complementary subtorus
exists, and the K\"unneth decomposition identifies the Gysin map with a
primitive rank-two summand of $H^3(F_t;\Z)$.

For the upper row put $U=W_1\setminus\operatorname{Sing}(W_1)$ and
$\Sigma=W_1\setminus U$.  The normalization identifies
$U\simeq(\C^*)^2$, whereas $\Sigma$ is a compact union of double
curves and triple points, of real dimension at most two.  Thus
$H^3(\Sigma;\Z)=0$.  The open--closed cohomology sequence, using the
compactness of $W_1$, gives
\[
 H^2(\Sigma;\Z)\longrightarrow H_c^3(U;\Z)
 \longrightarrow H^3(W_1;\Z)\longrightarrow0.
\]
Here $H_c^3(U;\Z)\simeq\Z^2$ because
$U\simeq K\times\R^2$, and $H^3(W_1;\Z)\simeq\Z^2$ by
Theorem~\ref{thm:cusp-cohomology}.  The displayed surjection between
free groups of the same finite rank is an isomorphism.  The Thom map
for $K\subset K\times\R^2$, followed by extension of supports,
identifies $H^1(K;\Z)$ with $H_c^3(U;\Z)$.  Its composition with the
open--closed map is exactly the upper Gysin map.  This proves that the upper Gysin map is an isomorphism.  Tensoring with $\Q$ and using the commutative diagram
and the injective lower Gysin map proves the specialization assertion.
\end{proof}

\begin{proposition}[Specialization to invariant cycles]
\label{prop:rational-index-one-specialization}
For the index-one Type B or Type C toroidal cusp, specialization induces
an isomorphism
\[
 H^q(W_1;\Q)\xrightarrow{\ \sim\ }
 H^q(F_t;\Q)^{T_0}
 \qquad(0\le q\le4).
\]
For index $r$, specialization is surjective onto the
monodromy-invariant subspace of $H^q(F_t;\Q)$; its kernel is
$\Q^{r-1}$ in degrees $2$ and $4$ and zero in degrees $0,1,3$.
\end{proposition}

\begin{proof}
For $r=1$, Proposition~\ref{prop:branch-nearby-descent} and
Theorem~\ref{thm:cusp-cohomology} give
\[
 E_2^{a,b}=H^a(W_1,\mathcal H^b\psi_1)
 \Longrightarrow H^{a+b}(F_t;\Q)
\]
with rank table
\[
\begin{array}{c|ccccc}
 b=2&2&0&0&0&0\\
 b=1&2&1&3&0&0\\
 b=0&1&2&4&2&1\\ \hline
 &a=0&a=1&a=2&a=3&a=4.
\end{array}
\]
Write
\[
 \alpha=\operatorname{rk}d_2^{0,1},\quad
 \beta=\operatorname{rk}d_2^{1,1},\quad
 \gamma=\operatorname{rk}d_2^{2,1},\quad
 \delta=\operatorname{rk}d_2^{0,2},\quad
 \varepsilon=\operatorname{rk}d_3^{0,2}.
\]
These are all possible nonzero differentials.  Since the nearby fibre is
a four-torus with Betti numbers $(1,4,6,4,1)$, comparison in total degrees
one and four gives $\alpha=\gamma=0$, while total degree two gives
\[
 \beta+\delta+\varepsilon=1.
\]
The bottom edge in total degree three is the specialization map
$H^3(W_1;\Q)\to H^3(F_t;\Q)$.  Lemma~\ref{lem:rational-supported-specialization}
shows that it has rank two.  Hence no differential can hit
$E_2^{3,0}$: $\beta=\varepsilon=0$, and therefore $\delta=1$.
All bottom-row classes survive.  Thus specialization is injective in every
degree.

A class extending across the disc is monodromy invariant.  Direct
exterior-power calculation for the cusp monodromy with square-zero logarithm gives
invariant ranks $(1,2,4,2,1)$, exactly the ranks of $H^q(W_1;\Q)$.
Therefore the specialization image is the full invariant subspace.

For general $r$, the cyclic cover of cusp neighbourhoods restricts to a
degree-$r$ cover $W_r\to W_1$ and a degree-$r$ isogeny of nearby tori.
Transfer makes pullback injective on $H^*(W_1;\Q)$ and an isomorphism on
the rational cohomology of nearby tori.  Comparing ranks with
Theorem~\ref{thm:cusp-cohomology} gives the stated kernels; in degree four
the kernel is the sum-zero subspace of the $r$ component classes.
\end{proof}

\begin{lemma}[Rational cohomology of a finite fibre]
\label{lem:rational-finite-transfer}
At either multiple fibre, pullback along the free cyclic quotient of the
covering torus identifies
\[
 H^q(S_j;\mathbb Q)\simeq H^q(F;\mathbb Q)^{T_j}.
\]
Hence finite multiple fibres produce no point-supported kernel or
cokernel in the rational cohomology sheaves $R^qf_*\mathbb Q$.
\end{lemma}

\begin{proof}
For a free quotient $q:T\to T/G$, transfer satisfies
$q_*q^*=|G|\,\mathrm{id}$ and
$q^*q_*=\sum_{g\in G}g^*$.  Over $\mathbb Q$, the image of $q^*$ is
therefore precisely the invariant subspace.  Proper base change gives the
stalk statement.
\end{proof}

\begin{proposition}[Constructible rational cohomology sheaves]

Let $j:B^\circ\hookrightarrow B$ be the complement of the three special
points and put $\mathcal L_r^q=R^q(f|_{B^\circ})_*\mathbb Q$.  There are
natural exact sequences
\[
 0\longrightarrow\mathcal K_r^q\longrightarrow R^qf_*\mathbb Q
 \longrightarrow j_*\mathcal L_r^q\longrightarrow0,
\]
where
\[
 \mathcal K_r^q\simeq
 \begin{cases}
 \mathbb Q_{p_0}^{r-1},&q=2,4,\\
 0,&\text{otherwise}.
 \end{cases}
\]
Moreover
\[
 H^p(B,j_*\mathcal L_r^q)\simeq
 \begin{cases}
 \mathbb Q,&p=0,2,\\
 0,&p=1
 \end{cases}
 \qquad(0\le q\le4).
\]
\end{proposition}

\begin{proof}
The stalk statement is Proposition~\ref{prop:rational-index-one-specialization}
at the cusp and Lemma~\ref{lem:rational-finite-transfer} at the two finite
points.  For the second assertion, the cellular Euler formula for a local
system $\mathcal L$ on the complement of three points is
\[
 \chi(B,j_*\mathcal L)
 =-\operatorname{rk}\mathcal L+
   \sum_{i=0}^2\dim\mathcal L^{T_i}.
\]
Global sections are simultaneous invariants and top cohomology is the
coinvariant space.  For the five exterior powers of the present rank-four
local systems, both spaces are one-dimensional and the local-invariant
ranks are $(1,2,4,2,1)$ at the cusp and the corresponding finite-order
ranks; the formula gives Euler characteristic two and hence vanishing of
$H^1$.
\end{proof}

\end{document}